\documentclass[12pt]{article}
\usepackage{makeidx}
\usepackage{amsmath}
\usepackage{amsthm}
\usepackage{amsbsy}
\usepackage{amssymb}
\usepackage{latexsym}
\usepackage[nottoc]{tocbibind}
\usepackage{accents}
\usepackage{enumerate}
\usepackage{tocstyle}

\usepackage{hyperref}
\usepackage[capitalize]{cleveref}
\usetocstyle{standard}

\makeindex

\def\bbC{{\mathbb C}}

\def\bbF{{\mathbb F}}
\def\bbZ{{\mathbb Z}}
\def\bbM{{\mathbb M}}
\def\bbN{{\mathbb N}}

\def\bbR{{\mathbb R}}

\def\bbT{{\mathbb T}}

\def\bfg{{\boldsymbol G}}

\def\b{{\cal B}}

\def\p{{\cal P}}

\newtheorem{theory}{Theorem}[section]
\newtheorem{thm}{Theorem}[section]
\newtheorem{lem}[thm]{Lemma}
\newtheorem{dfn}[thm]{Definition}
\newtheorem{cor}[thm]{Corollary}
\newtheorem{pro}[thm]{Proposition}
\newtheorem{claim}[thm]{Claim}
\newtheorem{prop}[theory]{Proposition}
\newtheorem{define}[theory]{Definition}
\newtheorem{lemma}[theory]{Lemma}

\newtheorem{problem}[theory]{Problem}

\newtheorem{remark}[thm]{Remark}
\newtheorem{prob}[thm]{Problem}

\newcommand{\dee}{\mathrm{d}}

\begin{document}
\title{\bf Weak containment of measure preserving group actions}
\author{Peter J. Burton  and Alexander S. Kechris}
\maketitle

\section*{Introduction}
This paper is a contribution to the study of the global structure of measure preserving actions of countable (discrete) groups on non-atomic standard probability spaces. For such a group $\Gamma$ and space $(X,\mu)$, we let $A(\Gamma, X, \mu)$ be the space of measure preserving actions of $\Gamma$ on $(X, \mu)$. In the book \cite{K} a hierarchical notion of complexity of such actions, called {\bf weak containment}, was introduced, motivated by analogous notions of weak containment of unitary representations. Roughly speaking an action $a\in A(\Gamma, X, \mu)$ is weakly contained  in an action $b\in A(\Gamma, X, \mu)$, in symbols $a\preceq b$,  if the action of $a$ on finitely many group elements and finitely many Borel sets in $X$ can be simulated with arbitrarily small error by the action of $b$. We also denote by $a\simeq b \iff a\preceq b \ \& \ b\preceq a$ the corresponding notion of {\bf weak equivalence}. 

This notion of equivalence is much coarser than the notion of isomorphism (conjugacy) $a\cong b$. It is well understood now that, in general, isomorphism is a very complex notion, a fact which manifests itself, for example, in the lack of any reasonable structure in the space $A(\Gamma, X, \mu)/\!\cong $ of actions modulo isomorphism. On the other hand, weak equivalence is a smooth equivalence relation and the space of weak equivalence classes $A(\Gamma, X, \mu)/\!\simeq$ is quite well behaved. 

Another interesting fact that relates to the study of weak containment is that many important parameters associated with actions, such as the type, cost, combinatorial parameters, etc., turn out to be invariants of weak equivalence and in fact exhibit desirable monotonicity properties with respect to the pre-order of weak containment, a fact which can be useful in certain applications.

There has been quite a lot of activity in this area in the last few years and our goal in this paper is to provide a survey of this work. We include detailed references to the literature, where the reader can find proofs of theorems that are discussed here. We do include a few proofs either of results that have not appeared before or when we thought that a more detailed or sometimes a simplified presentation is desirable.


The paper is organized as follows. \cref{1} reviews concepts of weak containment for unitary representations. In \cref{2}, we define weak containment for actions and provide several equivalent reformulations. In \cref{S3}, we start the study of the pre-order of weak containment and discuss its relationship with concepts such as freeness, ergodicity, strong ergodicity,  and co-induction for group actions. In \cref{limits}, we discuss the connection of weak containment of actions with that of their Koopman representations. \cref{space} continues the study of the pre-order of weak containment, concentrating on the existence and description of a maximum action. In \cref{S6}, we discuss the relationship of weak containment to the classical concept of factoring, which is a strong form of weak containment, including some recently established rigidity phenomena concerning these notions. \cref{6} surveys the invariance and monotonicity properties of various parameters associated with actions with respect to weak equivalence and weak containment. In \cref{7}, we relate weak containment  and weak equivalence to the concept of invariant random subgroup. In \cref{8}, we discuss a variant of weak containment, called {\bf stable weak containment}, due to Tucker-Drob. \cref{9} introduces the compact, metrizable topology on the space of weak equivalence classes, defined by Ab\'{e}rt and Elek, and studies its properties. \cref{10} concerns some relations of weak containment with soficity and entropy. \cref{add1} refers to extensions of the study of weak containment in the case of Polish locally compact groups and also in the context of stationary actions. The Appendices contain proofs of selected results.

\medskip
{\it Acknowledgments.} The authors were partially supported by NSF Grant DMS-1464475.

We would like to thank Anton Bernshteyn, Lewis Bowen, Ronnie Chen, Clinton Conley, Adrian Ioana, Martino Lupini, Robin Tucker-Drob and an anonymous referee for many useful comments and suggestions and Alessandro Carderi for pointing out some errors in an earlier version of this paper.

\newpage

\tableofcontents

\newpage
\section{Weak containment of unitary representations}\label{1}
The concept of weak containment for actions was motivated by the classical concept of weak containment for unitary representations that we quickly review here. See 
\cite[Part II, F]{BdlHV} and \cite[Appendix H]{K} for a more detailed treatment.

Let $\Gamma$ be a countable (discrete group) and $H$ a separable complex Hilbert space. We denote by $U(H)$\index{$U(H)$} the unitary group of $H$ with the strong (equivalently weak) topology, which makes it a Polish group. Let ${\rm Rep}(\Gamma, H)$\index{${\rm Rep}(\Gamma, H)$} be the space of unitary representations of $\Gamma$ on $H$, i.e., homomorphisms of $\Gamma$ into $U(H)$,  with the Polish topology it inherits as a closed subspace of the Polish product space $U(H)^\Gamma$. If $\pi\in {\rm Rep}(\Gamma, H)$, we usually write $H_\pi = H$.\index{$
H_\pi$}
 \begin{dfn}
Let $\pi\in {\rm Rep}(\Gamma, H_\pi)$, $\rho\in {\rm Rep}(\Gamma, H_\rho)$ be two unitary representations. We say that $\pi$ is {\bf weakly contained}\index{weakly contained for representations} in $\rho$, in symbols
\[
\pi\preceq \rho,\index{$\preceq$}
\]
if for any $v\in H_\pi, \epsilon >0, F\subseteq \Gamma$ finite, there are $v_1, \dots , v_k \in H_\rho$ such that $|\langle \pi (\gamma)(v), v\rangle - \sum_{i=1}^k
\langle \rho (\gamma)(v_i), v_i\rangle|<\epsilon, \forall \gamma \in F$.

\end{dfn}
Equivalently this states that every positive-definite function realized in $\pi$ (i.e., a function of the form $\gamma \mapsto \langle\pi(\gamma)(v), v\rangle$, for some $v\in H_\pi$) is the pointwise limit of a sequence of finite sums of positive-definite functions realized in $\rho$.

\begin{remark}
{\rm The symbol $\prec$ is traditionally used for weak containment but $\preceq$ seems more appropriate as it does not give the impression of a strict relation.}
\end{remark}
It is easy to see that $\preceq$ is a {\bf pre-order}\index{pre-order} (i.e., a transitive, reflexive relation). We put 
\[
\pi\simeq \rho \iff \pi\preceq \rho \  \&  \ \rho\preceq \pi\index{$\simeq$},
\]
for the associated relation of {\bf weak equivalence}\index{weak equivalence for representations}.

We also have the following variant of weak containment due to Zimmer.

\begin{dfn}
Let $\pi\in {\rm Rep}(\Gamma, H_\pi)$, $\rho\in {\rm Rep}(\Gamma, H_\rho)$ be two unitary representations. We say that $\pi$ is {\bf weakly contained in the sense of Zimmer}\index{weakly contained in the sense of Zimmer} in $\rho$, in symbols
\[
\pi\preceq_Z \rho\index{$\preceq_Z$},
\]
if for any $v_1, \dots , v_n\in H_\pi, \epsilon >0, F\subseteq \Gamma$ finite, there are $w_1, \dots , w_n \in H_\rho$ such that $|\langle \pi (\gamma)(v_i), v_j\rangle - 
\langle \rho (\gamma)(w_i), w_j\rangle|<\epsilon, \forall \gamma \in F, i, j \leq n$.

\end{dfn}
Let also 
\[
\pi\simeq_Z \rho\index{$\simeq_Z$}\iff \pi\preceq_Z\rho \ \& \ \rho\preceq_Z \pi,
\]
be the associated notion of {\bf weak equivalence in the sense of Zimmer}\index{weak equivalence in the sense of Zimmer}.

We have the following connection between these two notions:

\[
\pi\preceq_Z\rho\implies\pi\preceq \rho \iff \pi\preceq_Z \infty\cdot \rho,
\]
where for $n = 1,2, \dots , \infty$, $n\cdot \pi$ is the direct sum of $n$ copies of $\pi$ (and $\infty$ means $\aleph_0$ here). When $H_\pi$, $H_\rho$ are infinite-dimensional, then it turns out that 
\[
\pi\preceq_Z\rho \iff \pi \in \overline{\{\sigma\in {\rm Rep}(\Gamma, H_\pi)\colon \sigma\cong \rho\}},
\]
where $\cong$ denotes isomorphism (unitary equivalence) between representations (see \cite[Proposition 11.2]{K}).

We write $\pi\leq \rho$\index{$\pi\leq \rho$} if $\pi$ is a {\bf subrepresentation}\index{subrepresentation} of $\rho$, i.e., $\pi$ is isomorphic to the restriction of $\rho$ to an invariant subspace of $H_\rho$. Clearly $\pi\leq \rho\implies \pi\preceq_Z\rho$.

\begin{remark}
{\rm The notions of weak equivalence and weak equivalence in the sense of Zimmer are distinct, even for the group $\bbZ$ and infinite-dimensional representations.
Let for example $\pi$ be the one-dimensional representation of $\bbZ$ given by multiplication by some fixed $\alpha$ in the unit circle and $\rho$
the one-dimensional representation of $\bbZ$ given by multiplication by $-\alpha$. Let $\sigma = \pi \oplus\rho\oplus\rho\oplus\rho\cdots$. Then $\sigma\simeq \sigma\oplus\sigma$ but $\sigma\not\simeq_Z \sigma\oplus \sigma$, since it is easy to see that $\pi\oplus \pi\npreceq_Z\sigma$.

We do not know examples of weak mixing unitary representations for which weak containment and weak containment in the sense of Zimmer differ. Recall that a weak mixing representation is one that has no non-0 finite-dimensional subrepresentations.}
\end{remark}

\newpage
\section{Weak containment of measure preserving group actions}\label{2} 
\subsection{The main definition} Let now $(X,\mu )$ be a standard probability space (i.e., $X$ is a standard Borel space and $\mu$ a probability Borel measure on $X$) and let ${\rm MALG}_\mu$\index{${\rm MALG}_\mu$} be the measure algebra of  $(X,\mu )$. We denote by ${\rm Aut}(X,\mu)$\index{${\rm Aut}(X,\mu)$} the Polish group of all Borel automorphisms of $X$ which preserve the measure $\mu$ (and in which we identify two such automorphisms if they agree $\mu$-a.e.), with the weak topology. See \cite[Section 1]{K} for more details about this group.

 {\bf As it is common practice, we usually neglect null sets in the sequel, unless there is a danger of confusion.} 


For a countable group $\Gamma$, we denote by $A(\Gamma, X, \mu)$ the space of measure preserving actions of $\Gamma$ on $(X,\mu)$, i.e., homomorphisms of $\Gamma$ into ${\rm Aut}(X,\mu)$, with the weak topology, i.e., the Polish topology it inherits as a closed subspace of the product space ${\rm Aut}(X,\mu)^\Gamma$. For $a\in A(\Gamma, X, \mu), \gamma \in \Gamma$, let $\gamma^a = a(\gamma) \in {\rm Aut}(X,\mu)$\index{$\gamma^a$}.

\begin{dfn}[{\cite[Section 10]{K}}]
Let $a\in A(\Gamma, X, \mu)$,  $b\in A(\Gamma, Y, \nu)$ be two actions. We say that $a$ is {\bf weakly contained}\index{weakly contained for actions} in $b$, in symbols
\[
a\preceq b\index{$\preceq$},
\]
 if for any $A_1, \dots , A_n \in {\rm MALG}_\mu, {\rm finite} \ F\subseteq \Gamma, \epsilon >0$, there are $B_1, \dots , B_n \in {\rm MALG}_\nu$ such that $|\mu (\gamma^a(A_i)\cap A_j)-\nu (\gamma^b(B_i)\cap B_j)|<\epsilon, \forall\gamma\in F, i,j\leq n$.
\end{dfn}
Again $\preceq$ is a pre-order and we let also 
\[
a\simeq b\index{$\simeq$}\iff a\preceq b \ \& \ b\preceq a
\]
be the associated notion of {\bf weak equivalence}\index{weak equivalence for actions}.

One can check that in the definition of weak containment we may take the sets $A_1, \dots , A_n$ to belong to any countable dense subalgebra of ${\rm MALG}_\mu$ and form a partition of $X$ and then also require that the sets $B_1, \dots , B_n$ form also a partition of $Y$.

\begin{remark}
{\rm If $a\in A(\Gamma, X, \mu)$,  $b\in A(\Gamma, Y, \nu)$ are weakly equivalent, then the measure spaces $(X,\mu), (Y,\nu)$ are isomorphic (and if $a$ is weakly contained in $b$,  then $(X,\mu)$ is a factor of $(Y,\nu)$); see \cite[Proposition A.4]{T-D1}.}
\end{remark}

The following is also a characterization of weak containment in the case of non-atomic standard probability spaces.

\begin{thm}[{\cite[Proposition 10.1]{K}}]\label{t1}
Let $a\in A(\Gamma, X, \mu)$,  $b\in A(\Gamma, Y, \nu)$ with $(X,\mu), (Y, \nu)$ non-atomic. Then

\[
a\preceq b \iff a \in \overline{\{c\in A(\Gamma, X, \mu)\colon c\cong b\}},
\]
where $\cong$\index{$\cong$} denotes isomorphism\index{isomorphism (conjugacy)} (conjugacy) between actions.
\end{thm}
It is thus clear that weak containment of actions is an analog of weak containment of unitary representations in the sense of Zimmer. We will later discuss in \cref{8} a variant of weak containment of actions, called stable weak containment, that corresponds to weak containment of unitary representations.

The group ${\rm Aut}(X,\mu)$ acts continuously on $A(\Gamma, X, \mu)$ by conjugation, i.e., $T\cdot a = b$ , where $\gamma^b = T\gamma^a T^{-1}$. Thus from \cref{t1}, if $(X,\mu)$ is non-atomic and $a, b \in A(\Gamma, X, \mu)$, then
\[
 a\preceq b\iff a\in \overline{{\rm Aut}(X,\mu)\cdot b}
 \]
 and 
 \[
 a\simeq b\iff  \overline{{\rm Aut}(X,\mu)\cdot a} = \overline{{\rm Aut}(X,\mu)\cdot b}.
 \]
 It follows that the pre-order $\preceq$ is $G_\delta$ in the space $A(\Gamma, X, \mu)^2$, so that in particular every weak equivalence class is a $G_\delta$ subset of $A(\Gamma, X, \mu)$ and thus a Polish space in its relative topology. Moreover its initial segments $\preceq^b =\{a\colon a\preceq b\}$ are closed. Therefore the equivalence relation $\simeq$ is smooth and thus the quotient space $A(\Gamma, X, \mu)/\simeq$ is well-behaved. We will see later in \cref{9} that it actually carries a nice compact, metrizable topology. This should be contrasted with the fact that for infinite $\Gamma$ the isomorphism equivalence relation $\cong$ is very complicated, in particular not smooth (see, e.g., \cite[Theorem 13.7]{K}), so the quotient space $A(\Gamma, X, \mu)/\cong$ is not well-behaved.

A special case of weak containment comes from factoring. Given two actions $a\in A(\Gamma, X, \mu)$,  $b\in A(\Gamma, Y, \nu)$, a {\bf homomorphism}\index{homomorphism} of $a$ to $b$ is a Borel map $f\colon X\to Y$ such that $f_* \mu = \nu$ and $f(\gamma^a (x)) = \gamma^b (f(x))$, $\mu$-a.e. $ \forall \gamma\in \Gamma$. If such a homomorphism exists, we say that $b$ is a {\bf factor}\index{factor} of $a$ or that $a$ is an {\bf extension}\index{extension} of $b$, in symbols 
\[
b\sqsubseteq a\index{$\sqsubseteq$}.
\]
It is then easy to see that 
\[
b\sqsubseteq a \implies b\preceq a
\]
but the converse in general fails, see \cref{S6}.

\begin{remark}
{\rm Although we are primarily interested in standard probability spaces, we note that the definitions of weak containment, homomorphism and factors make perfectly good sense for measure preserving actions  of countable groups on {\it arbitrary} probability spaces and we will occasionally make use of these more general notions. We keep the same notation as before in this more general context.

We note that in \cite[Theorem A]{C} it is shown that every measure preserving action on a non-atomic probability space is weakly equivalent to an action on a standard non-atomic probability space.}

\end{remark}

\subsection{Alternative descriptions}\label{22} We proceed next to see some alternative ways of describing weak containment.

\medskip
(1) Let $\Gamma = \{ \gamma_0, \gamma_1, \dots \}$ be an enumeration of $\Gamma$. Let $a\in A(\Gamma, X, \mu)$ and let $\bar{A} = \{A_0, A_1, \dots , A_{k-1}\}$ be a partition of $X$ into $k>1$ Borel pieces. For each $n >1$, let $M_{n,k}^{\bar{A}} (a) \in [0,1]^{n\times k\times k}$ be the point whose value at $(l,i,j)$, where $l <  n, i,j <k$, is equal to $\mu(\gamma_l^a (A_i) \cap A_j)$. Then let $C_{n,k}(a) $\index{$C_{n,k}(a) $} be the closure of the set $\{ M_{n,k}^{\bar{A}}(a)\colon \bar{A} \ \textrm{is a Borel partition of}  \  X \}$. Then we have 
\[
a\preceq b \iff \forall n, k (C_{n,k}(a) \subseteq C_{n,k}(b))
\]
and 
\[
a\simeq b \iff \forall n, k (C_{n,k}(a) = C_{n,k}(b)). 
\]
This description will be useful in defining later the topology of $A(\Gamma, X, \mu)/\simeq$; see \cref{101}.

\medskip
(2) (\cite{AW}; see also \cite[Section 3]{T-D1}) Let $K$ be a compact, metrizable space. We consider the product space $K^\Gamma$ and the shift action $s = s_{K,\Gamma}$\index{$s = s_{K,\Gamma}$} of $\Gamma$ on $K^\Gamma$. We denote by $M_s(K^\Gamma)$\index{$M_s(K^\Gamma)$} the compact, metrizable, convex set of shift-invariant probability Borel measures on $K^\Gamma$ with the usual weak* topology.

Let now $a\in A(\Gamma, X, \mu)$. For each Borel function $f\colon X\to K$, define $\Phi^{f,a}\colon X\to K^\Gamma$ by  $\Phi^{f,a}(x)(\gamma) = f((\gamma^a)^{-1}(x))$. Then $\Phi^{f,a}$ is $\Gamma$-equivariant and so if $\nu = (\Phi^{f,a})_* \mu$, then $\nu$ is shift-invariant, i.e., $s\in A(\Gamma, K^\Gamma , \nu)$, and $\Phi^{f,a}$ is a homomorphism of $a$ to $s$, so $s\in A(\Gamma, K^\Gamma , \nu)$ is a factor of $a$. Conversely if $s\in A(\Gamma, K^\Gamma , \nu)$, for some shift-invariant probability Borel measure $\nu$ on $K^\Gamma$, is a factor of $a$, via the homomorphism $\Phi$, then for $f(x) = \Phi(x)(e_\Gamma)$ we have
$\Phi^{f,a} = \Phi$ and $\nu = (\Phi^{f,a})_* \mu$ (here $e_\Gamma$ is the identity of the group $\Gamma$).\index{$e_\Gamma$}

Let now 
\[
E(a, K) = \{ (\Phi^{f,a})_* \mu\colon f\colon X\to K, f \ \textrm{Borel}\}\index{$E(a, K)$}
\]
Thus $E(a,K)$ is the set of all shift-invariant probability Borel measures $\nu$ on $K^\Gamma$ such that $s\in A(\Gamma, K^\Gamma , \nu)$ is a factor of $a$.

We now have:
\begin{thm}[{\cite[Lemma 8]{AW}; see also \cite[Proposition 3.6]{T-D1}}]\label{25}
 The following are equivalent for any two actions $a\in A(\Gamma, X, \mu),b\in A(\Gamma, Y, \nu)$:
\begin{enumerate}[\upshape (i)]
\item $a \preceq b$.
\item For each compact, metrizable space $K$, $E(a,K) \subseteq \overline{E(b,K)}$.
\item For $K= 2^\bbN$, $E(a,K) \subseteq \overline{E(b,K)}$.
\item For each finite space $K$, $E(a,K) \subseteq \overline{E(b,K)}$.
\end{enumerate}

\end{thm}


\medskip
(3) (\cite[Section 2]{IT-D}) Let $a\in A(\Gamma, X, \mu)$,  $b\in A(\Gamma, Y, \nu)$. We say that $b$ is an {\bf approximate factor}\index{approximate factor} of $a$, in symbols
\[
b\sqsubseteq_{ap} a\index{$\sqsubseteq_{ap}$}
\]
if there are measure preserving Borel maps $f_n\colon X \to Y$ such that for any Borel $A\subseteq Y$ and $\gamma \in \Gamma$, we have that 
\[
\mu(\gamma^a (f_n^{-1}(A))\Delta f_n^{-1}(\gamma^b(A)))\to 0.
\]
Clearly $b\sqsubseteq a \implies b\sqsubseteq_{ap} a$. We now have, using, e.g.,  \cref{t1}.

\begin{pro}[\cite{IT-D}, Lemma 2.2] Let $a\in A(\Gamma, X, \mu)$,  $b\in A(\Gamma, Y, \nu)$ with $(X, \mu) , (Y, \nu)$ non-atomic. Then 
\[
b\preceq a \iff b\sqsubseteq_{ap} a.
\]
\end{pro}
(4) (\cite[Section 5]{CKT-D}) Our next description involves the concept of ultrapower of actions, see, e.g., (\cite[Section 4]{CKT-D}). For an action $a$ and a non-principal ultrafilter $\mathcal{U}$ on $\bbN$, we denote by $a_\mathcal{U}$\index{$a_\mathcal{U}$} the {\bf ultrapower}\index{ultrapower} of $a$ by $\mathcal{U}$.

\begin{thm}[{\cite[Corollary 5.4]{CKT-D}}]\label{2.6} Let $a\in A(\Gamma, X, \mu)$,  $b\in A(\Gamma, Y, \nu)$ with $(X, \mu) , (Y, \nu)$ non-atomic, and let $\mathcal{U}$
be a non-principal ultrafilter on $\bbN$. Then
\[
b\preceq a \iff b\sqsubseteq a_\mathcal{U}.
\]
\end{thm}
The following consequence of \cref{2.6} gives another connection between weak containment and factoring.
\begin{pro}[{\cite[Proposition 5.7]{CKT-D}, \cite[Corollary 3.1]{AE1}}]\label{28a} Let $a\in A(\Gamma, X, \mu)$,  $b\in A(\Gamma, Y, \nu)$ with $(X, \mu) , (Y, \nu)$ non-atomic. Then
\[
b\preceq a \iff \exists c\in A(\Gamma, X, \mu)( c\simeq a \ \&  \ b\sqsubseteq c).
\]
\end{pro}

\medskip
(5) A final description of weak containment, due to Martino Lupini, uses the concepts of the model theory of metric structures for which we refer to \cite{BYBHU}. Each action $a\in A(\Gamma, X, \mu)$ can be identified with the metric structure $\bbM_a= \langle {\rm MALG}_\mu, d_\mu, \mu, \emptyset,X, \cap,\cup, -, \{\gamma^a\}_{\gamma\in\Gamma}\rangle$,\index{$\bbM_a$} where we let $d_\mu (A,B) = \mu(A\Delta B)$\index{$d_\mu$}, $\mu$ is a unary predicate, $\emptyset, X$ are constants, $ \cap,\cup, -$ are the Boolean operations in  ${\rm MALG}_\mu$, and for each $\gamma\in \Gamma$, $\gamma^a$ is viewed as a unary function on  ${\rm MALG}_\mu$.

An {\bf infimum formula}\index{infimum formula} in the language of this structure is one of the form $\inf_{x_1}\inf_{x_2}\cdots\inf_{x_n}\varphi$, where $\varphi$ is a quantifier-free formula. It is an {\bf infimum sentence\index{infimum sentence}} if in addition it has no free variables. Finally, for each $a\in A(\Gamma, X, \mu)$ and infimum sentence $\varphi$, let $\varphi^a \in [0,1]$ be the interpretation of $\varphi$ in the structure $\bbM_a$. Then we have:

\begin{thm}\label{contlog}
Let  Let $a\in A(\Gamma, X, \mu)$,  $b\in A(\Gamma, Y, \nu)$ with $(X, \mu) , (Y, \nu)$ non-atomic. Then the following are equivalent:
\begin{enumerate}[\upshape (i)]
\item $a \preceq b$.
\item For every infimum sentence $\varphi$, $\varphi^b\leq\varphi^a$.

\end{enumerate}

\end{thm}

\begin{remark}\label{211}
{\rm Let $\Gamma = \{\gamma_0, \gamma_1, \dots \}$. For each $n, k >0, r \in [0,1]^{n\times k\times k}$, consider the infimum sentence
\[
\varphi_{n,k,r} = \inf_{x_0}\inf_{x_1}\cdots\inf_{x_{k-1}}\max_{i,j<k, l< n} |\mu (\gamma_l (x_i)\cap x_j) - r(l,i,j)|.
\]
Then in \cref{contlog} (ii) it is enough to use the sentences $\varphi_{n,k,r}$ instead of arbitrary $\varphi$.
}
\end{remark}

\begin{remark} 
{\rm Another reformulation of weak containment can be also found in \cite[Definition 3.1]{AP}, using the action of $\Gamma$ on $L^\infty (X,\mu)$ associated to any $a\in A(\Gamma, X, \mu)$.}
\end{remark}


\newpage
\section{The weak containment order}\label{S3} 
{\bf For the rest of the paper, unless it is otherwise explicitly stated or is clear from the context, we assume that groups are countably infinite and standard probability spaces are non-atomic. Of course it does not matter which space we use, since all non-atomic standard probability spaces are isomorphic.}

\medskip
We will discuss in this section some basic properties of the pre-order $\preceq$ on the space $A(\Gamma, X, \mu)$.
\subsection{General properties}\label{3.1} We start with the following result.
\begin{thm}[Glasner-Thouvenot-Weiss \cite{GTW}, Hjorth; see {\cite[Theorem 10.7]{K}}] \label{max}There is a maximum element in the pre-order $\preceq$ of $A(\Gamma, X, \mu)$, denoted by $a_{\infty, \Gamma}$\index{$a_{\infty, \Gamma}$}.
\end{thm}
Of course $a_{\infty , \Gamma}$ is unique up to weak equivalence and is characterized by the property that its conjugacy class is dense in $A(\Gamma, X, \mu)$. One way to obtain such an $a_{\infty , \Gamma}$ is to take the product of a dense sequence of actions in $A(\Gamma, X, \mu)$. If $a_n\in A(\Gamma, X_n, \mu_n)$, then the product $a=\prod_n a_n$ is the action on $\prod_n (X_n, \mu_n)$ given by $\gamma^a((x_n)) = (\gamma^{a_n}(x_n))$.

\begin{remark}
{\rm Bowen \cite{Bo2} has shown that for any free group $\Gamma$ and any free action $a\in A(\Gamma, X, \mu)$, the orbit equivalence class of $a$ is dense in $A(\Gamma, X, \mu)$. Thus, in this case, $a_{\infty , \Gamma}$ can be realized as  the product of a sequence of actions orbit equivalent to $a$. Recall that $a,b\in A(\Gamma, X, \mu)$ are {\bf orbit equivalent}\index{orbit equivalent} if there is an automorphism of $(X,\mu)$ that takes the $a$-orbits to the $b$-orbits.
}
\end{remark}

We also note that the preorder $\preceq$ is large:
\begin{thm}[{\cite[Corollary 4.2]{Bu}}]
For any group $\Gamma$, there are continuum many weak equivalence classes.
\end{thm}

\subsection{Freeness}\label{3.2} Recall next that an action $a\in A(\Gamma, X, \mu)$ is {\bf free}\index{free} if $\forall \gamma\not= e_\Gamma (\gamma^a(x) \not= x, \mu \textrm{-a.e.})$ The set $\textrm{FR}(\Gamma, X, \mu)$\index{$\textrm{FR}(\Gamma, X, \mu)$} of free actions is a dense $G_\delta$ subset of $A(\Gamma, X, \mu)$ (Glasner-King \cite{GK}; see also \cite[Theorem 10.8]{K}). We now have the following result:

\begin{thm}\label{fre}
The set  ${\rm FR}(\Gamma, X, \mu)$ of free actions is upwards closed in $\preceq$, i.e., 
\[
a\preceq b, a\in {\rm FR}(\Gamma, X, \mu)\implies b\in {\rm FR}(\Gamma, X, \mu).
\]
In particular, freeness is a weak equivalence invariant.
\end{thm}
This can be easily seen, for example, using \cref{2.6}. It also follows by an application of Rokhlin's Lemma (see, e.g., \cite[Theorem 7.5]{KM}). Indeed assume $a\preceq b$, $a$ is free but $b$ is not free, towards a contradiction. Then there is some $\gamma\not= e_\Gamma$ such that $B= \{x: \gamma^b (x) = x\}$ has positive measure $\epsilon$. Now by Rokhlin's Lemma, when $\gamma$ has infinite order, and trivially when $\gamma$ has finite order, we can find $n\geq 1$ and a finite Borel partition $A_1, \dots , A_n , A_{n+1}$ of $X$ such that $\gamma^a(A_i) \cap A_i =\emptyset$, for $1\leq i\leq n$, and $\mu(A_{n+1}) <\frac{\epsilon}{n+1}$. Then there is a Borel partition $B_1, \dots, B_n, B_{n+1}$ such that $\mu (\gamma^b(B_i) \cap B_i) < \frac{\epsilon}{n+1}$, for $1\leq i\leq n$, and $\mu (B_{n+1}) <\frac{\epsilon}{n+1}$. It follows that for some $1\leq i\leq n$, $\mu (B\cap B_i)\geq \frac{\epsilon}{n+1}$ and then $\mu (\gamma^b(B_i) \cap B_i))\geq \mu (\gamma^b(B\cap B_i) \cap B_i) =  \mu (B\cap B_i)\geq \frac{\epsilon}{n+1}$, a contradiction.

 Also note that \cref{fre} implies that $a_{\infty, \Gamma}$ is free. Below by a {\bf free weak equivalence class}\index{free weak equivalence class} we mean one which contains free actions.

Below let $s_\Gamma = s_{[0,1], \Gamma}$\index{$s_\Gamma = s_{[0,1], \Gamma}$} be the shift \index{shift}(Bernoulli) action of $\Gamma$ on $[0,1]^\Gamma$ with the usual product measure. Then we have the following result of Ab\'{e}rt-Weiss:

\begin{thm}[\cite{AW}]\label{min}
The action $s_\Gamma$ is minimum in the pre-order $\preceq$ on ${\rm FR}(\Gamma, X, \mu)$.\end{thm}

The same result is true for the shift action of $\Gamma$ on any product space $(X^\Gamma, \mu^\Gamma)$, where $(X,\mu)$ is a standard probability space and $\mu$ does not concentrate on a single point. In \cite{H} it is shown, using also \cref{min}, that certain actions of $\Gamma$ by automorphisms on a compact metrizable abelian group (equipped with the Haar measure) are weakly equivalent to $s_\Gamma$.  Moreover in \cite{H1} the author studies further properties of weak containment in the context of actions by automorphisms on compact metrizable groups. In particular it is shown in \cite{H1} that for every  such action there is a unique maximal, invariant closed subgroup for which the action (with the associated Haar measure) is weakly contained in $s_\Gamma$. Moreover this subgroup is also characterized by an appropriate minimality property.

In a recent preprint, Bernshteyn proves a pointwise strengthening of \cref{min} as well as a Borel version for finitely generated groups of subexponential growth. For the precise statement and proofs of these results, see \cite{Be1}

We give the detailed proof of \cref{min} for the shift action on the product space $(2^\Gamma,\mu_0^\Gamma)$ (where $\mu_0$ is the measure on the two point space that gives measure $\frac{1}{2}$ to each point) in Appendix B, \cref{appA}. The case of $s_\Gamma$  can be proved with minor modifications.



Thus among the free actions there is a minimum, $s_\Gamma$, and a maximum $a_{\infty, \Gamma}$, in the sense of weak containment. We will see an appropriate  generalization of this for non-free actions in  \cref{73}, \cref{74}.

The following strengthening of \cref{min} was proved by Tucker-Drob:

\begin{thm}[{\cite[Corollary 1.6]{T-D1}} ]\label{td}
Let $a\in {\rm FR}(\Gamma, X, \mu)$. Then $s_\Gamma\times a\simeq a$.
\end{thm}

Again an appropriate generalization of this result for non-free actions is given in \cref{73}.

 When the group $\Gamma$ is amenable, there is only one free weak equivalence class, i.e., all free actions are weakly equivalent to $s_\Gamma$, and moreover $a\preceq s_\Gamma$ for any action $a\in A(\Gamma, X, \mu)$, see \cite[page 91]{K}. Also if we denote by $i_\Gamma$\index{$i_\Gamma$} the {\bf trivial action}\index{trivial action} in $ A(\Gamma, X, \mu)$, i.e., $\gamma^{i_\Gamma}(x) = x$, then $i_\Gamma$ is the minimum in $\preceq$ on $A(\Gamma, X, \mu)$, i.e.,  $i_\Gamma\preceq a$, for any $a\in A(\Gamma, X, \mu)$ (this follows from {\cref{3.5} below, the fact that no action of an amenable group is strongly ergodic and the ergodic decomposition).

When $\Gamma$ is not amenable, then there are continuum many weakly inequivalent free actions, in fact there is a continuum size $\preceq$-antichain in $\textrm{FR}(\Gamma, X, \mu)$ (see \cite[Section 4, {\bf (C)}]{CK} and \cite[Remark 4.3]{T-D1}). Moreover $i_\Gamma$ is $\preceq$-incomparable with $s_\Gamma$  (see \cref{3.5} below and note that $s_\Gamma$ is strongly ergodic and $a\preceq i_\Gamma\implies a = i_\Gamma$). 

Combining these facts we also have the following characterization of amenability:

\begin{thm}
A group is amenable iff the pre-order of weak containment has a minimum element.

\end{thm}

\begin{prob}\label{prob39}
Is there a continuum size $\preceq$-antichain in $A(\Gamma, X, \mu)$, for any amenable group $\Gamma$?
\end{prob}

Bowen mentions that from the results in the paper \cite{BGK} a positive answer to \cref{prob39} can be obtained for the lamplighter groups.

Tucker-Drob in \cite{T-D} defines a group $\Gamma$ to be {\bf shift-minimal} if  $\forall a \in A(\Gamma, X, \mu) ( a\preceq s_\Gamma \implies a\simeq s_\Gamma)$. Thus the shift-minimal groups are exactly those for which every non-free action in $A(\Gamma, X, \mu)$ is $\preceq$-incomparable with $s_\Gamma$. The structure of shift-minimal groups is studied in detail in \cite{T-D} and later in \cite{BDL}, where it is shown that a group is shift-minimal iff it has no non-trivial normal amenable subgroups.

\subsection{The richness of free weak equivalence classes} Each free equivalence classes is quite rich in the sense that it contains many isomorphism classes. In fact we have the following result:

\begin{thm}[{\cite[Theorem 1.7 and Remark 6.5]{T-D1}}] \label{37}
The isomorphism (i.e., conjugacy) relation in each free weak equivalence class of any group is not classifiable by countable structures. The same holds for weak isomorphism and unitary equivalence. Moreover $E_0$ can be Borel reduced to these equivalence relations.
\end{thm}

For the concept of classification by countable structures, see, e.g., \cite[page 35]{K}. Two actions $a,b$ are {\bf weakly isomorphic}\index{weakly isomorphic}, in symbols $a\cong^w b $\index{$\cong^w$}, if $a\sqsubseteq b \ \& \ b\sqsubseteq a$. Also $E_0$ is the the eventual equality equivalence relation on $2^\bbN$.

Ab\'{e}rt and Elek raised in \cite[Question 6.1]{AE1} the question of whether there are (nontrivial) {\bf weakly rigid actions}\index{weakly rigid actions}, i.e., actions $a$ for which the weak equivalence class of $a$ coincides with its isomorphism class. \cref{37} shows that no free weakly rigid actions exist.

\subsection{Ergodicity and strong ergodicity} Recall that an action $a \in A(\Gamma, X, \mu)$ is {\bf ergodic}\index{ergodic} if it has no non-trivial invariant Borel sets, i.e., there is no invariant Borel set $A\subseteq X$ such that $0<\mu(A)<1$. It is called {\bf strongly ergodic}\index{strongly ergodic} if it has no {\bf non-trivial almost invariant Borel sets}\index{non-trivial almost invariant Borel sets}, i.e., there is no sequence $A_n$ of Borel sets such that $\mu(\gamma^a (A_n) \Delta A_n)\to 0, \forall \gamma\in \Gamma$, but $\mu(A_n)(1-\mu(A_n))\not\to 0$. This condition is equivalent to so-called $\mathbf{E_0}$-{\bf ergodicity}\index{$E_0$-ergodicity}, which asserts that any Borel homomorphism of the equivalence relation $E_a$\index{$E_a$} induced by $a$ into a hyperfinite Borel equivalence relation $E$ trivializes, i.e., maps $\mu$-a.e. to a single $E$-class. This is a result of Jones-Schmidt, see \cite[Theorem A2.2]{HK} for a proof (but note that the
ergodicity assumption in the statement of that theorem is unnecessary). Moreover the shift action $s_\Gamma$ is strongly ergodic for any non-amenable $\Gamma$ (Losert-Rindler \cite{LR}, Jones-Schmidt \cite{JS}). More details about these notions can be also found in \cite[Appendix A]{HK}. Denote by $\textrm{ERG}(\Gamma, X, \mu)$\index{$\textrm{ERG}(\Gamma, X, \mu)$} the set of ergodic and by $\textrm{SERG}(\Gamma, X, \mu)$\index{$\textrm{SERG}(\Gamma, X, \mu)$} the set of strongly ergodic actions in $A(\Gamma, X, \mu)$. We now have:

\begin{thm}[{\cite[Proposition 10.6]{K}}]\label{3.5}
If $a \in A(\Gamma, X, \mu)$ and $i_\Gamma\preceq a$, then $a$ admits non-trivial almost invariant sets. If $a$ is also ergodic, then
\[
i_\Gamma\preceq a \iff a \notin {\rm SERG}(\Gamma, X, \mu).
\]
\end{thm}

We also have the following connection with ergodicity.

\begin{thm}[{\cite[Theorem 5.6]{CKT-D}}]
Let $a \in A(\Gamma, X, \mu)$. Then the following are equivalent:
\begin{enumerate}[\upshape (i)]
\item $a\in  {\rm SERG}(\Gamma, X, \mu)$.
\item $\forall b\preceq a(b\in {\rm ERG}(\Gamma, X, \mu))$.
\item $\forall b\simeq a( b\in {\rm ERG}(\Gamma, X, \mu))$.
\end{enumerate}
\end{thm}

\begin{cor}\label{cor3.6}
The set ${\rm SERG}(\Gamma, X, \mu)$ of strongly ergodic actions is downwards closed under $\preceq$, i.e,.
\[
a\preceq b,  b\in  {\rm SERG}(\Gamma, X, \mu) \implies a\in  {\rm SERG}(\Gamma, X, \mu).
\]
In particular, strong ergodicity is a weak equivalence invariant. 
\end{cor}


If the group $\Gamma$ has property (T), then we have that $\textrm{SERG}(\Gamma, X, \mu) = \textrm{ERG}(\Gamma, X, \mu)$ (Schmidt \cite{S}; see also \cite[Theorem 11.2]{K}), thus \cref{cor3.6} holds in this case for $\textrm{ERG}(\Gamma, X, \mu)$. If on the other hand $\Gamma$ does not have property (T), then $\textrm{SERG}(\Gamma, X, \mu) \subsetneqq \textrm{ERG}(\Gamma, X, \mu)$ (Connes-Weiss \cite{CW}; see also \cite[Theorem 11.2]{K}), so ergodicity is not a weak equivalence invariant.

Another characterization of strong ergodicity, for ergodic actions, is the following:
\begin{thm}[{\cite[Theorem 3] {AW}}] \label{38}

Let $a\in A(\Gamma, X, \mu)$ be ergodic. Then the following are equivalent:
\begin{enumerate}[\upshape (i)]
\item $a\notin {\rm SERG}(\Gamma, X, \mu)$.
\item $a\simeq \frac{1}{2}a + \frac{1}{2} a$.
\item $a\simeq \lambda a + (1-\lambda) a$, for some (resp., all) $0<\lambda<1$.
\item $a\simeq i_\Gamma\times a$.
\end{enumerate}
\end{thm}

For the definition of {\bf convex combination}\index{convex combination} $\sum_{i=1}^n\lambda_i a_i$\index{$\sum_{i=1}^n\lambda_i a_i$} of $a_i\in A(\Gamma, X, \mu)$, where $0\leq\lambda_i\leq 1, \sum_{i=1}^n\lambda_{i} =1$, see \cite[Section 10, {\bf (F)}]{K}.

A corollary of this result is another characterization of amenability. Below we let $a^2= a\times a$.

\begin{cor}
Let $\Gamma$ be an infinite group. Then the following are equivalent:
\begin{enumerate}[\upshape (i)]
\item $\Gamma$ is amenable.
\item For any $a \in {\rm FR}(\Gamma, X, \mu)$, $a^2\simeq a$.

\end{enumerate}

\end{cor}

 To see this, notice that if $\Gamma$ is amenable, then $a\simeq s_\Gamma$ and so $a^2 \simeq (s_\Gamma)^2 \cong s_\Gamma \simeq a$. On the other hand, if $\Gamma$ is not amenable, let $b=s_\Gamma$, which is strongly ergodic, and put $a = \frac{1}{2}b + \frac{1}{2} b$. Then $a^2 \simeq \frac{1}{4}b + \frac{1}{4}b + \frac{1}{4}b + \frac{1}{4}b\not\simeq a$.

 As it was shown in \cite{T-D1} weak containment behaves well with respect to the ergodic decomposition.
 
 \begin{thm}[{\cite[Theorems 3.12, 3.13]  {T-D1}}] \label{3.8}
  Let $a,b \in A(\Gamma, X, \mu)$. If $a$ is ergodic, then $a$ is weakly contained in $b$ iff $a$ is weakly contained in almost every ergodic component of $b$.
 \end{thm}
 It is also shown in \cite[Theorem 4.1]{BT-D1} that if $b$ is ergodic and $a\preceq b$, then almost every ergodic component of $a$ is weakly contained in $b$.
 
We next have the following:

\begin{thm}[{\cite[Theorem 13.1]    {K}}]\label{maxerg}
There is a maximum element in the pre-order $\preceq$ on the set ${\rm ERG}(\Gamma, X, \mu)$, denoted by $a^{erg}_{\infty, \Gamma}$.\index{$a^{erg}_{\infty, \Gamma}$}This action is free.
\end{thm}

\begin{prob}\label{prob3.10} Is there is a maximum element in the pre-order $\preceq$ on the set ${\rm SERG}(\Gamma, X, \mu)$?
\end{prob}
If the group $\Gamma$ has property (T), then we have that $\textrm{SERG}(\Gamma, X, \mu) = \textrm{ERG}(\Gamma, X, \mu)$, so \cref{prob3.10} has a positive answer.

Tucker-Drob (private communication) showed the following: The answer to \cref{prob3.10} is positive for any group of the form $\Gamma = H\times K$, where $H$ is an (infinite) simple group with property (T) and $K$ is (infinite) amenable with no non-trivial finite dimensional unitary representations. On the other hand, if $\Gamma$ is non-amenable and satisfies property EMD (see \cref{space} below), e.g., if $\Gamma$ is a free group, then the answer to problem \cref{prob3.10} is negative. 


 Finally Tucker-Drob shows that if \cref{prob3.10} has a positive answer for a group $\Gamma$, then there is a finitely generated subgroup $H$ of $\Gamma$ such that the action of $\Gamma$ on $\Gamma/H$ is amenable.


We have mentioned earlier that for any non-amenable group $\Gamma$, the shift action $s_\Gamma$ is strongly ergodic. The answer to the following problem seems to be unknown:

\begin{prob}
Is there a non-amenable group $\Gamma$ such that every strongly ergodic action of $\Gamma$ is weakly equivalent to $s_\Gamma$?
\end{prob}

An {\bf ergodic weak equivalence class}\index{ergodic weak equivalence class} is one which contains at least one ergodic action. Similarly a {\bf strongly ergodic weak equivalence class}\index{strongly ergodic weak equivalence class} is one which consists of strongly ergodic actions.

\begin{thm}[{\cite[Corollary 4.2]    {T-D1}}] \label{3.15}
A group is amenable iff every free weak equivalence class is ergodic.
\end{thm}

In fact \cite[Theorem 1.3]{T-D1} shows that if $a\in \textrm{SERG}(\Gamma, X, \mu)$, then the weak equivalence class of $i_\Gamma \times a$ is not ergodic.

We also have the following characterization:
\begin{thm}
A group $\Gamma$ has property {\rm (T)} iff $a_{\infty, \Gamma}$ is not ergodic.
\end{thm}

One direction follows from the fact that for property (T) groups we have $\textrm{SERG}(\Gamma, X, \mu) = \textrm{ERG}(\Gamma, X, \mu)$ and $a_{\infty, \Gamma}$ cannot be strongly ergodic. For the other direction, note that if $\Gamma$ does not have property (T), then the weak mixing actions are dense in $A(\Gamma, X, \mu)$ (Kerr-Pichot \cite{KP}; see also \cite[Theorem 12.9]{K}), so $a_{\infty, \Gamma}$ can be realized as a product of a countable sequence of weak mixing actions, which is therefore weakly mixing, thus ergodic.

As we mentioned earlier, if $\Gamma$ is amenable, it has exactly one free weak equivalence class and, by the third paragraph after \cref{td}, if $\Gamma$ is not amenable it has a continuum size $\preceq$-antichain of free actions. The following is an important open problem.

\begin{prob}
If $\Gamma$ is not amenable, does it have continuum many free, ergodic weak equivalence classes? Does it have a continuum size $\preceq$-antichain of free, ergodic actions?
\end{prob}

The following partial results are known concerning these questions: Ab\'{e}rt and Elek have shown in \cite{AE} that there are continuum size $\preceq$-antichains of free, ergodic actions for any finitely generated free group and any linear group with property (T). Bowen and Tucker-Drob \cite{BT-D1} have shown that there are continuum many free, strongly ergodic weak equivalence classes for any group containing a non-abelian free subgroup. It is unknown whether {\it every} non-amenable group has at least three distinct free, ergodic weak equivalence classes.


Finally an analog of \cref{37} holds for free, ergodic weak equivalence classes.

\begin{thm}[\cite{T-D1}, Remark 6.5]
The isomorphism (i.e., conjugacy) relation of the ergodic actions in each free, ergodic weak equivalence class of a group is not classifiable by countable structures. The same holds for weak isomorphism and unitary equivalence. Moreover $E_0$ can be Borel reduced to these equivalence relations.
\end{thm}

\subsection{Co-induction} We note that weak containment respects co-induction (see \cite[pages 72-73]{K} for the concept of co-induction).

\begin{thm}[{\cite[Proposition A.1]   {K1}}]\label{321}
Let $\Gamma\leq \Delta$ and $a, b\in A(\Gamma, X, \mu)$. Then
\[
a\preceq b \implies {\rm CIND}_\Gamma^\Delta (a) \preceq {\rm CIND}_\Gamma^\Delta (b)\index{${\rm CIND}_\Gamma^\Delta (a)$} 
\]
and therefore
\[
a\simeq b \implies {\rm CIND}_\Gamma^\Delta (a) \simeq {\rm CIND}_\Gamma^\Delta (b).
\]
\end{thm}
Concerning the conduction construction we also have the following open problem:

\begin{prob}[{\cite[Problem A.4]{K1}}]\label{3.22}
Let $\Gamma\leq \Delta$ and assume that the action of $\Delta$ on $\Delta/\Gamma$ is amenable. Is it true that for any $a\in A(\Delta, X, \mu)$,
\[
a\preceq {\rm CIND}_\Gamma^\Delta (a|\Gamma)?
\]
\end{prob}
As explained in the paragraph following \cite[Problem A.4]{K1}, the assumption about the amenability of the action of $\Delta$ on $\Delta/\Gamma$ is necessary for a positive answer (for arbitrary $a$). Positive answers to this problem have been obtained for certain groups and actions in \cite{BT-D}. For example, it holds when $\Gamma$ is normal of infinite index in $\Delta$ and $a$ is an ergodic but not strongly ergodic action or if $a\simeq a_{\infty, \Delta}$. Using this the authors show the following result:

\begin{thm}[{\cite[Theorem 1.2]{BT-D}}]
If $\Gamma$ is normal of infinite index in $\Delta$, with $\Delta/\Gamma$ amenable, and $a \simeq a_{\infty, \Gamma}$, then ${\rm CIND}_\Gamma^\Delta (a) \simeq a_{\infty, \Delta}$.
\end{thm}

Moreover \cite[Theorem 1.3]{BT-D} shows that \cref{3.22} has a positive answer for Gaussian actions $a$ of $\Delta$.

\subsection{Restriction} We next note that for each $\Delta\leq \Gamma$, the operation of restriction
\[
a\in A(\Gamma, X, \mu)\mapsto a|\Delta\in A(\Delta, X, \mu)
\]
respects weak containment.

\begin{pro}\label{restr}
Let $\Delta\leq \Gamma$ and $a,b \in A(\Gamma, X, \mu)$. Then 
\[
a\preceq b\implies a|\Delta\preceq b|\Delta,
\]
and therefore
\[
a\simeq b\implies a|\Delta\simeq b|\Delta.
\]
\end{pro}

\subsection{Products and ultraproducts} The following facts are easy to verify.

\begin{pro}\label{ultr}
Let $a_n, b_n\in A(\Gamma, X, \mu), n\in \bbN$. Then 
\[
\forall n(a_n\preceq b_n)\implies \prod_n a_n\preceq\prod_n b_n
\]
and therefore
\[
\forall n(a_n\simeq b_n)\implies \prod_n a_n\simeq\prod_n b_n.
\]
Moreover, if $\mathcal{U}$ is a non-principal ultrafilter on $\bbN$, then
\[
\forall n(a_n\preceq b_n)\implies \prod_n a_n/\mathcal{U}\preceq\prod _nb_n/\mathcal{U}
\]
and therefore
\[
\forall n(a_n\simeq b_n)\implies \prod_n a_n/\mathcal{U}\simeq\prod_n b_n/\mathcal{U}.
\]
\end{pro}
Here $\prod_n a_n/\mathcal{U}$ is the ultrapoduct of the actions $a_n$ by $\mathcal{U}$, see, e.g., \cite[Section 4]
{CKT-D}.

\subsection{Hyperfiniteness and treeability} We have the following result concerning the relation of weak containment and the hyperfiniteness of the equivalence relation induced by an action.

\begin{thm}[{Tucker-Drob; see \cite[Corollary 16.11]{K3}}]\label{329} Let $a,b$ be actions in  $A(\Gamma, X, \mu)$. Then
\[
a\preceq b \  \& \ E_b \ \textrm{is hyperfinite} \ \implies E_a \   \textrm{is hyperfinite}.
\]
In particular, hyperfiniteness of $E_a$ is a weak equivalence invariant.
\end{thm}
The proof given in \cite{K3} is based on the following generalization of \cref{28a}: 

Let $a, b \in A(\Gamma, X, \mu), c\in A(\Delta, X, \mu)$ and assume that $a\preceq b$ and $E_b\subseteq E_c$. Then there are $d \in A(\Gamma, X, \mu), e\in A(\Delta, X, \mu)$ such that $b\simeq d, c\simeq e, a\sqsubseteq d$ and $E_d\subseteq E_e$.

We call an action $a\in A(\Gamma, X, \mu)$ {\bf hyperfinite}\index{hyperfinite action} if $E_a$ is hyperfinite. Similarly a {\bf hyperfinite weak equivalence class}\index{hyperfinite weak equivalence class} is one which consists of hyperfinite actions.

It was observed by Todor Tsankov (see \cite[last remark in page 78]{K}) that if $a,b\in {\rm FR}(\Gamma, X, \mu)$, $a\sqsubseteq b$ and $E_a$ is treeable, then so is $E_b$. The following question seems to be open:

\begin{prob}\label{tr}
 Let $a,b\in {\rm FR}(\Gamma, X, \mu), a\preceq b$ and assume that $E_a$ is treeable. Is $E_b$ treeable?
\end{prob}
Recall that a group $\Gamma$ is called {\bf treeable} (resp., {\bf strongly treeable})\index{treeable group}\index{strongly treeable group} if for {\it some} $a \in {\rm FR}(\Gamma, X, \mu)$, $E_a$ is treeable (resp., for {\it all} $a \in {\rm FR}(\Gamma, X, \mu)$, $E_a$ is treeable). It is unknown if these two notions are equivalent. We note that an affirmative answer to \cref{tr} implies that strong treeability is equivalent to $E_{s_\Gamma}$ being treeable.

\newpage
\section{Connection with the Koopman representation}\label{limits}

For each $a \in A(\Gamma, X, \mu)$, we let $\kappa^a\in \textrm{Rep}(\Gamma, L^2(X, \mu))$\index{$\kappa^a$} be the corresponding {\bf Koopman representation}\index{Koopman representation}, given by  $\kappa^a(\gamma)(f)(x) = f((\gamma^{-1})^a(x))$,  and $\kappa_0^a\in \textrm{Rep}(\Gamma, L_0^2(X, \mu))$\index{$\kappa_0^a$} its restriction to the orthogonal $L_0^2(X, \mu)$ of the constant functions. We have:

\begin{prop}[{\cite[page 67]   {K}}] \label{4.1}
\[
a\preceq b \implies \kappa_0^a\preceq_Z \kappa_0^b \ (\implies \kappa^a\preceq_Z \kappa^b
)\]
\end{prop}
It was known for a while that the converse in \cref{4.1} does not hold but the counterexamples failed to be (both) ergodic (see \cite[page 68]{K} and \cite[page 155]{CK}). However the following was recently shown, where for each cardinal $n =1,2, ... , \infty (= \aleph_0)$, we let $\bbF_n$\index{$\bbF_n$} be the free group with $n$ generators.

\begin{theory}[\cite{BuK}]\label{4.2}
Let $\Gamma = \bbF_\infty$. Then there are free, ergodic actions $a, b \in A(\Gamma, X, \mu)$ such that $ \kappa_0^a\preceq_Z \kappa_0^b$ but $a\npreceq b$.
\end{theory}
It is easy to see that there is a (unique up to $\simeq_Z$) maximum under $\preceq_Z$ unitary representation in $\textrm{Rep}(\Gamma, H)$, where $H$ is infinite-dimensional (see \cite[Proposition H.1]{K}). We denote it by $\pi_{\infty, \Gamma}$\index{$\pi_{\infty, \Gamma}$}. Note that $\pi_{\infty, \Gamma}\simeq_Z \kappa_0^{a_\infty, \Gamma}$, since for any $\pi\in{\rm Rep}(\Gamma, H)$, there is $a\in A(\Gamma, X, \mu)$ with $\pi\leq \kappa_0^a$ (see \cite[Theorem E.1]{K}). Thus \cref{4.2} is a consequence of the following stronger result.

\begin{theory}[\cite{BuK}]\label{4.3}
Let $\Gamma = \bbF_\infty$. Then there is a free, ergodic action $a\in A(\Gamma, X, \mu)$ such that $a\not\simeq a_{\infty, \Gamma}$ but $\kappa_0^a\simeq_Z \pi_{\infty, \Gamma}$.
\end{theory}

Below $\prec$\index{$\prec$} is the strict part of the order $\preceq$ and $a \bot b$\index{$\bot$} means that $a,b$ are $\preceq$-incomparable.

Lewis Bowen more recently considered the following ``dual" question concerning the minimum free action, where by $\lambda_\Gamma$\index{$\lambda_\Gamma$} we denote the {\bf left-regular representation}\index{left-regular representation} of a group $\Gamma$: 
Is there a group $\Gamma$ and a free, ergodic action $a\in A(\Gamma, X, \mu)$, such that $s_\Gamma\prec a$ but $\kappa_0^a\preceq \lambda_\Gamma\simeq \kappa_0^{s_\Gamma}$(see \cite[Appendix D, {\bf (E)}]{K} and \cite[Exercise E.4.5]{BdlHV}). He recently communicated to the authors that a positive answer for a free group can be derived from the papers Bordenave-Collins \cite{BC} and Gamarnik-Sudan \cite{GSu}.


A different example of the distinction between weak containment of ergodic actions and their Koopman representations is the following:
\begin{theory}\label{44}
Let $\Gamma=\bbF_2$. Then there is an ergodic action $a\in A(\Gamma, X, \mu )$ such that $\lambda_\Gamma\simeq \kappa_0^{s_\Gamma}\prec \kappa_0^a$ but $s_\Gamma \bot a$.
\end{theory}
Note that in \cref{44} the action $a$ is not free. The proof of \cref{44} and related issues are discussed in Appendix F, \cref{appF1}.

\newpage
\section{The maximum weak equivalence class}\label{space}
We have seen in \cref{min} that there is a minimum, in the sense of weak containment, free action and this can be concretely realized as the shift action of the group. We have also seen in  \cref{max} that there is a maximum, in the sense of weak containment, action but the proof of that result does not provided a concrete realization of this action. For various reasons, for example in connection with the theory of cost that we will discuss later in \cref{71}, it is important to be able to compute a concrete realization of this action. 

\begin{prob}
Find an explicit realization of the maximum, in the sense of weak containment, action $a_{\infty, \Gamma}$.
\end{prob}
We will discuss here the solution of this problem for certain classes of groups, including the free ones.

Let $\Gamma$ be a residually finite group. We consider the {\bf profinite completion}\index{profinite completion} $\hat{\Gamma}$\index{$\hat{\Gamma}$} of $\Gamma$, equipped with the Haar measure, on which $\Gamma$ acts by left-translation, so that it preserves this measure. We denote this action by $p_\Gamma$\index{$p_\Gamma$}. It is a free, ergodic, profinite action and it is the maximum in the sense of weak containment $\preceq$ (and even in the sense of $\sqsubseteq$) among ergodic, profinite actions (see \cite[Proposition 2.3]{K1}. Recall that an action $a\in A(\Gamma, X, \mu)$ is {\bf profinite}\index{profinite action} if there is a decreasing sequence of finite Borel partitions $\{X\} = \p_0\geq \p_1\geq \dots$ such that each $\p_n$ is $\Gamma$-invariant and $\{\p_n\}$ separates points. For more details about profinite actions, see, e.g., \cite[Section 2]{K1}. We now have the following result:

\begin{thm}[{\cite[Theorem 3.1] {K1}}] \label{5.2}
Let $\Gamma= \bbF_n, 1\leq n \leq\infty$. Then $ p_\Gamma\simeq a_{\infty, \Gamma}$.
\end{thm}
Another realization of $a_{\infty, \Gamma}$ for the free groups is in terms of generalized shifts (compare this with the realization of the minimum, in terms of weak containment, free action as a shift). Given a group $\Gamma$ and a subgroup $H\leq \Gamma$, consider the shift action of $\Gamma$ on the product space $[0,1]^{\Gamma/H}$ (with the product measure), where $\Gamma$ acts on $\Gamma/H$ in the usual way. This is called the {\bf generalized shift}\index{generalized shift} corresponding to $H$ and is denoted by $s_{H, \Gamma}$\index{$s_{H, \Gamma}$} (see \cite[Section 5]{K1} and \cite[Section 2]{KT} for more details about these actions).

\begin{thm}[{\cite[Theorem 5.1]  {K1}}] 
Let $\Gamma= \bbF_n, 1\leq n \leq\infty$. Then there is $H\leq \Gamma$ of infinite index in $\Gamma$ such that $s_{H,\Gamma} \simeq a_{\infty, \Gamma}$.
\end{thm}
\begin{remark}
{\rm This result is stated in \cite[Theorem 5.1]{K1} for the shift on $2^{\Gamma/H}$ but the proof can be easily modified to work for the shift on 
$[0,1]^{\Gamma/H}$ as well.}
\end{remark}
\begin{remark}
{\rm In \cite[Theorem 4]{EK} the authors show that for the group $\Gamma= \bbZ_2\star\bbZ_2\star\bbZ_2\star\bbZ_2\star\bbZ_2$, there is a continuum size $\preceq$-antichain consisting of generalized shifts.

}
\end{remark}

{\bf In the rest of this section, we will assume, unless otherwise explicitly stated, that all groups are residually finite}.

It turns out that \cref{5.2} (or a small variation) can be extended to a wider class of residually finite groups.  To describe this extension, we first need a few definitions. An action $a\in A(\Gamma, X, \mu)$ is {\bf finite}\index{finite action} if it factors through an action of a finite group, i.e., there is a finite group $\Delta$, a surjective homomorphism $f\colon \Gamma \to \Delta$ and an action $b\in A(\Delta, X, \mu)$ such that for all $\gamma\in \Gamma, \gamma^a= f(\gamma)^b$. Equivalently this means that $\{\gamma^a\colon \gamma\in \Gamma\}$ is finite.

A group $\Gamma$ has property {\bf MD}\index{MD} if the finite actions are dense in $A(\Gamma, X, \mu)$. Equivalently this means that the profinite actions are dense in $A(\Gamma, X, \mu)$ (see \cite[Proposition 4.8]{K1}).  A group $\Gamma$ has property {\bf EMD}\index{EMD} if the ergodic, profinite actions are dense in $A(\Gamma, X, \mu)$. 

These notions where introduced in \cite{K1}, where the reader can find much more information about them. Bowen had also introduced earlier a property of groups called {\bf PA\index{PA}} that turns out to be equivalent to MD. A variant of EMD, called {\bf EMD*}\index{EMD*}, was also defined in \cite{K1}, which asserts that the ergodic, profinite actions are dense in $\textrm{ERG}(\Gamma, X, \mu)$. However Tucker-Drob \cite[Theorem 1.4]{T-D1} has shown that it is equivalent to MD; this is a consequence of \cref{3.8}. We have ${\rm EMD}\implies {\rm MD}$ but the problem of whether they are equivalent is open. Tucker-Drob \cite[Corollary 4.7, Theorem 4.10]{T-D1} has shown that they are equivalent for all groups without property (T) and they are equivalent for all groups iff $({\rm MD}\implies \neg {\rm (T)})$. 

We now have:

\begin{thm}[{\cite[Propositions 4.2, 4.5, 4.8]    {K1}}]
Let $\Gamma$ be a residually finite group. Then
\begin{enumerate}[\upshape (i)]
\item $\Gamma$ has property {\rm EMD} $\iff p_\Gamma\simeq a_{\infty, \Gamma}$.
\item $\Gamma$ has property {\rm MD} $\iff i_\Gamma\times p_\Gamma\simeq a_{\infty, \Gamma} \iff  p_\Gamma\simeq a^{erg}_{\infty, \Gamma}$.
\end{enumerate}
\end{thm}
Concerning the extent of the classes MD and EMD, we have the following results and open problems:

(1) All amenable and free groups have property EMD, see \cite[page 486]{K1}; also see \cite{Bo1} for the property PA of free groups.

(2) The free product $\Gamma\star \Delta$ has EMD, if $\Gamma, \Delta$ are nontrivial and each is either finite or has property MD ({\cite[Theorem 4.8]{T-D1}).

(3)  A subgroup of a group with property MD also has property MD (see \cite[page 486]{K1}). This is unknown for EMD and in fact \cite[Theorem 4.10]{T-D1} shows that this statement for EMD implies the equivalence of MD and EMD. 

(4) A finite index extension of a group with MD also has MD (see \cite[page 486]{K1}).

(5) Let $N\vartriangleleft\Gamma$ and suppose the following are satisfied: (1) $N$ is finitely generated and satisfies MD, 
(2) $\Gamma/N$ is a residually finite amenable group. Then $\Gamma$ satisfies MD \cite[Theorem 1.4]{BT-D}. From this it follows that the groups of the form $H\ltimes\bbF_n$, for $H$ an amenable group, the surface groups, and the fundamental groups of virtually fibered closed hyperbolic 3-manifolds (such as ${\rm SL}_2(\bbZ[i]))$ have property MD (see \cite[page 487]{K1} and \cite[page 212]{BT-D}).

(6) A representation theoretic analog of the property MD, called property {\bf FD}\index{FD} was introduced earlier in \cite{LS}. Since ${\rm MD}\implies {\rm FD}$ (see \cite[page 486]{K1}), any group that fails FD also fails MD. Examples of such groups are given in \cite[Section 9.1]{LZ}.

(7) It is not known if the product of two groups with property MD also has property MD. In fact it is not even known if $\bbF_2\times\bbF_2$ has property MD. A positive answer would imply that this group also has property FD, which in turn implies a positive answer to the Connes Embedding Problem\index{Connes Embedding Problem},
see \cite{PU}.

\newpage
\section{Weak containment versus factoring}\label{S6}
The most straightforward way for an action $a$ to be weakly contained in an action $b$ is for $a$ to be a factor of $b$:
\[
a\sqsubseteq b \implies a\preceq b.
\]
In general, weak containment does not imply factoring. For example, if $\Gamma$ is amenable, $i_\Gamma \preceq s_\Gamma$ but $i_\Gamma\sqsubseteq s_\Gamma$ fails. For another example, let $\Gamma = \bbF_2$, let $a$ be not mixing and $b$ be mixing with $a\preceq b$ (for example we can take $b$ to be a mixing representative of the maximum weak equivalence class, which exists since the mixing actions are dense in $A(\Gamma, X, \mu)$). Then clearly $a\sqsubseteq b$ fails.

However in certain situations weak containment implies factoring. Ab\'{e}rt and Elek \cite[Theorem 1]{AE} showed that if $a$ is an action on a finite space and $b$ is strongly ergodic, then $a\preceq b \iff a\sqsubseteq b$ and from this they deduced that if $a,b$ are profinite and $b$ is strongly ergodic, then $a\simeq b \iff a\cong b$. Since all profinite actions are compact, this has been substantially extended by the following result of Ioana and Tucker-Drob. For the definition of compact,\index{compact action} measure distal actions\index{measure distal action} and the notion of maximal distal factor\index{maximal distal factor} of an action, see \cite[Section 1]{IT-D}.

\begin{thm}[{\cite[Theorem 1.1, Corollary 1.3]{IT-D}}]\label{61}
If $a$ is measure distal (in particular if $a$ is compact) and $b$ is strongly ergodic, then $a\preceq b \iff a\sqsubseteq b$. Moreover if $a,b$ are compact, then $a\simeq b \iff a\cong b$.
\end{thm}
More generally it is shown in \cite[Corollary 1.2]{IT-D} that for any $a$ and strongly ergodic $b$, $a\preceq b$ implies that the maximal distal factor of $a$ is a factor of $b$.

In the recent preprint \cite{IT} the authors use methods of continuous model theory to give another proof of \cref{61} and moreover strengthen the last part by dropping the assumption of compactness.

\begin{thm}[{\cite[Theorem 1,1, (ii)]{IT}}]
If $a,b$ are measure distal and strongly ergodic, then $a\simeq b \iff a\cong b$.
\end{thm}

\newpage
\section{Numerical invariants of weak equivalence}\label{6} 
We will discuss here the behavior of various numerical parameters associated to group actions in the context of weak containment and equivalence.

\subsection{Cost}\index{cost}\label{71} For the concept of cost of a countable, measure preserving equivalence relation $E$ on $(X,\mu)$ we refer to \cite{G} and \cite{KM}. It is denoted by $C_\mu(E)$. For a group $\Gamma$ and $a\in A(\Gamma, X, \mu)$, we let $C_\mu(a) = C_\mu(E_a)$\index{$C_\mu(a)$}, where $E_a$ is the equivalence relation induced by the action $a$. Then $1\leq C_\mu(a)\leq \infty$, for $a \in\textrm{FR}(\Gamma, X, \mu)$. The cost of a group $\Gamma$, $C(\Gamma)$\index{$C(\Gamma)$}, is defined as the infimum of $C_\mu(a)$, as $a$ varies over $\textrm{FR}(\Gamma, X, \mu)$. It has been shown in \cite[Theorem 10.13]{K} that when the group $\Gamma$ is finitely generated, the function $a\in \textrm{FR}(\Gamma, X, \mu)\mapsto C_\mu(a)\in\bbR$ is upper semicontinuous, from which the next result immediately follows:

\begin{thm}[{\cite[Corollary 10.14]  {K}}] \label{6.1}
Let $\Gamma$ be finitely generated. Then for $a,b\in {\rm FR}(\Gamma, X, \mu)$,
\[
a\preceq b \implies C_\mu(a)\geq C_\mu(b)
\]
and therefore
\[
a\simeq b \implies C_\mu(a) = C_\mu(b).
\]
\end{thm}
Thus cost is a weak equivalence invariant for free actions and 
\[
C_\mu(s_\Gamma)\geq C_\mu(a)\geq C_\mu(a_{\infty, \Gamma}) = C(\Gamma),
\]
for finitely generated $\Gamma$ and any $a \in\textrm{FR}(\Gamma, X, \mu)$. The group $\Gamma$ has {\bf fixed price} if the cost function is constant on $\textrm{FR}(\Gamma, X, \mu)$. It is an important open problem whether {\it every} group has fixed price.

So we saw that the cost of a finitely generated  group is equal to the cost of the maximum action $a_{\infty, \Gamma}$. This is one motivation for calculating explicit realizations of $a_{\infty, \Gamma}$. This method has been used in \cite{K1} to give a new proof of Gaboriau's result in \cite{G} that the free group $\bbF_n, n = 1, 2, \dots ,$ has cost $n$ and fixed price. Indeed, by \cref{5.2}, the profinite action $p_{\bbF_n}$ of $\bbF_n$ is weakly equivalent to $a_{\infty, \bbF_n}$, thus its cost is equal to $C(\bbF_n)$. But Ab\'{e}rt-Nikolov \cite{AN} had already computed that $C_\mu(p_{\bbF_n})= n$ (as part of a more general result concerning profinite actions).

\begin{remark}
{\rm A modified notion of cost for an action $a\in A(\Gamma, X, \mu)$ has been introduced in \cite{AW}. It dominates the standard notion of cost discussed earlier but agrees with it for free actions. In \cite[Theorem 9]{AW} it is shown that this modified notion of cost satisfies the monotonicity \cref{6.1} for arbitrary, {\it not necessarily free}, actions.}
\end{remark}

Tucker-Drob \cite[Definition 6.6]{T-D} defines another variation of cost, called {\bf pseudocost}\index{pseudocost}, for equivalence relations and therefore actions, using exhaustions of equivalence relations by increasing sequences of subequivalence relations. It is always dominated by the cost but he shows in \cite[Corollary 6.8]{T-D} that it agrees with the cost, when the latter is finite or the equivalence relation is treeable. He also shows in \cite[Corollary 6.20]{T-D} that pseudocost satisfies the monotonicity \cref{6.1} for arbitrary, {\it not necessarily finitely generated}, groups. He then deduces the following generalization of \cref{6.1}.

\begin{thm}[{\cite[Corollary 6.22]{T-D}}] For any group $\Gamma$ and actions $a,b \in {\rm FR}(\Gamma, X, \mu)$,
\[
a\preceq b \ \&  \ C_\mu(b)<\infty\implies C_\mu(a)\geq C_\mu (b)
\]
and therefore
\[
a\simeq b \ \&  \ C_\mu(a)<\infty \ \&  \ C_\mu(b)<\infty \implies C_\mu(a) = C_\mu (b).
\]
\end{thm}
Thus cost is a weak equivalence invariant for free actions of finite cost.

It is also shown in \cite[Corollary 6.22]{T-D} that for an arbitrary group and any $a,b \in {\rm FR}(\Gamma, X, \mu)$, $a\preceq b \implies C_\mu(a)\geq C_\mu(b)$, if $E_b$ is treeable, and also that $a\preceq b$ and $C_\mu(a)=1$ implies $C_\mu(b)=1$. Finally for any $\Gamma$,  the condition $C_\mu(s_\Gamma)=1$ is equivalent to $\Gamma$ having cost 1 and fixed price; see \cite[Corollary 6.24]{T-D}.}

\begin{prob}
Does \cref{6.1} hold for arbitrary groups?
\end{prob} 

\subsection{Combinatorial parameters}\label{7.2} Consider now a {\bf marked group}\index{marked group} $(\Gamma, S)$\index{$(\Gamma, S)$}, i.e., a finitely generated group $\Gamma$ with a fixed finite symmetric set of generators $S$ not containing the identity $e_\Gamma$. Given $a\in \textrm{FR}(\Gamma, X, \mu)$, we define the {\bf Cayley graph}\index{Cayley graph} of this action, denoted by ${\boldsymbol G}(S,a)$\index{${\boldsymbol G}(S,a)$},  as follows: the vertices of this graph are the points of $X$ and two points $x,y$ are connected by an edge iff $\exists s\in S (s^a (x) = y)$. For more information concerning these graphs, we refer to the survey paper \cite[Section 5, {\bf (E)}]{KMa}, and references therein. We next define four combinatorial parameters associated with these graphs.

The {\bf independence number}\index{independence number} $i_\mu ({\boldsymbol G}(S,a))$\index{$i_\mu ({\boldsymbol G}(S,a))$} is the supremum of the measures of Borel independent sets in ${\boldsymbol G}(S,a)$. The {\bf measurable chromatic number}\index{measurable chromatic number} $\chi_\mu ({\boldsymbol G}(S,a))$\index{$\chi_\mu ({\boldsymbol G}(S,a))$} is the smallest cardinality of a Polish space $Y$ for which there is a $\mu$-measurable coloring $c\colon X\to Y$ of ${\boldsymbol G}(S,a)$. The {\bf approximate measurable chromatic number}\index{approximate measurable chromatic \\ number} $\chi^{ap}_\mu ({\boldsymbol G}(S,a))$ \index{$\chi^{ap}_\mu ({\boldsymbol G}(S,a))$} is the smallest cardinality of a Polish space $Y$ such that for each $\epsilon >0$, there is a Borel set $A$ with $\mu(X\setminus A)<\epsilon$ and a $\mu$-measurable coloring $c\colon A\to Y$ of the induced graph ${\boldsymbol G}(S,a)|A$. Given a matching $M$ in ${\boldsymbol G}(S,a)$, i.e., a set of edges no two of which have a common vertex, we denote by $X_M$\index{$X_M$} the set of matched vertices, i.e., those belonging to some edge in $M$. We define the {\bf matching number}\index{matching number} $m_\mu ({\boldsymbol G}(S,a))$\index{$m_\mu ({\boldsymbol G}(S,a))$} to be one-half of the supremum of $\mu(X_M)$, for $M$ a Borel matching of ${\boldsymbol G}(S,a)$. We say that a Borel matching $M$ is a {\bf perfect matching a.e.} \index{perfect matching a.e.} if $X_M$ is invariant under the action $a$ and $\mu (X_M) =1$. This is equivalent to saying that $m_\mu ({\boldsymbol G}(S,a)) =\frac{1}{2} $ and the sup is attained.

We now have the following monotonicity and invariance results.

\begin{thm}[{\cite[4.1, 4.2, 4.3]{CK}}]\label{4.351}
Let $(\Gamma, S)$ be a marked group. The map 
\[
a\in{\rm{FR}}(\Gamma,X,\mu)\mapsto i_\mu({\boldsymbol G}(S,a))
\]
is lower semicontinuous. In particular, for $a, b\in {\rm FR}(\Gamma, X, \mu)$, we have
\[
a\preceq b\implies i_\mu({\boldsymbol G}(S,a))\leq i_\mu({\boldsymbol G}(S,b)).
\]
Moreover for $a, b\in {\rm FR}(\Gamma, X, \mu)$,
\[
a\preceq b\implies \chi^{ap}_\mu({\boldsymbol G}(S,a))\geq \chi^{ap}_\mu({\boldsymbol G}(S,b)).
\]
Thus both $i_\mu$ and $\chi_\mu^{ap}$ are invariants of weak equivalence.
\end{thm}
The parameter $\chi_\mu$ is not an invariant of weak equivalence; see the paragraph following \cite[Theorem 5.40]{KMa}. However, as shown in 
\cite[Proposition 5.41]{KMa}, 
\[
a\sqsubseteq b \implies \chi_\mu({\boldsymbol G}(S,a))\geq \chi_\mu({\boldsymbol G}(S,b)),
\]
so it is an invariant of weak isomorphism. 

The supremum in the definition of independence number may not be attained (see, e.g., \cite[page 148]{CK}) but we have the following:

\begin{thm}[{\cite[Theorem 2]{CKT-D}}]\label{4.33}

Let $(\Gamma, S)$ be a marked group. Then for any $a\in{\rm{FR}}(\Gamma,X,\mu)$, there is $b\in{\rm{FR}}(\Gamma,X,\mu)$ with $a \simeq b$, such that $ i_\mu(\bfg(S, a)) = i_\mu(\bfg(S, b)) $ with the supremum in $i_\mu(\bfg(S, b)) $ attained.
\end{thm}
Also we have the following connection between measurable/approximate measurable chromatic numbers and weak equivalence.

\begin{thm}[{\cite[Theorem 2]{CKT-D}}]\label{444}
Let $(\Gamma, S)$ be a marked group. Then for any $a\in{\rm{FR}}(\Gamma,X,\mu)$, there is $b\in{\rm{FR}}(\Gamma,X,\mu)$ with $a \simeq b$ and 
\[
\chi^{ap}_\mu(\bfg(S, a)) = \chi^{ap}_\mu(\bfg(S, b)) = \chi_\mu(\bfg(S, b)).
\]
\end{thm}
Concerning matching numbers, there is a similar monotonicity result.

\begin{thm}[{\cite[6.1]{CKT-D}}]\label {11.1}
Let $(\Gamma, S)$ be a marked group. Then for any $a, b \in {\rm FR}(\Gamma, X, \mu)$, we have
\[
a\preceq b \implies m_\mu(\bfg(S, a))\leq m_\mu(\bfg(S, b)).
\]
Thus $m_\mu$ is an invariant of weak equivalence.
\end{thm}
Again the supremum in the definition of $m_\mu$ might not be attained; see the paragraph following \cite[Proposition 13.2]{KMa}. But as in \cref{4.33}  we have:
\begin{thm}[{\cite[6.2]{CKT-D}}]\label{68}
For any $a\in {\rm FR}(\Gamma, X, \mu)$, there is $b \in {\rm FR}(\Gamma, X, \mu)$ such that $a\simeq b$ and $m_\mu(\bfg(S, a))= m_\mu(\bfg(S, b))$, with the supremum in $m_\mu(\bfg(S, b))$ attained.
\end{thm}
By combining \cref{11.1} and \cref{68} and results of \cite{CKT-D}, Lyons-Nazarov \cite{LN} and Cs\'{o}ka-Lippner \cite{CL}, one finally obtains the following, see \cite[Section 13]{KM} for details.

\begin{thm}\label{11.4}
Let $(\Gamma, S)$ be a marked group. Then for any $a\in {\rm FR}(\Gamma, X, \mu)$, $m_\mu(\bfg(S, a)) = \frac{1}{2}$ and there is $b \in {\rm FR}(\Gamma, X, \mu)$ such that $a\simeq b$ and $\bfg (S, b)$ admits a Borel perfect matching $\mu$-almost everywhere.
\end{thm}

A more general form of \cref{444} was proved in \cite[Theorem 2]{AE1} concerning {\it almost satisfaction versus proper satisfaction} of local rules. In a similar vein \cite[Theorem 3]{AE1} shows that for every free action $a$ of a non-amenable group, one can find a weakly equivalent action $b$ of that group which satisfies the measurable version of the von Neumann conjecture, i.e., there is a free action $c$ of $\bbF_2$ such that $E_c\subseteq E_b$. Recall that in \cite{GL} it was shown that the measurable version of the von Neumann conjecture holds for the shift action $s_\Gamma$ of a non-amenable group.

Finally consider the {\bf Cheeger constant}\index{Cheeger constant} $h(S,a)$\index{$h(S,a)$} associated with the action $a\in A(\Gamma, X, \mu)$ of a marked group $(\Gamma, S)$, given by 
\[
h(S,a) = \inf \{ \frac{\mu(S^a(A)\setminus A)}{\mu(A)} \colon A \  \textrm{a Borel subset of} \  X, 0<\mu(A)\leq \frac{1}{2} \}
\]
where $S^a(A) = \{\gamma^a(x)\colon \gamma \in S , x\in A\}$. Then we have:

\begin{thm}[{\cite[Lemma 5.1]  {AE}}]\label{711}
For $a,b \in A(\Gamma, X, \mu)$,
\[
a\preceq b \implies h(S,b) \leq h(S,a).
\]
Thus $h(S,a)$ is an invariant of weak equivalence.
\end{thm}

\subsection{The norm of the averaging operator}\label{7.3} Let $(\Gamma, S)$ be a marked group and $\pi$ a unitary representation of $\Gamma$. The {\bf averaging operator}\index{averaging operator} $T_{S,\pi}$\index{$T_{S,\pi}$} on $H_\pi$ is defined by 
\[
T_{S,\pi} (f) =\frac{1}{|S|}\sum_{s\in S }\pi (s)(f).
\]
Then $\pi\preceq \rho \implies \| T_{S,\pi} \| \leq \| T_{S,\rho} \|$. If now $a\in A(\Gamma, X, \mu )$, we let $T_{S,a} = T_{S, \kappa_0^a}$\index{$T_{S,a}$}. Then we have, using \cref{4.1}:

\begin{pro}\label{6.10}
If $a,b \in A(\Gamma, X, \mu)$, then 
\[
a\preceq b \implies \| T_{S,a}\|\leq \|T_{S,b}\|.
\]
Thus the norm of the averaging operator $T_{S,a}$ is an invariant of weak equivalence.

\end{pro}

\newpage
\section{Invariant random subgroups}\label{7}
For a countable group $\Gamma$, we denote by $\textrm{Sub}(\Gamma)$\index{$\textrm{Sub}(\Gamma)$} the compact, metrizable space of subgroups of $\Gamma$ viewed as a closed subspace of the product $2^\Gamma$. The group $\Gamma$ acts continuously on $\textrm{Sub}(\Gamma)$ by conjugation. A probability Borel measure on ${\rm Sub}(\Gamma)$ invariant under this action is called an {\bf invariant random subgroup}\index{invariant random subgroup}, abbreviated {\bf IRS}\index{IRS}. We denote by  $\textrm{IRS}(\Gamma)$\index{$\textrm{IRS}(\Gamma)$} the space of the IRS's on $\Gamma$. This is a compact, metrizable space with the weak* topology. This concept can be viewed as a probabilistic analog of the concept of normal subgroup (which corresponds to an IRS concentrating on a single point); see \cite{AGV}. For $\theta\in {\rm IRS}(\Gamma)$, we denote by $c_\theta$ the conjugacy action of $\Gamma$ on $(\textrm{Sub}(\Gamma), \theta)$.

\begin{remark}
{\rm Note that when the group $\Gamma$ is abelian, ${\rm IRS}(\Gamma)$ is simply the space of all probability Borel measures on ${\rm Sub}(\Gamma)$.
}
\end{remark}

If $a\in A(\Gamma, X, \mu)$, consider the stabilizer function $\textrm{stab}_a: X \to \textrm{Sub}(\Gamma)$\index{$\textrm{stab}_a$} that takes $x$ to its stabilizer $\textrm{stab}_a (x)$ in the action $a$. It is $\Gamma$-equivariant, so the pushforward measure $(\textrm{stab}_a)_* \mu$ is an IRS. This measure is called the {\bf type}\index{type} of $a$, in symbols $\textrm{type}(a)$\index{$\textrm{type}(a)$}. Clearly $a$ is free iff $\textrm{type}(a)$ is the Dirac measure concentrating on $\{e_\Gamma\}$.

Conversely, given an IRS $\theta$, it is shown in  \cite[Proposition 13]{AGV} that there is $a\in A(\Gamma, X, \mu)$ with $\textrm{type}(a)= \theta$. A particular realization of such an $a$ is the so-called $\boldsymbol{\theta}${\bf-random shift}\index{$\theta$-random shift}, see \cite[Proposition 13]{AGV} and also \cite[Section 5.3]{T-D1}, denoted by $s_\theta$\index{$s_\theta$}. When $\theta$ is the Dirac measure concentrating on $\{e_\Gamma\}$, then $s_\theta \cong s_\Gamma$.

We next have the following invariance property for the type:

\begin{thm}[{\cite[Section 4] {AE1}}; see also {\cite[Theorem 5.2]  {T-D1}}]\label{7.1}
\[
a\simeq b\implies {\rm type}(a) = {\rm type}(b)
\]
\end{thm}
We give a proof in Appendix C, \cref{appB}. For amenable groups, the converse is also true.

\begin{thm}[{\cite[Proposition 5.1]{Bu}}; see also {\cite[Theorem 1.8]  {T-D1}},  and {\cite[Theorem 9]   {E}}]\label{72}
If $\Gamma$ is amenable, then
\[
a\simeq b\iff {\rm type}(a) = {\rm type}(b).
\]
\end{thm}
This clearly fails for any non-amenable $\Gamma$, since any such group has many weakly inequivalent free actions (see \cref{3.2}), which of course have the same type.

From \cref{102} and \cref{104} below, it follows that the set $A_\theta(\Gamma, X, \mu)$\index{$A_\theta(\Gamma, X, \mu)$} of actions of type $\theta$ is a (weak equivalence invariant) $G_\delta$ set in the space $A(\Gamma, X, \mu)$.

Tucker-Drob proved the following generalization of \cref{td}. Below for each IRS $\theta$ and $a\in A_\theta (\Gamma,X, \mu)$, we denote by $s_\theta\times_{c_\theta} a$\index{$s_\theta\times_{c_\theta} a$} the {\bf relative independent joining}\index{relative independent joining} of $s_\theta$ and $a$ over the common factor $c_\theta$, see 
\cite[page 126]{Gl}. When the IRS $\theta$ is the Dirac measure concentrating on $\{e_\Gamma\}$, then $s_\theta\times_{c_\theta} a \cong s_\Gamma\times a$. Then we have:

\begin{thm}[{\cite[Theorem 1.5]{T-D1}}] \label{73}
For any $a\in A_\theta(\Gamma, X, \mu)$, $s_\theta\times_{c_\theta} a\simeq a$. In particular $s_\theta$ is minimum, in the sense of weak containment, among all actions  $a\in A_\theta(\Gamma, X, \mu)$.
\end{thm}

There is also an analog of \cref{max} for actions of type $\theta$.
\begin{thm}[{\cite[Theorem 5.15]{T-D1}}] \label{74}
Let $\theta \in {\rm IRS}(\Gamma)$. Then there is a maximum, in the sense of weak containment, among all actions
$a\in A_\theta(\Gamma, X, \mu)$.
\end{thm}



\newpage

\section{Stable weak containment}\label{8}
Tucker-Drob \cite[Appendix B]{T-D1} introduced a variant of weak containment called {\bf stable weak containment}, which is an analog of weak containment  for unitary representations (as opposed to weak containment in the sense of Zimmer).

\begin{dfn}
Let $a\in A(\Gamma, X, \mu)$,  $b\in A(\Gamma, Y, \nu)$ be two actions. We say that $a$ is {\bf stable weakly contained} \index{stable weakly contained}in $b$, in symbols
\[
a\preceq_s b,\index{$\preceq_s$}
\]
if for any partition $A_1, \dots , A_n \in {\rm MALG}_\mu, {\rm finite} \ F\subseteq \Gamma, \epsilon >0$, there is a convex combination $c = \sum_{i=1}^m \lambda_i b $, on say $(Z,\rho)$, and a partition $B_1, \dots , B_n \in {\rm MALG}_\rho$ such that $|\mu (\gamma^a(A_i)\cap A_j)-\rho (\gamma^c(B_i)\cap B_j)|<\epsilon, \forall\gamma\in F, i,j\leq n$.
\end{dfn}
We also put $a\simeq_s b$\index{$\simeq_s$} for the associated notion of {\bf stable weak equivalence}\index{stable weak equivalence}.

Equivalently (see \cite[Appendix B]{T-D1})

\[
a\preceq_s b \iff a\preceq i_\Gamma\times b\iff i_\Gamma \times a\preceq i_\Gamma\times b\]
and 
\[
a\simeq_s b \iff i_\Gamma\times a \simeq i_\Gamma \times b.
\]
Clearly $a\preceq b \implies a\preceq_s b$ and from \cref{3.8} it follows that if $a$ is ergodic, then $a\preceq b \iff a\preceq_s b$. Also weak equivalence and stable weak equivalence coincide for amenable groups. This follows from \cref{72}, which implies that $i_\Gamma\times a\simeq a$.

\begin{thm}[{\cite[Theorem 5.1]{Bu}}]\label{9.2}
If $\Gamma$ is amenable, then for any actions $a,b$, $a\preceq b \iff a\preceq_s b$ and so $a\simeq b \iff a\simeq_s b$.
\end{thm}
On the other hand, by \cref{38}, we see that if  $a$ is strongly ergodic, and for every non-amenable group $s_\Gamma$ is such an action, then $a\simeq_s i_\Gamma\times a$ but $a\not\simeq i_\Gamma\times a$. Thus stable weak equivalence coincides with weak equivalence iff the group is amenable.


In \cite[Theorem 1.1]{T-D1} it is shown that $\overline{E(i_\Gamma\times a, K)}$ is equal to the closed convex hull $\overline{co}(E(a,K))$\index{$\overline{co}(E(a,K))$} of $E(a,K)$, for any compact, metrizable $K$. From this we have the following analog of the characterization of weak containment in \cref{22}, (2). 

\begin{thm}[{\cite[Proposition B.2] {T-D1}}]\label{92} The following are equivalent for any two actions $a,b$:

\begin{enumerate}[\upshape (i)]
\item $a \preceq_s b$.
\item For each compact metrizable space $K$, $E(a,K) \subseteq \overline{co}(E(b,K))$.
\item For $K= 2^\bbN$, $E(a,K) \subseteq \overline{co}(E(b,K))$.
\item For each finite space $K$, $E(a,K) \subseteq \overline{co}(E(b,K))$.
\end{enumerate}

\end{thm}
As in the second paragraph following  \cref{td}, it can be shown that any non-amenable group has a continuum size $\preceq_s$-antichain in $\textrm{FR}(\Gamma, X, \mu)$.\

In \cite[Theorem 8.1]{BT-D1} it is proved that $\preceq_s$ persists through the ergodic decomposition, i.e., $a\preceq_s b$ iff there is a coupling of the ergodic decomposition measures of $a,b$ which concentrates on the pairs of ergodic components $a',b'$ of $a,b$, resp., with $a'\preceq_s b'$ (equivalently $a'\preceq  b'$). Similarly for stable weak equivalence.

 Concerning the connection with the Koopman representation, we have in (partial) analogy with \cref{4.1}, the following result due to Bowen and Tucker-Drob:
 
\begin{thm}
\[
a\preceq_s b \implies \kappa^a\preceq \kappa^b
\]
\end{thm}
This is because $\kappa^{a\times b} \cong \kappa^a\otimes  \kappa^b$, therefore if $a\preceq_s b$, $a\preceq i_\Gamma \times b$, so $\kappa^a\preceq_Z \kappa^{i_\Gamma \times b}\cong \kappa^{i_\Gamma}\otimes \kappa^b\cong \infty \cdot\kappa^b$, thus $\kappa^a\preceq \kappa^b$.

For ergodic $a$, it follows that $a\preceq_s b \implies \kappa^a_0\preceq \kappa^b_0$, since $a\preceq_s b \iff a\preceq b$. On the other hand, Tucker-Drob pointed out that $a\preceq_s b \implies \kappa_0^a\preceq \kappa_0^b$ may fail if $a$ is not ergodic. Indeed let $\Gamma$ be non-amenable and let $a=\frac{1}{2}s_\Gamma + \frac{1}{2}s_\Gamma$. Then $a\preceq_s s_\Gamma$ but $\kappa_0^a$ contains the trivial 1-dimensional representation, which is not weakly contained in $\kappa_0^{s_\Gamma}$, therefore $\kappa_0^a\not\preceq \kappa_0^{s_\Gamma}$. 

For $\pi\in {\rm Rep}(\Gamma, H_\pi)$, let $\pi_1$ be the restriction of $\pi$ to the orthogonal of the $\Gamma$-invariant vectors in $H_\pi$. Then, using ultrapowers of unitary representations, Tucker-Drob shows that $\pi\preceq \rho\implies \pi_1\preceq\rho_1$. It then follows that 
\[
a\preceq_s b \implies \kappa_1^a\preceq \kappa_1^b.
\]



Since $\textrm{type}(i_\Gamma\times a) = \textrm{type} (a)$, it follows from \cref{7.1} that the type is also an invariant of stable weak equivalence:

\begin{cor}
$a\simeq_s b\implies {\rm type}(a) = {\rm type} (b)$
\end{cor}

Also the monotonicity and invariance properties in \cref{6.1}, \cref{4.351} and \cref{11.1} hold as well for stable weak equivalence.



As before a {\bf free stable weak equivalence class}\index{free stable weak equivalence class} is one consisting of free actions and an {\bf ergodic stable weak equivalence class}\index{ergodic stable weak equivalence class} is one which contains an ergodic action. Then we have the following analog of \cref{3.15}:

\begin{thm}
A group is amenable iff every free stable weak equivalence class is ergodic.
\end{thm}
This is clear for amenable groups, since an amenable group has exactly one free stable weak equivalence class, namely that of the shift. If a group is not amenable, then it has at least two free stable weak equivalence classes and then it follows that there must exist a non-ergodic one, using \cite[Proposition 7.3]{Bu} (see also \cref{1016} below).

\begin{remark}
{\rm 
In \cite[Question 3.14]{T-D1} it was asked whether it is true that  every free action is stable weak equivalent to one with countable ergodic decomposition. Tucker-Drob shows that from the results in \cite{BT-D1} it follows that the answer is negative.
}
\end{remark}

\newpage
\section{The space of weak equivalence classes}\label{9}


We denote by $\undertilde{A}(\Gamma, X, \mu) = A(\Gamma, X, \mu)/ \! \simeq$\index{$\undertilde{A}(\Gamma, X, \mu)$} the set of weak equivalence classes. For each $a\in A(\Gamma, X, \mu)$, we let $\undertilde{a}$\index{$\undertilde{a}$} be its weak equivalence class. We will also use $\undertilde{a}, \undertilde{b}, \dots$ as variables over $\undertilde{A}(\Gamma, X, \mu)$. 

\subsection{The topology on the space of weak equivalence classes}\label{101} In the paper \cite{AE1}, Ab\'{e}rt and Elek defined a topology on $\undertilde{A}(\Gamma, X, \mu)$ which we now describe (in an equivalent form that is implicit in that paper).

For each $a\in A(\Gamma, X, \mu)$ recall from \cref{22}, (1) the definition of the sets 
\[
C_{n,k}(a)\in H([0,1]^{n\times k\times k}),
\]
where by $H([0,1]^{n\times k\times k})$ we denote the hyperspace of compact subsets of the cube $[0,1]^{n\times k\times k}$, equipped with the Vietoris topology, which is induced by the Hausdorff metric $d_H$ in this hyperspace. By \cref{22}, (1) the map 
\[
\undertilde{a}\mapsto (C_{n,k}(a))_{n,k}\in \prod_{n,k} H([0,1]^{n\times k\times k})
\]
is a bijection of $\undertilde{A}(\Gamma, X, \mu)$ with a subset of this product space and we define the topology of $\undertilde{A}(\Gamma, X, \mu)$ by transferring the relative topology of this subset back to $\undertilde{A}(\Gamma, X, \mu)$ by this bijection. 
We now have:

\begin{thm}[{\cite[Theorem 1]{AE1}}]\label{AE1}
The space $\undertilde{A}(\Gamma, X, \mu)$ is compact, metrizable.
\end{thm}

We can define a compatible metric for this topology as follows: Put for $a,b\in A(\Gamma, X, \mu)$:

\[
\delta(a,b) = \sum_{n,k}\frac{1}{2^{n+k}} d_H (C_{n,k}(a), C_{n,k}(b)).
\]
Then $\delta$\index{$\delta$} is a pseudometric on $A(\Gamma, X, \mu)$ and $\delta(a,b) = 0 \iff a\simeq b$. Thus $\delta$ descends to a metric $\undertilde{\delta}$\index{$\undertilde{\delta}$} on $\undertilde{A}(\Gamma, X, \mu)$. This gives the topology of $\undertilde{A}(\Gamma, X, \mu)$. 

We will give the proof of \cref{AE1} in Appendix A, \cref{app}.

One can describe limits of sequences in $\undertilde{A}(\Gamma, X, \mu)$ as follows:

\begin{thm}[{\cite[Theorem 2.22]{C}}]\label{seq}
Let $\undertilde{a}_n,\undertilde{a}\in \undertilde{A}(\Gamma, X, \mu), n\in \bbN$, and let $\mathcal{U}$ be a non-proncipal ultrafilter on $\bbN$. If $\lim_{n\to \mathcal{U}} \undertilde{a}_n = \undertilde{a}$, then $a \simeq \prod_n a_n /\mathcal{U}$.

In particular the following are equivalent:
\begin{enumerate}[\upshape (i)]
\item $\undertilde{a}_n \to \undertilde{a}$.
\item $\prod_n a_n /\mathcal{U} \simeq a$, for every non-principal ultrafilter $\mathcal{U}$ on $\bbN$.

\end{enumerate}
\end{thm}

We next discuss an alternative description of this topology, coming from \cite{T-D1}. From \cref{25}, the map $\undertilde{a}\mapsto \overline{E(a, 2^\bbN)}$ is a bijection from $\undertilde{A}(\Gamma, X, \mu)$ to a subset of the hyperspace of compact subsets of $M_s((2^\bbN)^\Gamma)$, which is of course equipped with the weak* topology. In \cite[Theorem 5.1]{T-D1} it is shown that if we transfer the relative topology of this subset back to $\undertilde{A}(\Gamma, X, \mu)$ via this bijection, we obtain a compact, metrizable topology on $\undertilde{A}(\Gamma, X, \mu)$. It was shown in \cite[Theorem 3.1]{Bu} that this coincides with the topology defined before

Finally another description of this topology can be given using the concepts of the model theory of metric structures discussed in \cref{22}, (5). Let $(\varphi_n)$ be a sequence of infimum sentences such that for every $\epsilon >0$ and each infimum sentence $\varphi$, there is $n \in \bbN$ such that for {\it every} $a\in A(\Gamma, X, \mu)$, $|\varphi^a - \varphi_n^a|< \epsilon$, see \cite[Section 6]{BYBHU}. Then we have for $a,b \in A(\Gamma, X, \mu )$,
\[
a\simeq b \iff \forall n (\varphi_n^a = \varphi_n^a),
\]
and therefore the map
\[
\pi\colon\undertilde{a}\mapsto (\varphi_n^a)\in [0,1]^\bbN
\]
is a bijection of $\undertilde{A}(\Gamma, X, \mu)$ with a subset $\Phi$ of $[0,1]^\bbN$. Therefore one can define a new topology on $\undertilde{A}(\Gamma, X, \mu)$ by transferring the relative topology of $\Phi$ to $\undertilde{A}(\Gamma, X, \mu)$. One next shows, by an argument similar to that in \cref{app}, that this subset is closed in $[0,1]^\bbN$ and therefore this new topology on $\undertilde{A}(\Gamma, X, \mu)$ is compact. Finally, the inverse map $\pi^{-1}\colon \Phi \to \undertilde{A}(\Gamma, X, \mu)$ can be shown (using, for example, \cref{seq}) to be continuous to the topology on $\undertilde{A}(\Gamma, X, \mu)$ introduced in the beginning of \cref{101},  thus our new topology coincides with that one.





\begin{remark}
{\rm Instead of the sentences $(\varphi_n )$ as above, one could also use the sentences $(\varphi_{n,k,r})$ as in \cref{211}, with $r$ taking only rational values.
}

\end{remark}

The following result relates the topology of $\undertilde{A}(\Gamma, X, \mu)$ with its quotient topology (from $A(\Gamma, X, \mu)$).

\begin{thm}[{\cite[Theorem 5.7]{T-D1}}]\label{102}
The topology of $\undertilde{A}(\Gamma, X, \mu)$ is strictly bigger (finer) than the quotient topology. The map $a\in A(\Gamma, X, \mu)\mapsto \undertilde{a}\in \undertilde{A}(\Gamma, X, \mu)$  is open and Baire class 1 but not continuous.
\end{thm}

Using the first statement of \cref{102}, note that to prove that this map is open, it is enough to show that it is open with respect to the quotient topology. In view of \cref{t1}, this is a consequence of the following general fact: If a group $G$ acts by homeomorphisms on a topological space $P$ and we define the following equivalence relation on $P$, $x\sim y \iff \overline{G\cdot x} = \overline{G\cdot y}$, then the map $\pi (x) = [x]_\sim$ is open with respect to the quotient topology on $P/ \!\! \sim$. Indeed, if $U\subseteq P$ is open, then $\pi^{-1}(\pi (U)) = G\cdot U$, which is open in $P$.

\begin{cor}\label{1050}
If $\undertilde{a}_n \to \undertilde{a}$, there is a sequence $n_0 < n_1 < \dots$ and for each $i$ an action $b_{n_i}\cong a_{n_i}$ such that $b_{n_i}\to a$.
\end{cor}

To see this, notice that, since the map $a\mapsto \undertilde{a}$ is open, for each open set $U$ containing $a$, we have that for all large enough $n$, there is an action $b_n \cong a_n$ with $b_n \in U$.

By \cref{7.1} we have a well defined function 
\[
{\rm type}\colon \undertilde{A}(\Gamma, X, \mu) \to {\rm IRS}(\Gamma)
\]
given by ${\rm type}(\undertilde{a}) = {\rm type}(a)$.

\begin{thm}[{\cite[Theorem 5.2, (2)]{T-D1}}]\label{10.6}
 The function ${\rm type}\colon \undertilde{A}(\Gamma, X, \mu) \to {\rm IRS}(\Gamma)$ is continuous.
\end{thm}

In particular, by \cref{72}, if $\Gamma$ is amenable, then ${\rm type}\colon \undertilde{A}(\Gamma, X, \mu) \to {\rm IRS}(\Gamma)$ is a homeomorphism.

Denote, for each $\theta\in {\rm IRS}(\Gamma)$, by $\undertilde{A}_\theta(\Gamma, X, \mu)$\index{$\undertilde{A}_\theta(\Gamma, X, \mu)$} the subset of $\undertilde{A}(\Gamma, X, \mu)$ consisting of all $\undertilde{a}$ of type $\theta$. In particular, for $\theta$ the Dirac measure concentrating on $\{e_\Gamma\}$, we have that $\undertilde{A}_\theta(\Gamma, X, \mu) = \undertilde{{\rm FR}}(\Gamma, X, \mu)$\index{$\undertilde{{\rm FR}}(\Gamma, X, \mu)$} is the set of free weak equivalence classes.

\begin{cor}[{\cite{AE1}, see \cite[Corollary 5.5]{T-D1}}]\label{104}
For $\theta\in {\rm IRS}(\Gamma)$, the set $\undertilde{A}_\theta(\Gamma, X, \mu)$, and thus in particular $\undertilde{\rm FR}(\Gamma, X, \mu)$, is compact.
\end{cor}
For the space ${\rm FR}(\Gamma, X, \mu)$ it is known that if $\Gamma$ does not have property (T), then its ergodic elements are dense, by a result of Glasner-Weiss \cite{GW}, see, e.g., \cite[Theorem 12.2]{K}. We have an analogous result for $\undertilde{\rm FR}(\Gamma, X, \mu)$ for the free groups.

\begin{thm}[{\cite[Theorem 1.4]{Bu}}]\label{105}
If $\Gamma$ is a free group, then the free ergodic elements of $\undertilde{{\rm FR}}(\Gamma, X, \mu)$ are dense in $\undertilde{{\rm FR}}(\Gamma, X, \mu)$.
\end{thm}

\begin{prob}
Is \cref{105} true for any group without property {\rm (T)}?
\end{prob}

\subsection{Continuity properties of functions} We have seen earlier various functions defined on the space $A(\Gamma, X, \mu)$ (or some of its subspaces) that are invariant under weak equivalence and thus descend to functions on $\undertilde{A}(\Gamma, X, \mu)$. We will consider here the continuity properties of such functions.

Consider first the restriction operation, see \cref{restr}. It is easy to see the following:

\begin{pro}
Let $\Delta\leq \Gamma$. The map

\[
\undertilde{a}\in \undertilde{A}(\Gamma, X, \mu)\mapsto \undertilde{a|\Delta}\in \undertilde{A}(\Delta, X, \mu)
\]
is continuous.
\end{pro}

Consider next the co-induction operation, see \cref{321}:

\begin{prob}
Let $\Gamma\leq \Delta$. Is the operation $\undertilde{a}\mapsto \undertilde{{\rm CIND}}^\Delta_\Gamma(a)$ continuous?
\end{prob}
Here $\undertilde{{\rm CIND}}^\Delta_\Gamma(a)$ is the weak equivalence class of ${\rm CIND}^\Delta_\Gamma(a)$. 
It should be pointed out also that ${\rm CIND}^\Delta_\Gamma(a)$\index{$\undertilde{{\rm CIND}}^\Delta_\Gamma(a)$} acts on the space $(X^T, \mu^T)$, where $T$ is a transversal for the left cosets of $\Gamma$ containing $e_\Gamma$.

For each $t\in T, a\in A(\Gamma, X, \mu)$, let $a_t \in A(\Gamma, X, \mu)$ be defined by $\gamma^{a_t} = (t^{-1}\gamma t)^a$. Then clearly $a\preceq b\iff a_t\preceq b_t, a\simeq b \iff a_t\simeq b_t$ and it is easy to see that the map $\undertilde{a}\mapsto \undertilde{a_t}$ is continuous on $\undertilde{A}(\Gamma, X, \mu)$. If now $\Gamma\vartriangleleft\Delta$, then ${\rm CIND}^\Delta_\Gamma(a)|\Gamma =\prod_{t\in T} a_t$ (see \cite[page 71]{K}), and in the special case where $\Gamma$ is contained in the centralizer of $T$, we have that ${\rm CIND}^\Delta_\Gamma(a)|\Gamma =\prod_{t\in T} a = a^T$, so if co-induction is a continuous operation, so is the power operation $\undertilde{a}\mapsto \undertilde{a^n}$, for $n =2,3, \dots, \bbN$, which is related to \cref{1021}.

et $\Gamma\leq \Delta$. It is shown in \cite{KQ} that for $a,b\in A(\Gamma, X, \mu)$, ${\rm type}(a) = {\rm type}(b) \implies {\rm type}({\rm CIND}_\Gamma^\Delta (a)) = {\rm type}({\rm CIND}_\Gamma^\Delta (b))$ and this is used to define and compute explicitly a co-induced IRS map
\[
 \theta\in {\rm IRS}(\Gamma)\mapsto {\rm CIND}_\Gamma^\Delta (\theta) \in {\rm IRS}(\Delta).
 \]
  Using this, it is shown in \cite{KQ} that, when $\Delta$ is amenable, then 
\[
\undertilde{{\rm CIND}}_\Gamma^\Delta \ \textrm{is continuous} \iff [\Delta:\Gamma] < \infty  \  \textrm{or} \  {\rm core} (\Gamma)  \textrm{ is trivial},
\]
where ${\rm core}(\Gamma) = \bigcap_{\delta\in \Delta} \delta\Gamma\delta^{-1}$.\index{core of a group}

On the other hand it is shown in \cite{Be} that there is a pair (of non-amenable) groups $\Gamma\leq\Delta$ such that $[\Delta : \Gamma] =2$ and $\undertilde{{\rm CIND}}_\Gamma^\Delta $ is not continuous, even when restricted to free actions. Finally, in \cite{KQ} it us shown that for any $\Gamma\leq\Delta$ with $[\Delta : \Gamma] =\infty$ and ${\rm core} (\Gamma)$ non-trivial, the map $\undertilde{{\rm CIND}}_\Gamma^\Delta $ is not continuous.

\medskip
For the next problem, we note that Ab\'{e}rt and Elek in \cite{AE1} define also an analogous  compact, metrizable topology on the space 
\[
\undertilde{{\rm Rep}}(\Gamma, H) = {\rm Rep}(\Gamma, H)/ \!  \simeq_Z\index{$\undertilde{{\rm Rep}}(\Gamma, H)$}
\]
 of unitary representations of $\Gamma$ on a separable, infinite-dimensional Hilbert space $H$ modulo weak equivalence in the sense of Zimmer. We again denote, for each $\pi\in {\rm Rep}(\Gamma, H)$, by $\undertilde{\pi}\in \undertilde{{\rm Rep}}(\Gamma, H)$ the $\simeq_Z$-equivalence class of $\pi$.

We then consider the Koopman representation operation, see \cref{4.1}:

\begin{prob}\label{107} Is the operation $\undertilde{a}\mapsto \undertilde{\kappa}_0^a$ continuous? Similarly for $\undertilde{\kappa}^a$.
\end{prob}
We now examine the cost function. Recall here that for finitely generated $\Gamma$ the cost function is upper semicontinuous on ${\rm FR}(\Gamma, X, \mu)$ and by \cref{102} the map $a\in {\rm FR}(\Gamma, X, \mu)\mapsto \undertilde{a} \in \undertilde{{\rm FR}}(\Gamma, X, \mu)$ is open. Thus we have:

\begin{thm}
Let $\Gamma$ be finitely generated. Then the map 
\[
\undertilde{a}\in\undertilde{{\rm FR}}(\Gamma, X, \mu)\mapsto C_\mu (a)
\]
is upper semicontinuous.
\end{thm}
The following is an open problem:

\begin{prob}
Let $\Gamma$ be finitely generated. Is the map $\undertilde{a}\in \undertilde{{\rm FR}}(\Gamma, X, \mu)\mapsto C_\mu (a)$ continuous?
\end{prob}

We next consider the parameters discussed in \cref{7.2} and \cref{7.3}. Concerning \cref{4.351},  \cref{68} and \cref{711}, we have:
\begin{thm}[{\cite[Remark 5.6]{T-D1}}]\label{1010} Let $(\Gamma, S)$ be a marked group. Then the maps $\undertilde{a}\mapsto i_\mu({\boldsymbol G}(S,a)), \undertilde{a}\mapsto m_\mu ( {\boldsymbol G}(S,a)),\undertilde{a}\mapsto h(S,a)$ are continuous.
\end{thm}
This can be seen, e.g., using \cref{seq} and the fact that, for each one of the parameters $p(a)$ in \cref{1010}, we have for each sequence $a_n \in A(\Gamma, X, \mu), n\in \bbN$,
\[
 \lim_{n\to\mathcal{U}} p(a_n) = p(\prod_n a_n/\mathcal{U}),
 \] for every non-principal ultrafilter $\mathcal{U}$ on $\bbN$. (Each of these parameters can be defined in the same way for any free measure preserving action of a countable group on an {\it arbitrary} probability space.)
  
Considering \cref{6.10}, we have:
\begin{prob}
Is the map $\undertilde{a}\mapsto \| T_{S,a}\|$ continuous?
\end{prob}

On the other hand, the function $\undertilde{a} \in \undertilde{{\rm FR}}(\Gamma, X, \mu)\mapsto \chi^{ap}_\mu  ({\boldsymbol G}(S,a))$, as in \cref{4.351}, takes only integer values. Since $\undertilde{{\rm FR}}(\Gamma, X, \mu)$ is (path) connected, as we will see in \cref{1012} below, 
if it is continuous, it must be constant. This is of course the case when $\Gamma$ is amenable. But if $\Gamma$ is not amenable and the Cayley graph of $(\Gamma, S)$ is bipartite, this function is not constant, by \cite[Proposition 4.13]{CK}. However, as pointed out by Clinton Conley, for the group $\Gamma = \bbZ/3\bbZ\star \bbZ/2\bbZ$ and $S=\{s,s^2, t\}$, where $s\in \bbZ/3\bbZ$ and $t \in \bbZ/2\bbZ$ are not the identity, the Cayley graph of $(\Gamma, S)$ is not bipartite and $ \chi^{ap}_\mu  ({\boldsymbol G}(S,a)) = 3$, for any $\undertilde{a} \in \undertilde{{\rm FR}}(\Gamma, X, \mu)$, by \cite[Theorem 2.19]{CK} and \cref{4.351}.



\subsection{The partial order}\label{po} The pre-order $\preceq$\index{$\preceq$} on $A(\Gamma, X, \mu)$ descends to a partial order, also denoted by $\preceq$, on $\undertilde{A}(\Gamma, X, \mu)$: $\undertilde{a}\preceq \undertilde{b} \iff a\preceq b$. In view of \cref{22}, (1), it is clear that $\preceq$ is closed (in $\undertilde{A}(\Gamma, X, \mu)^2$) thus compact. 

Also note that in view of \cref{contlog} and \cref{101}, the partial order $\preceq$ on $\undertilde{A}(\Gamma, X, \mu)$ is anti-isomorphic to the closed partial order $(a_n) \leq (b_n) \iff \forall n(a_n\leq b_n)$ on $[0,1]^\bbN$ restricted to a closed subset of 
$[0,1]^\bbN$.

The structure of the partial order is very little understood. From our earlier results in \cref{S3}, we know that it has the cardinality of the continuum and has a maximum element but has a minimum element iff the group is amenable. When restricted to actions of a given type (in particular free actions) it has both a maximum and a minimum element (see \cref{7}). 

We do not know if there is a supremum or infimum for even a single pair of incomparable elements of this partial order. However, as Ronnie Chen pointed out, from compactness it follows that we have the next result, which can be also derived from \cite[Proposition 2.24]{C}. A subset $D$ of $\undertilde{A}(\Gamma, X, \mu)$ is {\bf directed}\index{directed}  iff $\forall \undertilde{a}, \undertilde{b}\in D \exists \undertilde{c}\in D (\undertilde{a},\undertilde{b} \preceq \undertilde{c})$. 

\begin{pro}\label{dir}
Every directed subset of $\undertilde{A}(\Gamma, X, \mu)$ has a least upper bound in $\preceq$.
\end{pro}
To see this, view a directed set $D$ as a net in $\undertilde{A}(\Gamma, X, \mu)$ and take a subnet $D'$ which converges to a point $\undertilde{a}$. Then using the fact that $\preceq$ is closed, it is easy to see that $\undertilde{a}$ is the least upper bound of $D$. Note also that $\undertilde{a}$ is the $\preceq$-maximum element of the closure of $D$. In particular, a directed closed set has a maximum element.

Thus if $\undertilde{a}_0 \preceq \undertilde{a}_1 \preceq \cdots$ is an increasing sequence, then $\lim_n \undertilde{a}_n = \undertilde{a}$ exists and $\undertilde{a}$ is the least upper bound of the sequence $(\undertilde{a}_n)$ (see  \cite[Proposition 2.24]{C}). Also if $a_0 \sqsubseteq a_1 \sqsubseteq \cdots$, $f_n$ is a homomorphism from $a_{n+1}$ to $a_n$, and $a$ is the inverse limit of $(a_n,f_n)$, then $\lim_n \undertilde{a}_n = \undertilde{a}$. This can be seen as follows: Fix $n,k$ in order to show that $\lim_i C_{n,k}(a_i) = C_{n.k}(a)$. Note that $C_{n.k} (a_0) \subseteq C_{n,k}(a_1)\subseteq \cdots\subseteq C_{n,k}(a)$ and so
\[
\lim_i C_{n,k}(a_i) = \overline{\bigcup_i C_{n,k} (a_i)} = C_{n,k} (a).
\]

If $a_0\preceq a_1\preceq \cdots$, then by \cref{28a}, we can find $b_n\simeq a_n$, so that $b_0 \sqsubseteq b_1 \sqsubseteq \cdots$, and therefore if $b$ is any inverse limit of $(b_n)$, then $\lim_n \undertilde{a}_n = \lim_n \undertilde{b}_n = \undertilde{b}$. As a special case, if $(a_n)$ is a sequence of  actions, then the weak equivalence class of $\prod_na_n$ is the limit of the weak equivalence classes of $\prod_{i=0}^n a_i$, $n\in \bbN$. It also follows that if $\undertilde{a}_0 \preceq \undertilde{a}_1 \preceq \cdots$ are strongly ergodic, then $\lim_n \undertilde{a}_n$ is ergodic, since ergodicity is preserved under inverse limits (see \cite[Proposition 6.4]{Gl}).


The corresponding to \cref{dir} fact also holds for directed subsets of the reverse order $\succeq$ and greatest lower bounds.

\subsection{Hyperfiniteness}\label{10.4}
Let ${\rm HYP}(\Gamma, X, \mu)$\index{${\rm HYP}(\Gamma, X, \mu)$} be the set of hyperfinite actions and ${\undertilde {\rm HYP}}(\Gamma, X, \mu)$\index{${\undertilde {\rm HYP}}(\Gamma, X, \mu)$} the set of hyperfinite weak equivalence classes. Then ${\undertilde {\rm HYP}}(\Gamma, X, \mu)$ is closed downwards under $\preceq$ by \cref{329} and is directed, since for actions $a,b\in {\rm HYP}(\Gamma, X, \mu)$, the weak equivalence class of $a\times b$ is in $\undertilde{{\rm HYP}}(\Gamma, X, \mu)$, because $E_{a\times b}\subseteq E_a\times E_b$. Let then $\undertilde{a}^{hyp}_{\infty, \Gamma}$ be the least upper bound of ${\undertilde {\rm HYP}}(\Gamma, X, \mu)$, which is also the $\preceq$-maximum element of the closure of ${\undertilde {\rm HYP}}(\Gamma, X, \mu)$. If $\{\undertilde{a}_n\colon n\in \bbN\}$ is dense in 
${\undertilde {\rm HYP}}(\Gamma, X, \mu)$, then, by the second paragraph following \cref{dir}, we have that $\undertilde{a}^{hyp}_{\infty, \Gamma}$ is the weak equivalence class of $\prod_n a_n$.

If $\Gamma$ is amenable, ${\undertilde {\rm HYP}}(\Gamma, X, \mu) = \undertilde{A} (\Gamma, X, \mu)$, so $\undertilde{a}^{hyp}_{\infty, \Gamma}$ is equal to the $\preceq$-maximum element, which in this case is equal to $\undertilde{s}_\Gamma$. On the other hand, if $\Gamma$ is not amenable, $s_\Gamma$ is strongly ergodic, so $\preceq$-incomparable with any hyperfinite $\undertilde{a}$. Therefore $\undertilde{a}^{hyp}_{\infty, \Gamma} \not= \undertilde{s}_\Gamma$. 

Let us say that a group $\Gamma$ is {\bf approximately amenable}\index{approximately amenable} if there is a sequence $(a_n)$ of hyperfinite actions, $a_n\in A(\Gamma, X, \mu)$, such that for each $\gamma\not=e_\Gamma$, $\mu ({\rm Fix}_{a_n}(\gamma))\to 0$, where for $a\in A(\Gamma, X, \mu)$, ${\rm Fix}_{a}(\gamma) = \{ x\colon \gamma^a(x) =x\}$.\index{${\rm Fix}_{a}(\gamma)$} (The terminology is motivated by the fact that $\Gamma$ is amenable iff there is a free hyperfinite action of $\Gamma$.) For $(a_n)$ as above it follows that $\prod_n a_n$ is free, so $s_\Gamma\preceq \prod_n a_n\preceq a^{hyp}_{\infty, \Gamma}$, therefore $\undertilde{s}_\Gamma\preceq \undertilde{a}^{hyp}_{\infty, \Gamma}$.

It is not clear what is the extent of the class of approximately amenable groups. Clearly every residually amenable group is approximately amenable and it can be shown (see Appendix D, \cref{appD}) that every approximately amenable group is sofic. Clearly every subgroup of an approximately amenable group is approximately amenable and, using induced actions, it can be easily seen that any group that has a finite index approximately amenable subgroup is also approximately amenable. One can also show (see Appendix D, \cref{appD}) the following:

\begin{pro}\label{hypact}
 The following are equivalent for a group $\Gamma${\rm :}
\begin{enumerate}[\upshape (i)]
\item $\Gamma$ is approximately amenable.
\item There is a sequence $(a_n)$ of hyperfinite actions such that for some (resp., every) non-principal ultrafilter ${\cal U}$ on $\bbN$, $\prod_n a_n /{\cal U}$ is free, where $\prod_n a_n /{\cal U}$\index{$\prod_n a_n /{\cal U}$} is the {\bf ultraproduct}\index{ultraproduct} of $(a_n)$, see {\rm \cite[Section 4]{CKT-D}}.
\item There is a sequence $(a_n)$ of hyperfinite actions such that for some (resp., every) non-principal ultrafilter ${\cal U}$ on $\bbN$, $s_\Gamma\sqsubseteq\prod_n a_n /{\cal U}$.
\item There is a sequence $(a_n)$ of hyperfinite actions such that $a_n\to s_\Gamma$ (or equivalently $a_n\to a$, for some free action $a$).
\item There is a sequence $(a_n)$ of hyperfinite actions with $\prod_n a_n$ free.
\item For any $\gamma_1, \dots, \gamma_n\in \Gamma\setminus \{e_\Gamma\}$ and any $\epsilon>0$, there is hyperfinite $a\in A(\Gamma, X, \mu)$ such that $\mu ({\rm Fix}_{a}(\gamma_i))< \epsilon, \forall i\leq n$.
\item For any $\gamma_1, \dots, \gamma_n\in \Gamma\setminus\{e_\Gamma\}$ and any $\epsilon>0$, there is ergodic, hyperfinite $a\in A(\Gamma, Y, \nu)$, where $(Y,\nu)$ is non-atomic or finite (with uniform measure), such that $\mu ({\rm Fix}_{a}(\gamma_i))< \epsilon, \forall i\leq n$.
\item {\rm (Ioana)} For any $\gamma\in \Gamma\setminus\{e_\Gamma\}$, there is a hyperfinite action $a$ such that $\gamma^a\not= id$.
\item $ \Gamma$ embeds (algebraically) into the full group $[E_0]$ of the (unique up to isomorphism) measure preserving, ergodic, hyperfinite equivalence relation $E_0$\index{$E_0$}.

Moreover if $\Gamma$ is approximately amenable, the following holds:
\item There is a sequence $(\theta_n)$ of ergodic {\rm IRS} in $\Gamma$ such that: {\rm (a)} each $\theta_n$ is co-amenable (i.e., concentrates on co-amenable subgroups of $\Gamma$) and {\rm (b)} $\theta_n\to \delta_{e_\Gamma}$ (the Dirac measure on the identity of $\Gamma$).
\end{enumerate}

\end{pro}
We do not know if condition (x) of \cref{hypact} is equivalent to approximate amenability.

\begin{cor}
Let $\Gamma$ have property {\rm (T)}. Then the following are equivalent:

\begin{enumerate}[\upshape (i)]
\item $\Gamma$ is approximately amenable.
\item $\Gamma$ is residually finite.
\end{enumerate}

\end{cor}

\begin{proof}
Clearly (ii) implies (i) (for any $\Gamma$). Assume now $\Gamma$ has property (T) and is approximately amenable. Then by \cref{hypact}, (ix) and \cite[Proposition 4.14]{K}, $\Gamma$ is residually finite.
\end{proof}
It follows from this or \cref{hypact}, (x) and the existence of quasi-finite (i.e., having all proper subgroups finite) non-amenable groups that there are groups that are not approximately amenable.

\begin{remark}
{\rm We note that if $\Gamma$ is approximately amenable, then, by \cref{hypact} (iv), $s_\Gamma$ is in the closure of ${\rm HYP}(\Gamma, X, \mu)$. However it is not the case that $\undertilde{s}_\Gamma$ is in the closure of $\undertilde{\rm HYP}(\Gamma, X, \mu)$. This can be seen as follows: Using the ergodic decomposition theorem and \cref{3.5} (see also \cite[Theorem 3.13]{T-D1}), it follows that if $a\in {\rm HYP}(\Gamma, X, \mu)$, then $i_\Gamma\preceq a$, so $\undertilde{i}_\Gamma \preceq \undertilde{a}$. Therefore if $\undertilde{s}_\Gamma$ is in the closure of $\undertilde{{\rm HYP}}(\Gamma, X, \mu)$, we have $\undertilde{i}_\Gamma\preceq \undertilde{s}_\Gamma$, i.e., $i_\Gamma\preceq s_\Gamma$, contradicting the strong ergodicity of $s_\Gamma$.

This fact illustrates again the failure of continuity of the map $a\mapsto \undertilde{a}$ (see \cref{102}).
}

\end{remark}

\begin{prob}
If $\Gamma$ is not amenable, what is $\undertilde{a}^{hyp}_{\infty, \Gamma}$? Is it the $\preceq$-maximum weak equivalence class?
\end{prob}
If this last question has a positive answer, then, by \cref{1050}, the hyperfinite actions are dense in $A(\Gamma, X, \mu)$.

\subsection{Tempered actions}\label{temper} An action $a\in A(\Gamma, X, \mu)$ is called {\bf tempered}\index{tempered} if $\kappa_0^a\preceq\lambda_\Gamma$ (see \cite{K4}). Recall that $\lambda_\Gamma\simeq\kappa_0^{s_\Gamma}$, so that this is equivalent to $\kappa_0^a\preceq \kappa_0^{s_\Gamma}$.
(In fact, $\kappa_0^{s_\Gamma}\cong \infty\cdot\lambda_\Gamma$, see \cite[page 39]{KL}). Denote by ${\rm TEMP}(\Gamma, X, \mu)$\index{${\rm TEMP}(\Gamma, X, \mu)$} the set of tempered actions. By \cref{4.1} the set of tempered actions if closed downwards under $\preceq$ and so in particular it is $\simeq$-invariant. Let also $\undertilde{{\rm TEMP}}(\Gamma, X, \mu)$ be the set of tempered (i.e., containing tempered actions) weak equivalence classes.

\begin{thm}\label{1022}
The space $\undertilde{{\rm TEMP}}(\Gamma, X, \mu)$ of tempered weak equivalence classes is closed in $\undertilde{A}(\Gamma, X, \mu)$ and has a $\preceq$-maximum element, denoted by $\undertilde{a}_{\infty, \Gamma}^{temp}$.
\end{thm}
We give the proof in Appendix E, \cref{appE}.

In general, $\undertilde{s}_\Gamma\preceq \undertilde{a}_{\infty, \Gamma}^{temp}\preceq\undertilde{a}_{\infty, \Gamma}$. If $\Gamma$ is amenable, clearly $\undertilde{{\rm TEMP}}(\Gamma, X, \mu) = \undertilde{A}(\Gamma, X, \mu)$ and $\undertilde{a}_{\infty, \Gamma}^{temp} = \undertilde{a}_{\infty, \Gamma} = \undertilde{s}_\Gamma$. If $\Gamma$ is not amenable, then $i_\Gamma$ is not tempered, so $\undertilde{a}_{\infty, \Gamma}^{temp} \prec \undertilde{a}_{\infty, \Gamma}$. Bowen (see the last paragraph of \cref{limits}) has shown that if $\Gamma = \bbF_n$, for an appropriate $n$, then there is tempered $a$ such that $s_\Gamma \prec a$, so that $\undertilde{s}_\Gamma\prec \undertilde{a}_{\infty, \Gamma}^{temp}\prec \undertilde{a}_{\infty, \Gamma}$.

\begin{prob}
Is it true that for every non-amenable group $\Gamma$, $\undertilde{s}_\Gamma\prec \undertilde{a}_{\infty, \Gamma}^{temp}$?
\end{prob}

\begin{prob}
For a non-amenable group $\Gamma$, is there an explicit representative of  $\undertilde{a}_{\infty, \Gamma}^{temp}$?
\end{prob}



\subsection{The topology on the space of stable weak equivalence classes} For each $a\in A(\Gamma, X, \mu)$ denote by $\undertilde{a}_s$\index{$\undertilde{a}_s$} its stable weak equivalence class. Let also $\undertilde{A}_s (\Gamma, X, \mu)$\index{$\undertilde{A}_s (\Gamma, X, \mu)$} be the set of all stable weak equivalence classes. Let $K=2^\bbN$. By \cref{92} the map $\undertilde{a}_s\mapsto\overline{co}(E(a,K))$ is a bijection and thus as in Section 10.1 above it can be used to give a topology on $\undertilde{A}_s (\Gamma, X, \mu)$ which is again compact, metrizable. To see this it is enough to check that the range $R$ of this bijection is compact. But $R$ is the image of the set $\{\overline{E(a, K)}\colon a \in A(\Gamma, X, \mu)\}$, which is compact, by the function $L\mapsto \overline{co}(L)$ (in the hyperspace of compact subsets of $M_s(K^\Gamma)$), which is continuous (see \cite[Remark 5.8]{T-D1}). Equivalently we have that the map $\undertilde{a}\mapsto \undertilde{i_\Gamma\times a}$ is a continuous map on $\undertilde{A}(\Gamma, X, \mu)$, so its image is compact, and the map $\undertilde{a}_s\mapsto \undertilde{i_\Gamma\times a}$ is a homeomorphism of $\undertilde{A}_s (\Gamma, X, \mu)$ with this image. Thus $\undertilde{A}_s (\Gamma, X, \mu)$ can be also viewed as a compact subspace of $\undertilde{A} (\Gamma, X, \mu)$. 

Also $\undertilde{A}_s (\Gamma, X, \mu)$ is the quotient space of $\undertilde{A}(\Gamma, X, \mu)$ under the stable weak equivalence relation with associated quotient map $\pi(\undertilde{a}) = \undertilde{a}_s
$. It also turns out that the topology of $\undertilde{A}_s(\Gamma, X, \mu)$ is the quotient topology. To see this we verify that for each closed set $F$ in $\undertilde{A}_s(\Gamma, X, \mu)$, the preimage $\pi^{-1}(F)$ is closed in $\undertilde{A}(\Gamma, X, \mu)$. Let $F' = \{ \undertilde{i_\Gamma\times a}\colon \undertilde{a}_s \in F \}$. Then $F'$ is closed in $\undertilde{A}(\Gamma, X, \mu)$ and thus $\pi^{-1}(F) = \{\undertilde{a}\colon \undertilde{i_\Gamma\times a}\in F'\}$ is also closed.

It is now clear that the map $\undertilde{a}_s\mapsto {\rm type}(a)$ is continuous and therefore the sets $(\undertilde{A}_s)_\theta(\Gamma, X, \mu)$\index{$(\undertilde{A}_s)_\theta(\Gamma, X, \mu)$} and $\undertilde{{\rm FR}}_s (\Gamma, X, \mu)$\index{$\undertilde{{\rm FR}}_s (\Gamma, X, \mu)$} are compact. Also the analog of \cref{105} holds.

We also have the compact partial order $\undertilde{a}_s \preceq \undertilde{b}_s \iff i_\Gamma\times a \preceq i_\Gamma\times b$ on $\undertilde{A}_s (\Gamma, X, \mu)$, under which $\undertilde{A}_s (\Gamma, X, \mu)$ can be identified with a suborder of $\preceq$ on  $\undertilde{A} (\Gamma, X, \mu)$.

Finally, recall from \cref{8} that we have $\undertilde{A}_s (\Gamma, X, \mu) =\undertilde{A} (\Gamma, X, \mu)$ iff $\Gamma$ is amenable.

\subsection{Convexity} The space $\undertilde{A}(\Gamma, X, \mu)$ carries also a convex structure in the following sense. Given $\undertilde{a}, \undertilde{b} \in \undertilde{A}(\Gamma, X, \mu)$ and $\lambda\in [0,1]$ we define the convex combination: 
\[
\lambda\undertilde{a} + (1-\lambda )\undertilde{b} = \undertilde{\lambda a+(1-\lambda)b}.
\]
This requires a bit of explanation. The action $\lambda a+(1-\lambda)b$ is on the space $(Y,\nu)$, where $Y=X \sqcup X$ is the disjoint sum of two copies of $X$, each of which carries a copy of $\mu$ but weighted with ratios $\lambda$ and $1-\lambda$, resp. This standard probability space is non-atomic, so isomorphic to $(X,\mu)$, thus we can view this action as belonging to $A(\Gamma, X, \mu)$. It is straightforward that its weak equivalence class depends only on the weak equivalence classes of $a$ and $b$. This convex combination is a continuous function of the three variables $\lambda, \undertilde{a}, \undertilde{b}$, see \cite[Proposition 4.1]{Bu}, so it equips the space  $\undertilde{A}(\Gamma, X, \mu)$ with the structure of what is called a {\bf topological weak convex space}\index{topological weak convex space}, see \cite[Section 2.1 and Proposition 4.1]{Bu}. A consequence of these results is the following:

\begin{thm}[{\cite[Corollary 4.1]{Bu}}]\label{1012}
The space $\undertilde{A}(\Gamma, X, \mu)$ is path connected. Similarly for the subspaces $\undertilde{A_\theta}(\Gamma, X, \mu)$ and in particular $\undertilde{{\rm FR}}(\Gamma, X, \mu).$

\end{thm}
The concept of {\bf extreme point}\index{extreme point} in a topological weak convex space is defined in the usual way. One then has an analog of the Krein-Milman Theorem:

\begin{thm}[{\cite[Theorem 1.2]{Bu}}]
The space $\undertilde{A}(\Gamma, X, \mu)$ is the closed convex hull of its extreme points. Similarly for the subspaces $\undertilde{A_\theta}(\Gamma, X, \mu)$ and in particular $\undertilde{{\rm FR}}(\Gamma, X, \mu).$
\end{thm}

A general fact about the extreme points is the following:

\begin{thm}[{\cite[Theorem 1.3]{Bu}}]
If $\undertilde{a}$ is an extreme point, then in the ergodic decomposition of $a$ almost all ergodic components are weakly equivalent.
\end{thm}

When $\Gamma$ is amenable, the map $\undertilde{a}\mapsto {\rm type}(a)$ is an affine homeomorphism of 
$\undertilde{A}(\Gamma, X, \mu)$ with the compact, convex subset ${\rm IRS}(\Gamma)$ of the separable Banach space of (finite) measures on ${\rm Sub}(\Gamma)$. It is the closed convex hull of its extreme points, which in this case are exactly the ergodic IRS's, and is known to be a Choquet simplex (see, e.g., \cite[Proposition 8.6]{Gl}). Thus for amenable groups the extreme points of $\undertilde{A}(\Gamma, X, \mu)$ are the ones that have ergodic type. Note that if $\undertilde{a}$ has type $\theta$, which concentrates on infinite index subgroups, then $\undertilde{a}$ is ergodic iff $\theta$ is ergodic; see \cite[Theorem 5.11]{T-D1}. Note also that, for example, when $\Gamma=\bbZ$ this Choquet simplex is a Bauer simplex.

When the group is not amenable it is not clear what are the extreme points in $\undertilde{A}(\Gamma, X, \mu)$. However every strongly ergodic $\undertilde{a}$ is an extreme point. Also in the non-amenable case, by \cref{38}, if we take for example $\undertilde{a} = \undertilde{s}_\Gamma$, we have that $\undertilde{a}\not= \frac{1}{2}\undertilde{a} + \frac{1}{2}\undertilde{a}$, thus in this case the space $\undertilde{A}(\Gamma, X, \mu)$ cannot be realized as a compact convex subset of a topological vector space. We will next see how this problem can be overcome by passing to $\undertilde{A}_s(\Gamma, X, \mu)$.

The concept of convex combination in the space $\undertilde{A}_s(\Gamma, X, \mu)$ is defined as in the space $\undertilde{A}(\Gamma, X, \mu)$ and it forms again a topological weak convex space but in this case we have a much stronger statement:

\begin{thm}[{\cite[Theorem 1.5]{Bu}}]
For any group $\Gamma$, the space $\undertilde{A}_s(\Gamma, X, \mu)$ is affinely homeomorphic to a compact, convex subset of a separable Banach space. Similarly for  $(\undertilde{A}_s)_\theta(\Gamma, X, \mu)$ and thus in particular
 $\undertilde{{\rm FR}}_s(\Gamma, X, \mu)$.
 \end{thm}
 \begin{thm}[{\cite[Proposition 7.3]{Bu} and Bowen--Tucker-Drob}]\label{1016}
The extreme points of $\undertilde{A}_s(\Gamma, X, \mu)$ are exactly the ergodic stable weak equivalence classes.
 \end{thm}

 \begin{thm}[{\cite[Theorem 10.1]{BT-D1}}]
The space $\undertilde{A}_s(\Gamma, X, \mu)$ is a Choquet simplex. Similarly for $\undertilde{{\rm FR}}_s(\Gamma, X, \mu)$.
 \end{thm}

 Concerning the convex structure of $(\undertilde{A}_s)_\theta(\Gamma, X, \mu)$ the following are due to Tucker-Drob:
 
(1) When $\theta$ is ergodic, then $(\undertilde{A}_s)_\theta(\Gamma, X, \mu)$ is a face of $\undertilde{A}_s(\Gamma, X, \mu)$, since the map $\undertilde{a}_s\mapsto {\rm type}(a)$ is afffine. Therefore its extreme points are the ergodic stable weak equivalence classes of type $\theta$ and $(\undertilde{A}_s)_\theta(\Gamma, X, \mu)$ is a Choquet simplex. (In particular this of course includes the case of $\undertilde{{\rm FR}}_s(\Gamma, X, \mu)$.)  
 
(2)  On the other hand if $\theta$ is not ergodic, consider its ergodic decomposition $\theta = \int \theta' \, d\eta  (\theta' )$. Then the space $(\undertilde{A}_s)_\theta(\Gamma, X, \mu)$ can be identified with the space of all measurable maps $f\colon {\rm IRS}_\Gamma \to \undertilde{A}_s(\Gamma, X, \mu)$ so that $f(\theta') \in (\undertilde{A}_s)_{\theta'}(\Gamma, X, \mu)$, $\eta$-a.e.,  where two such maps are identified if they agree $\eta$-a.e. From this it follows that the extreme points of $(\undertilde{A}_s)_\theta(\Gamma, X, \mu)$ are the measurable maps $f$ as above such that for $\eta$-almost all $\theta'$, $f(\theta')$ is an ergodic stable weak equivalence class of type $\theta'$.

From this it can be shown that when $\theta$ is not ergodic, $(\undertilde{A}_s)_\theta(\Gamma, X, \mu)$ is not a Choquet simplex, except in the degenerate case when the measure $\eta$ contains an atom, say $\theta _0$, with the property that for $\eta$-a.e. $\theta' \neq \theta _0$, the space $(\undertilde{A}_s)_{\theta'}(\Gamma, X, \mu)$ consists of a single point.


When $\Gamma$ has property (T), we have the following result:

\begin{thm}[{\cite[Theorem 11.1]{BT-D1}}]\label{10250}
If $\Gamma$ has property {\rm (T)}, then $\undertilde{A}_s(\Gamma, X, \mu)$ is a Bauer simplex and similarly  for $\undertilde{{\rm FR}}_s(\Gamma, X, \mu)$.
\end{thm}
It also follows from the results mentioned above that more generally if $\Gamma$ has property {\rm (T)} and $\theta$ is ergodic, then $(\undertilde{A}_s)_\theta(\Gamma, X, \mu)$ is a Bauer simplex.
\begin{prob}
If $\Gamma$ does not have property {\rm (T)} is $\undertilde{{\rm FR}}_s(\Gamma, X, \mu)$ the Poulsen simplex? 
\end{prob}
From \cref{105}, we have
\begin{thm}[{\cite[Corollary 1.1]{Bu}}]
Let $\Gamma$ be a non-abelian free group. Then the space $\undertilde{{\rm FR}}_s(\Gamma, X, \mu)$ is the Poulsen simplex.
\end{thm}
\begin{remark} {\rm It is shown in \cite[Theorem 5.1]{BT-D1} that for $K= 2^\bbN$, the compact convex set $\overline{co}(E(a,K))$, for ergodic $a$, is a Choquet simplex whose extreme points are extreme points of $M_s(K^\Gamma)$. Moreover, $a$ is strongly ergodic iff $\overline{E(a,K)}$ is the set of extreme points of $\overline{co}(E(a,K))$ and the latter is therefore a Bauer simplex. If on the other hand, $a$ is ergodic but not strongly ergodic,  $\overline{co}(E(a,K))$ is the Poulsen simplex.
}
\end{remark}

\subsection{Ergodicity} Denote by $\undertilde{{\rm ERG}}(\Gamma, X, \mu)$ the set of ergodic weak equivalence classes and by $\undertilde{{\rm ERG}}_s (\Gamma, X, \mu)$ its stable counterpart. It follows from \cref{1016} that $\undertilde{{\rm ERG}}_s (\Gamma, X, \mu)$ is a $G_\delta$ set in $\undertilde{A}_s(\Gamma, X, \mu)$. The following seems to be open:

\begin{prob}
Is $\undertilde{{\rm ERG}}(\Gamma, X, \mu)$ a $G_\delta$ set in $\undertilde{A}(\Gamma, X, \mu)$?

\end{prob}

If the group $\Gamma$ has property (T), then ${\rm ERG} (\Gamma, X, \mu) = {\rm SERG} (\Gamma, X, \mu) $ and from \cref{maxerg} and \cref{cor3.6}, we have $\undertilde{a}\in \undertilde{{\rm ERG}}(\Gamma, X, \mu) \iff \undertilde{a} \preceq \undertilde{a}^{erg}_{\infty, \Gamma}$, so $\undertilde{{\rm ERG}}(\Gamma, X, \mu)$ is closed. Also by \cref{10250},  $\undertilde{{\rm ERG}}_s (\Gamma, X, \mu)$ is also closed.

\subsection{Multiplication} We also have a well defined multiplication operation on the space $\undertilde{A}(\Gamma, X, \mu)$: $\undertilde{a} \times \undertilde{b} =\undertilde{a\times b}$. Again this perhaps requires a bit of explanation. Given $a,b \in A(\Gamma, X, \mu)$, the product $a\times b$ is an action in $A(\Gamma, X^2, \mu^2)$. Since $(X^2, \mu^2)$ is non-atomic, it is isomorphic to $(X,\mu)$, so we can view this as action in $A(\Gamma, X, \mu)$. It is straightforward that its weak equivalence class depends only on the weak equivalence classes of $a$ and $b$. Note also that $\undertilde{a},\undertilde{b}\preceq \undertilde{a} \times \undertilde{b}$.

It is now easy to check that $(\undertilde{A}(\Gamma, X, \mu), \times)$ is an abelian semigroup. It has an identity (i.e., is a monoid) iff the group is amenable (in which case the identity is $\undertilde{i}_\Gamma$). On the other hand, 
 $\undertilde{{\rm FR}}(\Gamma, X, \mu)$ is a subsemigroup, in fact an ideal, and $(\undertilde{{\rm FR}}(\Gamma, X, \mu), \times)$ is a semigroup with identity $\undertilde{s}_\Gamma$, by \cref{td}. The main open question here is whether $(\undertilde{A}(\Gamma, X, \mu), \times)$ is a topological semigroup. 
 \begin{prob}\label{1021}
 Is multiplication a continuous operation in $\undertilde{A}(\Gamma, X, \mu)$?
  \end{prob}
It turns out that the answer is positive for amenable groups, as it was shown by the authors and Omer Tamuz. We discuss this next.

One first defines a multiplication operation on ${\rm IRS}(\Gamma)$ as follows: Let $\cap\colon {\rm Sub}(\Gamma)^2\to {\rm Sub}(\Gamma)$ be the intersection map $\cap(H,F) = H\cap F$. Given $\theta, \eta\in {\rm IRS}(\Gamma)$, define
\[
\theta\circ\eta = \cap_*(\theta\times \eta),
\]
 to be the pushforward of $\theta\times \eta $ by $\cap$. Then it is easy to see that this operation is continuous and $({\rm IRS}(\Gamma), \circ)$ is a compact topological semigroup with identity, which is the IRS concentrating on $\{\Gamma\}$, $\delta_{\Gamma}$\index{$\delta_\Gamma$}. Moreover it is clear that for each $a,b\in A(\Gamma, X, \mu)$ and $(x,y)\in X^2$, ${\rm stab}_{a\times b}(x,y) = {\rm stab}_a (x) \cap {\rm stab}_b(y)$. From this it follows that ${\rm type} (a\times b) = {\rm type}(a) \circ {\rm type}(b)$. Therefore ${\rm type}\colon \undertilde{A}(\Gamma, X, \mu)  \to {\rm IRS} (\Gamma)$ is a continuous homomorphism of $(\undertilde{A}(\Gamma, X, \mu), \times)$ into $({\rm IRS}(\Gamma), \circ)$ that takes $\undertilde{i}_\Gamma$ to $\delta_\Gamma$. Similarly for $(\undertilde{A}_s(\Gamma, X, \mu), \times)$. Since by \cref{72} and \cref{10.6}, the type function is a topological and algebraic isomorphism between $(\undertilde{A}(\Gamma, X, \mu), \times) = (\undertilde{A}_s(\Gamma, X, \mu), \times))$ and $({\rm IRS}(\Gamma), \circ)$, when $\Gamma$ is amenable, it follows that multiplication is continuous in $\undertilde{A}(\Gamma, X, \mu)$. Thus we have shown the following:
 
 \begin{thm}[with O. Tamuz]
If $\Gamma$ is amenable, then multiplication is a continuous operation in  $\undertilde{A}(\Gamma, X, \mu)$.
 \end{thm}

Burton in \cite{Bu1} introduced a finer (conjecturally non-separable) topology on $\undertilde{A}(\Gamma, X, \mu)$, for any group $\Gamma$, with respect to which multiplication is continuous. It is induced by the complete metric:
 \[
 d(\undertilde{a}, \undertilde{b}) = \sum_n \frac{1}{2^n}\left( \sup_k d_H(C_{n,k}(a), C_{n,k}(b))\right).
 \]
 
 {\bf Addendum. }It has now been shown in \cite{Be} that for certain non-amenable groups $\Gamma$, including the non-abelian free groups, multiplication is not continuous in $\undertilde{A} (\Gamma, X, \mu)$, in fact the square map $\undertilde{a} \mapsto \undertilde{a^2}$ is not continuous, even when restricted to free actions. It is also shown in \cite{KQ} that for every non-trivial group the infinite power map $\undertilde{a}\mapsto \undertilde{a^\bbN}$ is not continuous.

\subsection{Generalized shifts associated to an ergodic IRS} For $H\leq\Gamma$, the {\bf core} of $H$, in symbols
$\textrm{Core}(H)$,\index{core of a group}\index{${\rm Core}(H)$} is given by ${\rm Core}(H)= \bigcap_{\gamma\in\Gamma} \gamma H \gamma^{-1}$. Note that this is also the kernel of the action of $\Gamma$ on $\Gamma/H$, so the core of $H$ is trivial iff the action of $\Gamma$ on $\Gamma/H$ is faithful. In this case, we say that $H$ is {\bf core-free}\index{core-free}.
 
Let now $\theta\in {\rm IRS}(\Gamma)$ be an ergodic IRS. The map $H\mapsto {\rm Core}(H)$ is Borel in $\textrm{Sub}(\Gamma)$ and $\textrm{Core}(\gamma H \gamma^{-1} ) = \textrm{Core}(H)$, thus $\textrm{Core}(H)$ is fixed $\theta$-a.e. and we let the {\bf core} of $\theta$\index{core of an IRS} be this normal subgroup, denoted by $\textrm{Core}(\theta)$\index{$\textrm{Core}(\theta)$}. If this is trivial, we say that $\theta$ is {\bf core-free}\index{core-free IRS}. By \cite[Proposition 3.3.1]{CP}, there is $a\in \textrm{ERG}(\Gamma, X, \mu)$ such that $\textrm{type}(a) = \theta$. Then the kernel of $a$ is equal to $\textrm{Core}(\theta)$, thus $\theta$ is core-free iff $a$ is faithful iff the action of $\Gamma$ on $\Gamma/ \textrm{stab}_a(x)$ is faithful $\mu$-a.e. $(x)$.

 Recall now from \cref{space} the concept of the generalized shift $s_{H, \Gamma}$, where $H\leq\Gamma$. Let again $\theta\in {\rm IRS}(\Gamma)$ be an ergodic IRS. Then the map that sends $H$ to the weak equivalence class $\undertilde{s}_{H, \Gamma}$ of  $s_{H, \Gamma}$ is Borel from $\textrm{Sub}(\Gamma)$ to $\undertilde{A}(\Gamma, X, \mu)$ and since $s_{H, \Gamma} \cong s_{\gamma H \gamma^{-1}, \Gamma}$ and thus  $\undertilde{s}_{H, \Gamma} =  \undertilde{s}_{\gamma H \gamma^{-1}, \Gamma}$, it follows that $\undertilde{s}_{H, \Gamma}$ is fixed $\theta$-a.e. We call it the (weak equivalence class of the) {\bf generalized shift associated to} $\theta$\index{generalized shift associated to an IRS}, in symbols $\undertilde{s}_{\theta, \Gamma}$\index{$\undertilde{s}_{\theta, \Gamma}$}. Thus $\undertilde{s}_{\theta, \Gamma} = \undertilde{s}_{H, \Gamma}$ for the $\theta$-random $H$. Therefore there is a canonical weak equivalence class of a generalized shift associated to any ergodic IRS and therefore any ergodic action. In fact let $a\in \textrm{ERG}(\Gamma, X, \mu)$ and let $\theta = \textrm{type} (a)$. Put $\undertilde{s}_{a, \Gamma} = \undertilde{s}_{\theta, \Gamma}$. Note that, by \cite[Proposition 2.1]{KT}, $\undertilde{s}_{a, \Gamma}$ is ergodic.  By \cref{7.1}, if $a, b\in \textrm{ERG}(\Gamma, X, \mu)$ and $a\simeq b$, then $\textrm{type} (a) = 
 \textrm{type} (b)$. Therefore we can define $\undertilde{s}_{\undertilde{a}, \Gamma} = \undertilde{s}_{a, \Gamma}$, for each $\undertilde{a}\in \undertilde{\textrm{ERG}}(\Gamma, X, \mu)$ (the set of ergodic weak equivalence classes)\index{$\undertilde{\textrm{ERG}}(\Gamma, X, \mu)$}. Thus $\undertilde{a}\mapsto \undertilde{s}_{\undertilde{a}, \Gamma}$ gives a well-defined map on $\undertilde{\textrm{ERG}}(\Gamma, X, \mu)$.\index{$\undertilde{s}_{a, \Gamma}$}\index{$\undertilde{s}_{\undertilde{a}, \Gamma}$}

 We will consider next the connection of $\undertilde{s}_{\theta, \Gamma}$ with the weak equivalence class $\undertilde{s}_\Gamma$ of the shift.
 
 First let us consider the freeness of the $\undertilde{s}_{\theta, \Gamma}$. It is well-known, see, e.g., \cite[Proposition 2.4]{KT}, that $s_{H, \Gamma}$ is free iff $H$ is core-free, thus we have: 
 
 \begin{pro}
 For any ergodic {\rm IRS} $\theta$ on $\Gamma$, $\undertilde{s}_{\theta, \Gamma}$ is free iff $\theta$ is core-free.
 \end{pro}
 The next result characterizes exactly the relation, in terms of weak containment,  of $\undertilde{s}_{\theta, \Gamma}$ with $\undertilde{s}_\Gamma$. Below we say that an IRS $\theta$ is {\bf amenable}\index{amenable IRS}\index{non-amenable IRS} (resp., {\bf non-amenable}) if the $\theta$-random $H$ is amenable (resp., non-amenable). Also $\prec$\index{$\prec$} is the strict part of the order $\preceq$.
 
 \begin{thm}\label{irs1}
 Let $\theta$ be an ergodic {\rm IRS} on $\Gamma$. Then:
 
 \begin{enumerate}[\upshape (i)]
 
 \item If $\theta$ is amenable and not core-free, then $\undertilde{s}_{\theta, \Gamma}\prec \undertilde{s}_\Gamma$.
  \item If $\theta$ is amenable and core-free, then $\undertilde{s}_{\theta, \Gamma} =  \undertilde{s}_\Gamma$. 
 \item If $\theta$ is non-amenable and not core-free, then $\undertilde{s}_{\theta, \Gamma}$ is $\preceq$-incomparable with $\prec \undertilde{s}_\Gamma$. 
 \item If $\theta$ is non-amenable and core-free, then $ \undertilde{s}_\Gamma \prec\undertilde{s}_{\theta, \Gamma}$. 
 
 \end{enumerate}

 \end{thm}
 The proof will be given in Appendix G, \cref{appF}. It is easy to find examples of $\Gamma,\theta$ that fall under cases (i)-(iii) of \cref{irs1}. Case (iv) is the most interesting, as it provides canonical examples of generalized shifts that are strictly more complicated than the shift.
 
Suppose $a\in A(\Gamma. X, \mu)$ is a faithful, ergodic, hyperfinite action of a non-amenable group $\Gamma$. Then $\theta = \textrm{type}(a)$ is core-free and ergodic. It is also non-amenable, since otherwise the action $a$ would be amenable by \cite{Z1} and then $\Gamma$ would be amenable by \cite[Proposition 4.3.3]{Z2} (see also \cite{AEG}). Groups $\Gamma$ that have such actions are exactly those that can be embedded in $[E_0]$ in such a way that they generate $E_0$ (see also here \cref{hypact}). It is not clear exactly what groups have this property but it is known, for example, that it holds for the free non-abelian groups, see \cite[Page 29]{K}.

\begin{remark}
{\rm In \cite{DG} the authors show that if $a\in \textrm{ERG}(\Gamma, X, \mu)$ has $\textrm{type}(a) =\theta$, and for any $H\leq \Gamma$, we denote by $\lambda_{H,\Gamma}$\index{$\lambda_{H,\Gamma}$} the {\bf quasi-regular representation}\index{quasi-regular representation} of $\Gamma$ on $\ell^2 (\Gamma/H)$, then we have:

(1) $\lambda_{H, \Gamma}\preceq \kappa_0^a$, $\theta$-a.e.$(H)$.

(2) The weak equivalence class of $\lambda_{H, \Gamma}$ is constant $\theta$-a.e.$(H)$.

\noindent  (Note that (2) also follows from a simple ergodicity argument, since weak equivalence of representations is a smooth Borel equivalence relation.)

 (3) If $a$ is hyperfinite,  then $\lambda_{H, \Gamma}\simeq \kappa_0^a$, $\theta$-a.e.$(H)$.
 
 \noindent An analog of these results for generalized shifts would be as follows:
 
(1*) $s_{H, \Gamma}\preceq a$, $\theta$-a.e.$(H)$.

 (2*) The weak equivalence class of $s_{H, \Gamma}$ is constant $\theta$-a.e.$(H)$.
 
 (3*) If $a$ is hyperfinite,  then $s_{H, \Gamma}\simeq a$, $\theta$-a.e.$(H)$.
 
 As we have seen, (2*) is true but \cref{irs1} shows that  (1*) and (3*) fail.

}
\end{remark}

  \newpage
\section{Soficity and entropy}\label{10} 
 
\subsection{Sofic groups and entropy} In recent years there has been a lot of progress in developing the ergodic theory of sofic groups. Introduced by Gromov in \cite{Gr}, sofic groups are, intuitively, those groups which admit approximately free approximate actions on finite sets. 

\begin{define} Let $\Gamma$ be a countable group and for each natural number $n$, let $\sigma_n$ be a map from $\Gamma$ to $\mathrm{Sym}(V_n)$, where $V_n$ is a finite set. We say that $\Sigma = (\sigma_n)_{n=1}^\infty$ is a {\bf sofic approximation}\index{sofic approximation} to $\Gamma$ if for all $\gamma,\delta \in \Gamma$ we have \[ \lim_{n \to \infty} \frac{1}{|V_n|} \cdot |\{v \in V_n: \sigma_n(\gamma) \sigma_n(\delta)(v) = \sigma_n(\gamma\delta)(v) \}| = 1 \]
and for $\gamma \in \Gamma, \gamma\not= e_\Gamma$, we have \[ \lim_{n \to \infty} \frac{1}{|V_n|} \cdot |\{v \in V: \sigma_n(\gamma) (v)= v \}| = 0. \]
The group $\Gamma$ is {\bf sofic}\index{sofic group} if there exists a sofic approximation to $\Gamma$.
\end{define}

Every amenable group and every residually finite group is sofic. It is a major open problem to determine if every group is sofic.

In his paper \cite{Bo3}, Lewis Bowen introduced a notion of entropy for measure-preserving actions of sofic groups which admit a finite generating partition. The definition was extended to arbitrary measure preserving actions of sofic groups by Kerr in \cite{Ke1}. This sofic entropy is an isomorphism invariant, taking values in the extended real numbers, which generalizes the classical notion of entropy for actions of amenable groups. We refer the reader to Section $7$ of \cite{GS} for a concise definition. In addition to the action in question, sofic entropy can depend on a choice of sofic approximation to the acting group, although the nature and extent of this dependence is poorly understood.

\subsection{Completely positive entropy} \begin{define} Let $\Gamma$ be a sofic group and let $\Sigma$ be a sofic approximation to $\Gamma$. A measure-preserving action $a$ of $\Gamma$ is said to have {\bf completely positive entropy}\index{completely positive entropy} with respect to $\Sigma$ if every nontrivial factor of $a$ has positive entropy with respect to $\Sigma$. 
\end{define}

Bowen has raised the following question:

\begin{problem} \label{prob11.1} Suppose $a$ is a measure preserving action of a sofic group $\Gamma$ which has completely positive entropy with respect to some sofic approximation to $\Gamma$. Does it follow that $a$ is weakly equivalent to the Bernoulli shift of $\Gamma$? What if one assumes that $a$ has completely positive entropy with respect to every sofic approximation to $\Gamma$? \end{problem}

In \cite{Ke2}, Kerr shows that Bernoulli shifts of sofic groups have completely positive entropy with respect to any sofic approximation to the acting group. It follows immediately that factors of Bernoulli shifts have completely positive sofic entropy. In contrast to the amenable case, Popa \cite{Po} shows that for certain groups there are factors of Bernoulli shifts which are not isomorphic to Bernoulli shifts.





\subsection{Ultraproducts} Sofic entropy is, loosely speaking, the exponential growth rate of the number of labelings of the sofic approximation whose statistics replicate those of the relevant action. These replica labelings are known as {\it models}. In some cases, a sofic approximation $\Sigma$ may fail to admit arbitrarily precise models to an action $a$. In this case one defines the sofic entropy of $a$ with respect to $\Sigma$ to be $-\infty$. Given a sofic approximation $\Sigma$ to a group $\Gamma$, in \cite[Section 10.2]{CKT-D} the authors use ultraproducts to construct a canonical measure-preserving action of $\Gamma$ associated to $\Sigma$, which we will denote by $p_\Sigma$\index{$p_\Sigma$}. The following result is proved in Carderi \cite{C}.

\begin{theory}[{\cite[Proposition 3.6]{C}}]
 An action $a$ is weakly contained in $p_\Sigma$ iff the sofic entropy of $a$ with respect to $\Sigma$ is not $-\infty$.
\end{theory}

\subsection{Sofic actions} Elek and Lippner \cite{EL} have introduced the notion of a sofic measure preserving equivalence relation; see also \cite[Definition 10.1]{CKT-D} for an alternative definition due to Ozawa. One then defines a {\bf sofic action}\index{sofic action} to be one for which the associated equivalence relation $E_a$ is sofic. It is shown in \cite[Proposition 10.6]{CKT-D} that if $a_n\in A(\Gamma, X, \mu)$ are sofic actions and $a_n\to a\in {\rm FR}(\Gamma, X, \mu)$, then $a$ is sofic. Therefore if $a, b\in {\rm FR}(\Gamma, X, \mu), a\preceq b$ and $b$ is sofic, then $a$ is sofic.

These facts along with the ultraproduct construction of $p_\Sigma$ (see Section 11.3 above) and \cref{min} were used in \cite[Theorem 10.7]{CKT-D} to give a new proof of the result in \cite{EL} that, for any sofic group $\Gamma$, $s_\Gamma$ is a sofic action. 

It is an open problem whether for a sofic group $\Gamma$, every $a\in {\rm FR}(\Gamma, X, \mu)$ is sofic. It is shown in \cite[Section 10.3]{CKT-D} that it holds for $\Gamma$ with property MD and this also gives an alternative proof of the result of \cite{EL} that every treeable measure preserving equivalence relation is sofic.
 
\subsection{Rokhlin entropy} The recent paper \cite{Se} introduces a new concept of weak containment for joinings and uses it to study the Rokhlin entropy\index{Rokhlin entropy} of measure preserving actions.

\newpage
\section{Further extensions}\label{add1}
 
\subsection{Locally compact groups} Bowen and Tucker-Drob have studied weak containment and the spaces of weak equivalence classes and stable weak equivalence classes in the more general context of Polish locally compact groups.

\medskip
 \subsection{Stationary actions} In the paper \cite{BLT}, the authors consider weak containment and weak equivalence for stationary actions of countable groups in connection with the Furstenberg entropy.
 
 \medskip
 \subsection{Graphed groupoids} A notion of weak containment for graphed groupoids, which generalizes weak containment of actions, is introduced in \cite[Section 3.4]{CGdlS}. It relates also to the local-global convergence of graphs as in \cite{HLS}, see \cite[Remark 3.21]{CGdlS}. An analog of \cref{6.1} in this context is proved in \cite[Theorem 3.22]{CGdlS} and connections with factors, ultraproducts and soficity are studied in \cite[Sections 3.5 and 3.6]{CGdlS}.

 \newpage
 
 \section{Appendix A}\label{app} We give here the proof of \cref{AE1}.
 
  \begin{proof}
  Let $a_0, a_1, \dots\in A(\Gamma, X, \mu)$.  By a simple diagonal argument, we can assume, by going to a subsequence, that for each $n,k$, $C_{n,k}(a_i)$ converges in the Vietoris topology. We will then show that $a_0, a_1, \dots$ converges in the pseudometric $\delta$.
  
  Let $\mathcal{U}$ be a non-principal ultrafilter on $\bbN$ and let $b= \prod_n a_n/\mathcal{U}$ be the ultraproduct action, with associated measure $\nu$. For the action $b$, we define $M_{n,k}^{\bar{A}} (b)$ and $C_{n,k}(b)$ exactly as in \cref{22}, (1). Then it is easy to check that for each $n,k$,
\[
\lim_{i\to\infty}C_{n,k}(a_i) = C_{n,k} (b).
\]
Let now $M_{n,k}^{\bar{A}_{j,n,k}} (b), j\in \bbN, $ be a sequence dense in $C_{n,k} (b)$, for each $n,k$, and let $\mathcal{S}$ be a countably generated, non-atomic, $\sigma$-subalgebra of the measure algebra of $\nu$, invariant under the action $b$ and such that it contains all the pieces of the partitions $ \bar{A}_{j,n,k}$, for $ j,n,k\in \bbN$. Let $a\sqsubseteq b$ be the factor of $b$ corresponding to $\mathcal{S}$, which (up to isomorphism) we can assume that it belongs to $A(\Gamma, X, \mu)$. Since 
\[
C_{n,k} (a)= C_{n,k} (b),
\]
for all $n,k$, we have 
\[
\lim_{i\to\infty}C_{n,k}(a_i) = C_{n,k} (a),
\]
for all $n,k$ and therefore 
\[
\delta (a_i, a)\to 0,
\]
i.e.,  $a_0, a_1, \dots$ converges to $a$ in the pseudometric $\delta$.
\end{proof}

  \newpage
\section{Appendix B}\label{appA} We give here the proof of \cref{min}.
 
 \begin{proof} To clarify the exposition, we will prove this theorem for the shift action $s$ of $\Gamma$ on the product space $2^\Gamma$ with the product measure, where the two point space has the measure $\mu_0$ in which each point has measure $\frac{1}{2}$. We will write $\nu$ for $\mu_0^\Gamma$. The case of the shift on $[0,1]^\Gamma$ can be proved with minor modifications.

For an action $b\in A(\Gamma,Y,\rho)$, a Borel partition $\mathcal{P} = \{P_1,\ldots,P_k\}$ of $Y$, a finite set $F \subseteq \Gamma$ and a function $\tau:F \to \{1,\ldots,k\}$, we define $P^b_\tau = \bigcap_{\gamma \in F} \gamma^b (P_{\tau(\gamma)})$.

Let $a\in {\rm FR}(\Gamma, X, \mu)$. Let $F_0\subseteq \Gamma$ be finite, $\epsilon >0$ and $\{A_1, \dots , A_n \}$ be a Borel partition of $2^\Gamma$. We will find Borel sets $B_1, \dots , B_n$ in $X$ such that for $\gamma\in F_0, 1\leq m,m'\leq n$, we have 
\[
|\mu(\gamma^a (B_m) \cap B_{m'} ) - \nu (\gamma^s (A_m) \cap A_{m'})| < \epsilon.
\]

For $i \in \{0,1\}$, let $S_i = \{x \in 2^\Gamma: x(e_\Gamma) = i \}$ and $\mathcal{S} = \{S_0,S_1\}$. Note that $\gamma^s (S_i) = \{x\in 2^\Gamma\colon x(\gamma) =i\}$. Then we can find a finite set $F_1 \subseteq \Gamma$ and for each $m \in \{1,\ldots,n\}$ a family of (distinct) functions $(\tau^m_j)_{j=1}^{t_m}$ with $\tau^m_j: F_1 \to 2$ such that
\begin{equation} \label{eq6.1} \nu \left ( \left( \bigsqcup_{j=1}^{t_m} S^s_{\tau^m_j} \right) \triangle A_m \right) < \frac{\epsilon}{4}. \end{equation}

\begin{claim}\label{cla1} For any $\delta > 0$, there is a partition $\mathcal{Q} = \{Q_0,Q_1\}$ of $(X,\mu)$ such that for every $\theta: J \to 2$ with $J \subseteq F_0 F_1$, we have $|\nu(S^s_\theta) - \mu(Q^a_\theta)| < \delta$.  \end{claim}

Before establishing this claim, we will show how it completes the proof. For $J \subseteq F_1$, $\theta: J \to 2$ and $\gamma \in F_0$ define the shifted function $\gamma \cdot \theta: \gamma J \to 2$ by $\gamma \cdot \theta(\delta) = \theta(\gamma^{-1} \delta)$. Note that the domain of $\gamma \cdot \theta$ is contained in  $F_0 F_1$. Also note that for any $\theta$ as above, $\kappa: K \to 2$, $K\subseteq F_1$, and $\gamma \in F_0$, we have \[ \gamma^s (S^s_{\theta}) \cap S^s_{\kappa} = \begin{cases} S^s_{\gamma \cdot \theta \cup \kappa} &\mbox{ if }\gamma \cdot \theta \mbox{ and } \kappa \mbox{ are compatible},\\ \emptyset &\mbox{ if not,} \end{cases} \] and similarly \[ \gamma^a (Q^a_{\theta}) \cap Q^a_{\kappa} = \begin{cases} Q^a_{\gamma \cdot \theta \cup \kappa} &\mbox{ if }\gamma \cdot \theta \mbox{ and } \kappa \mbox{ are compatible},\\ \emptyset &\mbox{ if not.} \end{cases} \]

Hence by Claim \ref{cla1}, we have \begin{equation} \label{eq6.2} |\nu(\gamma^s (S^s_\theta) \cap S^s_\kappa) - \mu(\gamma^a (Q^a_{\theta}) \cap Q^a_{\kappa})| <  \delta \end{equation} for all $\theta,\kappa$ as above and $\gamma \in F_0$. For $m \in \{1,\ldots,n\}$, set $B_m = \bigsqcup_{j=1}^{t_m} Q^a_{\tau^m_j}$. Write $C_m = \bigsqcup_{j=1}^{t_m} S^s_{\tau^m_j}$. By (\ref{eq6.1}) and (\ref{eq6.2}) we have for any $m,m' \in \{1,\ldots,n\}$,

\[
 |\mu(\gamma^a (B_m) \cap B_{m'}) - \nu(\gamma^s (A_m) \cap A_{m'})| 
 \]
 \begin{multline*}
 \leq |\mu(\gamma^a (B_m) \cap B_{m'}) - \nu(\gamma^s (C_m) \cap C_{m'})| \\ + |\nu(\gamma^s (C_m) \cap C_{m'}) - \nu(\gamma^s (A_m) \cap A_{m'})|
 \end{multline*}
  \begin{multline*}
    \leq |\mu(\gamma^a (B_m) \cap B_{m'}) - \nu(\gamma^s (C_m) \cap C_{m'})| \\ + \nu((\gamma^s (C_m) \cap C_{m'}) \triangle (\gamma^s (A_m) \cap A_{m'}))
  \end{multline*}
  \begin{multline*}
  \leq |\mu(\gamma^a (B_m) \cap B_{m'}) - \nu(\gamma^s (C_m) \cap C_{m'})| \\+ \nu( C_m \triangle A_m) + \mu(C_{m'} \triangle A_{m'})
  \end{multline*}
\[
\leq \left \vert \nu \left( \bigsqcup_{j=1}^{t_m} \bigsqcup_{j'=1}^{t_{m'}}\gamma^s (S^s_{\tau^m_j}) \cap S^s_{\tau^{m'}_{j'}} \right) - \mu\left( \bigsqcup_{j=1}^{t_m}\bigsqcup_{i'=1}^{t_{m'}} \gamma^a (Q^a_{\tau^m_j}) \cap Q^a_{\tau^{m'}_{j'}} \right) \right \vert + \frac{\epsilon}{2}
\]
\[
\leq \left( \sum _{j= 1}^{t_m}\sum_{j'=1}^{t_{m'}}  \left \vert \nu \left(\gamma^s (S^s_{\tau^m_j}) \cap S^s_{\tau^{m'}_{j'}} \right) - \mu\left(\gamma^a (Q^a_{\tau^m_j}) \cap Q^a_{\tau^{m'}_{j'}} \right) \right \vert \right) + \frac{\epsilon}{2}
\]
 \[
 \leq t_m t_{m'} \delta + \frac{\epsilon}{2}.
 \] 

Therefore if we take  taking $\delta$ small enough in  \cref{cla1}, we are done.

\medskip
{\it Prof of \cref{cla1}.} Write $G = F_0F_1$ and assume without loss of generality that $G$ is closed under taking inverses. Note that it suffices to prove the claim for $\theta$ defined on all of $G$. Since distinct shifts of the sets $S_i$ are independent, we have $\nu(S^s_\theta) = 2^{-|G|}$, for any $\theta: G \to 2$. Thus we must find a partition $\mathcal{Q} = \{Q_0,Q_1\}$ such that \[ \left \vert 2^{-|G|} - \mu(Q^a_\theta) \right \vert < \delta, \] for each $\theta: G \to 2$. The idea is that a random $\mathcal{Q}$ should have this property.

Without loss of generality, we may assume that $X$ is a compact metric space with a compatible metric $d$. For $\eta > 0$, let 
\[
D_\eta = \{x \in X: \forall  \gamma_1, \gamma_2 \in G(\gamma_1 \neq \gamma_2 \implies d(\gamma_1^a(x), \gamma_2^a (x)) > \eta) \}
\] 
and \[ E_\eta = \{(x,x') \in D_\eta^2: \forall  \gamma_1, \gamma_2 \in G( d(\gamma_1^a (x), \gamma_2^a (x')) )> \eta)\}.
 \]

\begin{lem}  $\lim_{\eta \to 0} \mu(D_\eta) = 1$ and $\lim_{\eta \to 0} \mu^2(E_\eta) = 1$. \end{lem}

\begin{proof} Clearly if $\eta_1 < \eta_2$, then $D_{\eta_2} \subseteq D_{\eta_1}$ and $E_{\eta_2}\subseteq E_{\eta_1}$. We have \[ X \setminus \bigcup_{\eta > 0} D_\eta = \{ x \in X: \exists\gamma_1 \neq \gamma_2 \in G ( \gamma_1^a (x) = \gamma_2^a (x)) \}\]  and so, by the freeness of the action,  $\mu \left( X \setminus \bigcup_{ \eta > 0} D_\eta \right) = 0$. Now for any $\eta_0 > 0$, \[ D_{\eta_0}^2 \setminus \bigcup_{\eta > 0} E_\eta = \{(x,x') \in D_{\eta_0}^2: \exists\gamma_1,\gamma_2 \in G (\gamma_1^a (x) = \gamma_2^a (x')) \}.\]
For a fixed $x$, \[\{(x,x') \in D_{\eta_0}^2: \exists\gamma_1,\gamma_2 \in G ( \gamma_1^a (x) = \gamma_2^a (x'))\} \] is finite, so $\mu \left( D_{\eta_0}^2 \setminus \bigcup_{\eta > 0} E_\eta \right)=0$ by Fubini and taking the union over all $\eta_0$, we have that $\mu(X^2\setminus\bigcup_{\eta >0}E_\eta) =0$.  \end{proof}   

Let \begin{equation} \label{eq5.10} \delta' = \frac{\delta^2}{2^{|G|+5}}.\end{equation} Choose $\eta > 0$ so that $\min(\mu(D_\eta),\mu^2(E_\eta)) > 1 - \delta'$. Let $\{Y_1,\ldots,Y_m\}$ be a partition of $X$ into Borel pieces with diameter at most $\frac{\eta}{4}$. For $x \in X$, let $Y(x)$ be the unique $l \in \{1,\ldots,m\}$ such that $x \in Y_l$. Let $\mathbb{P}$ be the uniform (= product) probability measure on $2^m$ and for each $\omega \in 2^m$ define a partition $Z(\omega) = \{Z(\omega)_0,Z(\omega)_1\}$ by letting $x \in Z(\omega)_i$ if and only if $\omega(Y(x)) = i$. Thus we have a random variable $Z: (2^m,\mathbb{P}) \to \mathrm{MALG}_\mu^2 $ given by $\omega \mapsto Z(\omega)$. 

Fix now $\tau:G \to 2$. Then we have a corresponding real-valued random variable on $(2^m,\mathbb{P})$ given by $\omega \mapsto \mu(Z(\omega)^a_\tau)$. We now compute the expectation of this random variable. For $B \subseteq X$, let $\mathbf{1}_B$ be the characteristic function of $B$. Then
\begin{align} \mathbb{E}[\mu(Z^a_\tau)] &= \int_{2^m} \mu(Z(\omega)^a_\tau) \mathrm{d} \mathbb{P}(\omega) \nonumber \\  & = \int_{2^m} \int_X \mathbf{1}_{Z(\omega)^a_\tau}(x) \dee\mu(x) \dee\mathbb{P}(\omega) \nonumber \\ &= \int_X \int_{2^m} \mathbf{1}_{Z(\omega)^a_\tau}(x) \dee\mathbb{P}(\omega) \dee\mu(x) \nonumber \\ & = \int_{D_\eta} \int_{2^m} \mathbf{1}_{Z(\omega)^a_\tau}(x) \dee\mathbb{P}(\omega) \dee\mu(x) + \int_{X \setminus D_\eta} \int_{2^m} \mathbf{1}_{Z(\omega)^a_\tau}(x) \dee\mathbb{P}(\omega) \dee\mu(x). \nonumber \end{align}

Now if $x \in D_\eta$, then for all $\gamma_1 \neq \gamma_2 \in G$, we have $d(\gamma_1^a (x), \gamma_2^a (x)) > \eta$, so that $Y(\gamma_1^a( x)) \neq Y(\gamma_2^a (x))$ and hence the events $\omega(Y(\gamma_1^a (x))) = i$ and $\omega(Y(\gamma_2^a (x))) = j$ are independent. We have $x \in \gamma^a (Z(\omega)_{\tau(\gamma)})$ if and only if $\omega(Y((\gamma^{-1})^a (x))) = \tau(\gamma)$, so if $x \in D_\eta$ and $\gamma_1 \neq \gamma_2 \in G$ the events $\omega(Y((\gamma_1^{-1})^a (x))) = \tau(\gamma)$ and $\omega(Y((\gamma_2^{-1})^a (x))) = \tau(\gamma)$ are independent. So for $x \in D_\eta$,

\begin{align} \int_{2^m} \mathbf{1}_{Z(\omega)^a_\tau}(x) \dee \mathbb{P}(\omega) &= \mathbb{P} (\{\omega: x \in \gamma^a (Z(\omega)_{\tau(\gamma)}), \forall \gamma \in G \}) \nonumber \\& = \mathbb{P} \left (\bigcap_{\gamma \in G} \left\{ \omega: \omega(Y((\gamma^{-1})^a (x))) = \tau(\gamma) \right\} \right) \nonumber  \\ & = \prod_{\gamma \in G}\mathbb{P} \left ( \left\{ \omega: \omega(Y((\gamma^{-1})^a (x))) = \tau(\gamma) \right\} \right) \nonumber \\& = 2^{-|G|} .\label{eq4.6} \end{align}

Since $\mu(X \setminus D_\eta) < \delta'$, we have \begin{equation} \label{eq4.8} \left \vert \mathbb{E}[\mu(Z^a_\tau)] - 2^{-|G|} \right \vert < \delta'. \end{equation} We now compute the second moment of $\mu(Z^a_\tau)$, in order to estimate its variance:

\begin{align} \mathbb{E}\left[\mu(Z^a_\tau)^2 \right] & = \int_{2^m} \mu(Z^a_\tau(\omega))^2 \dee\mathbb{P} (\omega) \nonumber \\ & = \int_{2^m} \left( \int_X \mathbf{1}_{Z^a_\tau(\omega)}(x) \dee\mu(x) \right)^2 \dee \mathbb{P}(\omega) \nonumber \\ & = \int_{2^m} \int_{X^2} \mathbf{1}_{Z^a_\tau(\omega)}(x_1) \mathbf{1}_{Z^a_\tau(\omega)}(x_2) \dee\mu^2(x_1,x_2) \dee\mathbb{P}(\omega)\nonumber \\ & = \int_{X^2} \int_{2^m} \mathbf{1}_{Z^a_\tau(\omega)}(x_1) \mathbf{1}_{Z^a_{\tau}(\omega)}(x_2) \dee\mathbb{P} (\omega) \dee\mu^2(x_1,x_2) \nonumber \\ & =  \int_{E_\eta} \int_{2^m} \mathbf{1}_{Z^a_\tau(\omega)}(x_1) \mathbf{1}_{Z^a_{\tau}(\omega)}(x_2) \dee\mathbb{P} (\omega) \dee\mu^2(x_1,x_2) \nonumber \\ & \ \ \  \ + \int_{X^2 \setminus E_\eta} \int_{2^m} \mathbf{1}_{Z^a_\tau(\omega)}(x_1) \mathbf{1}_{Z^a_{\tau}(\omega)}(x_2) \dee\mathbb{P} (\omega) \dee\mu^2(x_1,x_2). \label{eq4.7} \end{align}

Now if $(x_1,x_2) \in E_\eta$, then for any $\gamma_1,\gamma_2 \in G$ we have $d(\gamma^a_1 (x_1),\gamma^a_2 (x_2)) > \eta$, so that $Y(\gamma^a_1 (x_1)) \neq Y(\gamma^a_2 (x_2))$ and thus for a fixed pair $(x_1,x_2)$ the events $\omega(Y(\gamma^{-1})^a (x_1)) = \tau(\gamma)$, for all $\gamma \in G$, and $\omega(Y(\gamma^{-1})^a (x_2)) = \tau(\gamma)$, for all $\gamma \in G$, are independent. Hence for a fixed $(x_1,x_2) \in E_\eta$, we have
\begin{multline*}
 \int_{2^m} \mathbf{1}_{Z^a_\tau(\omega)}(x_1) \mathbf{1}_{Z^a_\tau(\omega)}(x_2) \dee\mathbb{P}(\omega)\\
 = \mathbb{P}(\{ \omega: x_1 \in \gamma^a (Z(\omega)_{\tau(\gamma)}) \ {\rm and} \ x_2 \in \gamma^a (Z(\omega)_{\tau(\gamma)}), \forall\gamma \in G \})\\
  = \mathbb{P}(\{ \omega: \omega(Y((\gamma^{-1})^a (x_1)) = \tau(\gamma) \  {\rm and} \  \omega(Y((\gamma^{-1})^a(x_2)) = \tau(\gamma), \forall\gamma \in G \})\\
   =\mathbb{P} \left ( \left\{ \omega: \omega(Y((\gamma^{-1})^a (x_1))) = \tau(\gamma), \forall\gamma \in G \right\} \right) \\ \cdot
   \mathbb{P} \left ( \left\{ \omega: \omega(Y((\gamma^{-1})^a (x_2))) = \tau(\gamma), \forall \gamma \in G \right\} \right)
    = 2^{-2|G|},
  \end{multline*}
  by $(\ref{eq4.6})$ and the fact that $E_\eta \subseteq D_\eta^2$. Since $\mu^2(X \setminus E_\eta) < \delta'$, we see that
  \[
   \left(1 - \delta' \right) 2^{-2|G|} \leq (\ref{eq4.7}) \leq 2^{-2|G|} + \delta' 
   \]
    and hence \[ \left \vert \mathbb{E}[\mu(Z_\tau^a)^2] - 2^{-2|G|} \right \vert < \delta'. \] Therefore (assuming $\delta <1$), 
\[
 \mathrm{Var}(\mu(Z^a_\tau ))
 = \mathbb{E}[\mu(Z^a_\tau)^2] - \mathbb{E}[\mu(Z^a_\tau)]^2
 \]
 \[
  \leq \delta' + 2^{-2|G|} - \left(-  \delta' + 2^{-|G|} \right) ^2
  \]
  \[
   = \delta' - (\delta')^2 + 2^{-|G|+1} \delta'
   \leq 3 \delta' \leq \frac{\delta^2}{2^{|G|+3}},
   \]
where in the last step we recall the definition of $\delta'$ from (\ref{eq5.10}). Therefore Chebyshev's inequality for $\mu(Z^a_\tau)$ gives

\begin{align*}\mathbb{P} \left( \left\{ \omega: |\mu(Z^a_\tau(\omega)) - \mathbb{E}[\mu(Z^a_\tau)] | \geq \frac{\delta}{2} \right \} \right) &\leq \frac{\mathrm{Var}(\mu(Z^a_\tau))}{\left(\frac{\delta}{2}\right)^2} \\ & \leq \frac{1}{2^{|G|+1} }   \end{align*}

By (\ref{eq4.8}) we have \[ \left \vert \mathbb{E}[\mu(Z^a_\tau)] - 2^{-|G|} \right \vert < \frac{\delta}{2}, \] so that \[\mathbb{P} \left( \left\{ \omega: \left \vert \mu(Z^a_\tau(\omega)) - 2^{-|G|} \right \vert \geq \delta \right \} \right) \leq  \frac{1}{2^{|G| +1} }.\]

Since this is true for each $\tau \in 2^G$, we have 

 \[\mathbb{P} \left( \left\{ \omega:\left \vert \mu(Z^a_\tau(\omega)) - 2^{-|G|}\right \vert \geq \delta, \  \textrm{for some} \ \tau: G \to 2  \right \} \right ) \leq  \frac{1}{2}.\]

Thus any member of the nonempty complement of this set works as $\mathcal{Q}$.\end{proof}

  \newpage
\section{Appendix C}\label{appB}
We give here a proof of \cref{7.1}.
 
 \begin{proof} For each finite set $F\subseteq \Gamma$, let $B_F = \{H \leq \Gamma: F \subseteq H\}$, and $\b =\{B_F\colon F \textrm{ a finite subset of }\Gamma\}$. Then $\b$ is closed under finite intersections and generates the Borel subsets of $\mathrm{Sub}(\Gamma)$. Thus by the $\pi-\lambda$ Theorem (see, \cite[Theorem 10.1]{K2}), two probability Borel measures on $\mathrm{Sub}(\Gamma)$ are equal iff they agree on $\b$. It is thus enough to show the following, for $a\in A(\Gamma, X, \mu), b\in A(\Gamma, Y, \nu)$:
 \[ 
 a\preceq b\implies {\rm type}(a)(B_F)\geq {\rm type}(b)(B_F).
   \]
   So assume that  $a\preceq b$. Then by \cref{2.6}, $a\sqsubseteq b_\mathcal{U}$, where $\mathcal{U}$ is a non-principal ultrafilter on $\bbN$. 
   
   Recall that ${\rm Fix}_a(\gamma) = \{x\in X\colon \gamma^a (x) = x\}$ and similarly for $b, b_\mathcal{U}$. Then by the definition of the type function, ${\rm type}(a) (B_F) = \mu (\bigcap_{\gamma\in F} {\rm Fix}_a (\gamma))$ and similarly for $b, b_\mathcal{U}$. If $(Y_\mathcal{U}, \nu_\mathcal{U})$ is the ultrapower of $(Y,\nu)$ (on which $b_\mathcal{U}$ acts) and $f\colon  Y_\mathcal{U} \to X$ is a homomorphism of $b_\mathcal{U}$ to $a$, then 
   \[
   f^{-1}( \bigcap_{\gamma\in F} {\rm Fix}_a (\gamma))  \supseteq   \bigcap_{\gamma\in F} {\rm Fix}_{b_\mathcal{U}}
   (\gamma),
   \]  
   so $\mu(\bigcap_{\gamma\in F} {\rm Fix}_a (\gamma)) \geq \nu_{  \mathcal{U}}( \bigcap_{\gamma\in F} {\rm Fix}_{b_\mathcal{U}} (\gamma))$. Now it is easy to check that (in the notation of \cite[Section 3.1]{CKT-D}), ${\rm Fix}_{b_\mathcal{U}} (\gamma) = [({\rm Fix}_b (\gamma))]_{\mathcal{U}}$, therefore $ \nu_{  \mathcal{U}}( \bigcap_{\gamma\in F} {\rm Fix}_{b_\mathcal{U}} (\gamma)) =  \nu(\bigcap_{\gamma\in F} {\rm Fix}_b (\gamma))$, thus 
   
  \[  \mu(\bigcap_{\gamma\in F} {\rm Fix}_a (\gamma))   \geq \nu(\bigcap_{\gamma\in F} {\rm Fix}_b (\gamma))
  \]
   and the proof is complete. \end{proof}
 
 \newpage
\section{Appendix D}\label{appD} We prove here some of the facts stated in \cref{10.4}. It follows from \cref{hypact} (ix)  that every approximately amenable group is sofic. We give below a more direct proof.

\begin{proof} Let $\Gamma$ be approximately amenable. Recall that every Borel hyperfinite, measure preserving, equivalence relation $E$ can be written as $E=\bigcup_n E_n$, where each E$_n$ has finite classes of the same cardinality and $E_n\subseteq E_{n+1}$ for each $n$.

Fix now finite $F=\{\gamma_1, \dots , \gamma_n\}\subseteq \Gamma$ and $\epsilon >0$, in order to find some $m$ and a map $\varphi\colon F\to S_m$ (the symmetric group on $m$ elements) such that if $e_\Gamma\in F$, then $\varphi(e_\Gamma) = id$, for any $\gamma\in F\setminus\{e_\Gamma\}$, $\sigma(\{x\colon \varphi(\gamma)(x) = x\}) <\epsilon$, where $\sigma$ is the normalized counting measure, and if $\gamma, \delta, \gamma\delta\in F$, then $\sigma(\{x\colon \varphi(\gamma)\varphi(\delta)(x)\not=\varphi(\gamma\delta)(x)\})<\epsilon$.

Take $0<\delta<\frac{\epsilon}{8n}$ and note that there is hyperfinite $a\in A(\Gamma, X, \mu)$ such that $\mu ({\rm Fix}_{a}(\gamma))< \delta$, for every $\gamma\not= e_\Gamma \in F$. To see this, find for each such $\gamma$ a hyperfinite action $a_\gamma\in A(\Gamma, X, \mu)$ with $\mu ({\rm Fix}_{a_\gamma}(\gamma))< \delta$ and take $a$ to be the product of these $a_\gamma$.

Write then $E_a$ as the union of an increasing sequence $(E_n)$ as above, where all the classes of $E_n$ have size $m_n$. If $[E_n]$ is the full group of $E_n$ and $d_u$ is the uniform metric on ${\rm Aut}(X,\mu)$, given by $d_u(S,T) = \mu(\{x\colon S(x) \not= T(x)\})$, then there is large enough $n$, so that for each $\gamma\in F$, there is $T_\gamma\in [E_n]$ with $d_u(T_\gamma, \gamma^a)<\delta$. Put $R=E_n, m=m_n$.

Then for any Borel set $A\subseteq X$, we have
\[
\mu(A) = \int\frac{|A\cap[x]_R|}{m} \ {\rm d}\mu(x)
\]
(apply \cite[Sublemma 10.6]{KM} to $m$ distinct transversals of $R$). Therefore if $\gamma\in F\setminus\{e_\Gamma\}$, then 
\[
\int\frac{|{\rm Fix}_a(\gamma)\cap[x]_R|}{m} \ {\rm d}\mu(x) <\delta
\]
 and for $\gamma\in F$,
\[
\int\frac{|\{x\colon\gamma^a(x)\not= T_\gamma(x)\}\cap[x]_R|}{m} \ {\rm d}\mu(x) <\delta.
\]
Take now $\alpha= \frac{\epsilon}{4}$ and note that by Markov's inequality
\[
\gamma\in F\setminus\{e_\Gamma\} \implies \mu(\{x\colon \frac{|{\rm Fix}_a(\gamma)\cap[x]_R|}{m}\geq \alpha\}) <\frac{\delta}{\alpha}
\]
and
\[
\gamma\in F\implies \mu(\{x\colon \frac{|\{x\colon\gamma^a(x)\not= T_\gamma(x)\}\cap[x]_R|}{m}\geq \alpha\}) <\frac{\delta}{\alpha}.
\]
so, since $\frac{2n\delta}{\alpha} < 1$, there is some $x$ such that letting $ C= [x]_R$, we have 
\[
\gamma\in F\setminus\{e_\Gamma\} \implies \frac{|{\rm Fix}_a(\gamma)\cap C|}{m} <  \alpha
\]
and 
\[
\gamma\in F\implies \frac{|\{x\colon\gamma^a(x)\not= T_\gamma(x)\}\cap C|}{m} <\alpha.
\]
Identify now $C$ with $\{1,\dots , m\}$ and define $\varphi\colon F\to S_m$ by $\varphi (e_\Gamma) = id$, if $e_\Gamma\in F$, and $\varphi (\gamma) = T_\gamma|C$, if $\gamma\in F\setminus\{e_\Gamma\} $. Then 
\[
\gamma\in F\setminus\{e_\Gamma\} \implies \sigma(\{x\colon\gamma^a(x) = x\})< \alpha 
\]
and 
\[
\gamma\in F\implies \sigma(\{x\colon\gamma^a(x) \not= T_\gamma(x)\})< \alpha.
\]
Therefore
\[
\gamma\in F\setminus\{e_\Gamma\} \implies \sigma(\{x\colon \varphi(\gamma)(x) = x\})< 2\alpha <\epsilon
\]
and 
\[
\gamma, \delta, \gamma\delta \in F \implies \sigma(\{x\colon\varphi(\gamma)\varphi(\delta)(x) \not= \varphi(\gamma\delta)(x)\})< 3\alpha <\epsilon
\]
and the proof is complete.
\end{proof}

Next we prove \cref{hypact}.

\begin{proof}
The equivalence of (i) and (ii) follows from \cite[Proposition 4.2]{CKT-D}. That (ii) implies (iii) follows from the fact that there is $ a\in {\rm FR}(\Gamma, X, \mu)$ with $a \sqsubseteq\prod_n a_n / {\cal U}$, see \cite[page 345]{CKT-D}. Then, by \cite[Theorem 5.3]{CKT-D} and \cref{min}, $s_\Gamma\preceq a\preceq_{\cal U}(a_n)$, so $s_\Gamma\preceq_{\cal U} (a_n)$, so again by \cite[Theorem 5.3]{CKT-D}, $s_\Gamma\sqsubseteq \prod_n a_n / {\cal U}$. To see that (iii) implies (ii), note that if $\prod_n a_n / {\cal U}$ is not free and $s_\Gamma \sqsubseteq\prod_n a_n / {\cal U}$, then there is non-free $a\in A(\Gamma, X, \mu)$ with $s_\Gamma\sqsubseteq a \sqsubseteq\prod_n a_n / {\cal U}$, contradicting \cref{fre}. The equivalence of  (iii) and (iv) follows from \cite[Theorem 5.3]{CKT-D}.

Clearly (i) implies (v). We next show that (v) implies (iv). We have that the weak equivalence class of $\prod_n a_n$ is the limit of the weak equivalence classes of $\prod_{i =0}^n a_n$ (see the penultimate paragraph of \cref{po}). So by \cref{1050} there is a sequence $n_0 < n_1 < \cdots$ and $b_{n_i} \in {\rm HYP}(\Gamma, X, \mu)$ such that $b_{n_i}\to \prod_n a_n$. Since $s_\Gamma\preceq \prod_n a_n$, (iv) follows.

That (i) is equivalent to (vi), follows as in the third paragraph of the preceding proof in this Appendix, using finite products of actions. To see that (vi) implies (vii), fix $\gamma_1, \dots , \gamma_n\in \Gamma\setminus\{e_\Gamma\}$ and $\epsilon > 0$ and let $a$ be as in (vi) for $\frac{\epsilon}{n}$. Consider the ergodic decomposition of $a$, as in \cite[Theorem 3.3]{KM} (where we take $E= E_a$), whose notation we use below. Let $\pi_* \mu = \nu$. Then for $i = 1, \dots , n$ we have 
\[
\epsilon\cdot\nu(\{e\colon e ({\rm Fix}_{a|X_e}(\gamma_i)) \geq \epsilon\})\leq  \int e({\rm Fix}_{a|X_e}(\gamma_i)) \ {\rm d}\nu(e) = \mu( {\rm Fix}_a (\gamma_i)) <\frac{\epsilon}{n},
\]
so there is $e$ such that for every $i = 1, \dots , n$, we have $e ({\rm Fix}_{a|X_e}(\gamma_i)) < \epsilon$, so $a|X_e\in A(\Gamma, X_e, e)$ verifies (vii). That (vii) implies (vi) follows from the observation that if $a\in A(\Gamma, Y, \nu )$, with $Y$ finite, and $b=i_\Gamma \times a$, then ${\rm Fix}_b (\gamma) = {\rm Fix}_a (\gamma), \forall \gamma\in \Gamma$, and the action $b$ is on a non-atomic space. Clearly (i) implies (viii). To see that (viii) implies (vi), fix $\gamma_1, \dots , \gamma_n\in \Gamma\setminus\{e_\Gamma\}$ and $\epsilon >0$. Then there are hyperfinite actions $a_1, \dots , a_n$ such that $\gamma_i^{a_i}\not= id, i = 1, \dots , n$. Consider $b = a_1\times \cdots \times a_n$. Then $\gamma_i ^b \not= id, i=1, \dots , n$. For large enough $m$, if $a= b^m$, we have that $\mu ({\rm Fix}_a (\gamma_i))<\epsilon, i = 1, \dots , n$, and $a$ is hyperfinite.

To see that (ix) implies (viii), note that an embedding of $\Gamma$ into $[E_0]$ is an action $a\in A(\Gamma, X, \mu)$ with $E_a\subseteq E_0$ (therefore $a$ is hyperfinite) and $\gamma^a\not=id, \forall \gamma\in\Gamma\setminus \{ e_\Gamma\}$. We next show that (i) implies (ix). By \cite[Proposition 4.13]{K}, it is enough to find for each $\gamma\in\Gamma\setminus \{ e_\Gamma\}$ an action $a$ such that $E_a\subseteq E_0$ and $\gamma^a\not= id$. By (i), there is a hyperfinite action $a$ such that $\gamma^a\not=id$. So it is enough to find a measure preserving, ergodic, hyperfinite $E$ such that $E_a\subseteq E$. By \cite[Lemma 5.4]{K}, it is enough to show that $E_a\subseteq F$, for some measure preserving, aperiodic (i.e., having infinite classes), hyperfinite $F$. By separating the space $X$ into the two $E_a$-invariant Borel sets on which $E_a$ has finite, resp., infinite classes, it is enough to deal with the case when $E_a$ has finite classes. In this case, let $Y$ be a Borel set meeting each $E_a$-class in exactly one point. Let $\nu$ be the normalized restriction of $\mu$ to $Y$ and let $b\in A(\bbZ, Y, \nu)$ have infinite orbits. We can then take $F= E_a\vee E_b$ (the smallest equivalence relation containing $E_a, E_b$).

Finally we show that (vii) implies (x). Let $\Gamma \setminus\{e_\Gamma \}=\{ \gamma_1,\gamma_2, \dots\}$, and for each $n\geq 1$, let $a_n \in A(\Gamma, X_n, \mu_n)$ be an ergodic, hyperfinite action, where the space $(X_n, \mu_n)$ is either finite or non-atomic, with $\mu_n ({\rm Fix}_{a_n} (\gamma_i) )< \frac{1}{n}$ for all $i\leq n$. Let $\theta_n$ be the corresponding IRS. If $X_n$ is finite, clearly $\theta_n$ is co-amenable. If $(X_n, \mu_n)$ is non-atomic, then, by \cite{Ka} (using the Schreier graph associated to the action and a fixed finite subset of $\Gamma$), $\theta_n$ is co-amenable. Finally (by going to subsequences) it is enough to check that if $(\theta_n)$ converges, say to $\theta$, then $\theta$ is the Dirac measure at the identity of $\Gamma$. We have that for any finite $F\subseteq \Gamma$, $\theta_n(\{H\leq \Gamma\colon F\subseteq H \}) \to \theta(\{H\leq \Gamma\colon F\subseteq H \}) =\delta_{e_\Gamma}(\{H\leq \Gamma\colon F\subseteq H \})$, so (as in Appendix C, \cref{appB}) by the $\pi-\lambda$ Theorem $\theta=\delta_{e_\Gamma}$.
\end{proof}

\newpage
\section{Appendix E}\label{appE} We give here the proof of \cref{1022}.

\begin{proof} We start with the following:
\begin{lem}
If $a,b\in A(\Gamma, X, \mu)$ are tempered, so is $a\times b$.
\end{lem}
\begin{proof}
If $\Gamma$ is amenable, this is clear, so we assume below that $\Gamma$ is not amenable.

We have $\kappa^{a\times b}\cong \kappa^a \otimes \kappa^b$. So if $1_\Gamma$ is the trivial 1-dimensional representation of $\Gamma$\index{$1_\Gamma$}, so that $\kappa^a \cong 1_\Gamma \oplus \kappa_0^a$ (and similarly for $b, a\times b$), we have
\[
1_\Gamma\oplus\kappa_0^{a\times b}\cong (1\oplus \kappa_0^a)\otimes(1_\Gamma\oplus\kappa_0^b)\cong 1_\Gamma\oplus\kappa_0^a\oplus\kappa_0^b\oplus (\kappa_0^a\otimes\kappa_0^b).
\]
Now since $a,b$ are tempered, $\kappa_0^a, \kappa_0^b$ have no invariant non-0 vectors. It follows then (by looking at the subspaces of invariant vectors in $1_\Gamma\oplus \kappa_0^{a\times b}$ and $1_\Gamma\oplus\kappa_0^a\oplus\kappa_0^b\oplus (\kappa_0^a\otimes\kappa_0^b)$) that 
\[
\kappa_0^{a\times b}\cong \kappa_0^a\oplus\kappa_0^b\oplus (\kappa_0^a\otimes\kappa_0^b)\preceq \lambda_\Gamma\oplus(\kappa_0^a\otimes\kappa_0^b)\preceq\lambda_\Gamma\oplus (\lambda_\Gamma\otimes\lambda_\Gamma).\]
If $\Gamma$ acts on a countable set $X$, denote by $\pi_{X, \Gamma}$ the representation on $\ell^2 (X)$ given by $\gamma\cdot f (x) = f(\gamma^{-1}\cdot x)$. Then if $\Gamma$ acts freely on $X$, $\pi_{X, \Gamma}\cong n\cdot \lambda_\Gamma$, where $n$ is the cardinality of the set of $\Gamma$-orbits on $X$. Now it is easy to check that $\lambda_\Gamma\otimes \lambda_\Gamma\cong\pi_{X, \Gamma}$, where $X= \Gamma\times \Gamma$ and $\Gamma$ acts on $X$ by $\gamma\cdot (\delta, \epsilon) = (\gamma\delta, \gamma\epsilon)$. This is a free action, so $\lambda_\Gamma\otimes \lambda_\Gamma$ is a direct sum of countably many copies of $\lambda_\Gamma$. Thus $\kappa_0^{a\times b}\preceq\lambda_\Gamma$, i.e., $a\times b$ is tempered.
\end{proof}

\begin{lem}
If $a_0, a_1, \dots \in A(\Gamma, X, \mu)$ are tempered, so is $\prod_n a_n$.

\end{lem}
\begin{proof}
Let $b =\prod_n a_n$.  Consider $b_N = \prod_{n=0}^{N-1} a_n\in A(\Gamma, X^N, \mu^N)$, for $N\geq 1$. For $f\in L^2(X^N, \mu^N)$, let $\tilde{f}^N\in L(X^\bbN, \mu^\bbN)$ be given by $\tilde{f}^N (((x_n)_{n\in\bbN}) = f((x_n)_{n< N})$. Then the map $f\mapsto \tilde{f}^N$ is an isomorphism of $\kappa_0^{b_N}$ with a subrepresentation $\pi_N\leq \kappa_0^b$, say $\pi_N = \kappa_0^b|H_N$, for a $\Gamma$-invariant closed subspace $H_N$ of $L^2(X^\bbN, \mu^\bbN)$. Clearly $H_1\subseteq H_2\subseteq \dots$ and $\bigcup_{N=1}^\infty H_N$ is dense in $L^2(X^\bbN, \mu^\bbN)$. Since $\pi_N\preceq\lambda_\Gamma$ for each $N$, clearly $\kappa_0^b\preceq \lambda_\Gamma$, i.e., $b$ is tempered.
\end{proof}

Let now $\{\undertilde{a}_n\}$ be dense in $\undertilde{{\rm TEMP}}(\Gamma, X, \mu)$. Then, if $b= \prod_n a_n$, $\undertilde{b}\in \undertilde{{\rm TEMP}}(\Gamma, X, \mu)$ is the $\preceq$-maximum element of $\undertilde{{\rm TEMP}}(\Gamma, X, \mu)$ and since $\undertilde{{\rm TEMP}}(\Gamma, X, \mu)$ is downwards closed under $\preceq$, it follows that 
\[
\undertilde{{\rm TEMP}}(\Gamma, X, \mu) = \{\undertilde{a}\colon \undertilde{a}\preceq \undertilde{b}\}
\]
is closed. Also $\undertilde{\prod_n a_n} = \undertilde{a}^{temp}_{\infty, \Gamma}$. 
\end{proof}

\newpage
\section{Appendix F}\label{appF1}
We give here the proof of \cref{44}.
Recall from \cite[Definition 3.1]{LeM} that a transitive action of a group $\Gamma$ on a set $X$ is {\bf highly faithful} \index{highly faithful}if for any $\gamma_1, \dots , \gamma_n\in \Gamma\setminus\{e_\Gamma\}$, there is $x\in X$ such that $\gamma_i\cdot x\not= x, \forall i\leq n$.
Call a subgroup $H\leq \Gamma$ highly faithful if the action of $\Gamma$ on $\Gamma/H$ is highly faithful and call $a\in A(\Gamma, X, \mu)$ highly faithful if the action of $\Gamma$ on (almost) any orbit is highly faithful or equivalently (almost) all stabilizers are highly faithful.



\begin{lemma}
If $H\leq \Gamma$ is highly faithful, then $\lambda_\Gamma\preceq \lambda_{H, \Gamma}$.
\end{lemma}
\begin{proof}
It is enough to find for each $\gamma_1, \dots , \gamma_n\in \Gamma\setminus\{e_\Gamma\}, 1>\epsilon >0$ a unit vector $f \in \ell^2(\Gamma/H)$ such that, for the quasi-regular representation, we have $|\langle \gamma_i\cdot f, f\rangle|\leq \epsilon$, for each $1\leq i\leq n$ (see, e.g., \cite[Proposition 18.1.4]{D}). Put $X=\Gamma/H$ and also denote by $\gamma\cdot x$ the action of $\Gamma$ on $X$. Let $x$ be such that $\gamma_i\cdot x\not= x, \forall i\leq n,$ and put $A= \{x\}\cup \{\gamma_i^{-1}\cdot x \colon 1\leq i\leq n\}$ and $|A| = k+1,  k>0$. Then define $f\colon X\to \bbC$ by $f(y) = \sqrt{(1-\frac{\epsilon^2}{4k})}$, if $y=x$; $=\frac{\epsilon}{2k}$, if $y\in A\setminus \{x\}$; = 0, otherwise. Then $f$ is a positive unit vector and for any $1\leq i\leq n$, we have $\langle \gamma_i\cdot f, f\rangle \leq f(x)f(\gamma_i^{-1}\cdot x) +k(\frac{\epsilon}{2k}\cdot 1)\leq  1\cdot \frac{\epsilon}{2k} +\frac{\epsilon}{2}\leq\epsilon$.
\end{proof}
It is now shown in \cite[Corollary 1.13]{LeM} that for $\Gamma=\bbF_2$, there is a highly faithful, ergodic $a\in A(\Gamma, X, \mu)$ such that the action of $\Gamma$ on $\Gamma/H$, for (almost) all stabilizers $H$, is amenable, i.e., $1_\Gamma\preceq \lambda_{H,\Gamma}$ (see, e.g., \cite[Theorem 1.10]{KT}).

By \cite[Proposition 14]{DG}, $\lambda_{H,\Gamma}\preceq \kappa_0^a$, for (almost) all stabilizers $H$. Thus $\lambda_\Gamma\preceq \kappa_0^a$ and $1_\Gamma\preceq \kappa_0^a$. Moreover $\lambda_\Gamma\prec \kappa_0^a$, since otherwise $1_\Gamma\preceq \lambda_\Gamma$ contradicting the non-amenability of $\Gamma$.

Next we have that $s_\Gamma\not\preceq a$, since $a$ is not free. But also $a\not\preceq s_\Gamma$, since otherwise $a$ would be free, because the group $\Gamma$ is shift-minimal, see \cite[Theorem 5.1]{T-D}

\newpage

\section{Appendix G}\label{appF}
We give here the proof of \cref{irs1}.
\begin{proof}

We will use the following result:

\begin{prop}\label{17.1}
Let $H\leq \Gamma$. Then the following are equivalent:
\begin{enumerate}[\upshape (i)]
\item $s_{H, \Gamma}\preceq s_\Gamma$.
\item $H$ is amenable.
\end{enumerate}
\end{prop}
\begin{proof}
Assume first that $H$ is amenable. Then $i_H\preceq s_H$, so, by \cref{321}, ${\rm CIND}_H^\Gamma (i_H) \preceq {\rm CIND}_H^\Gamma (s_H)$. Now ${\rm CIND}_H^\Gamma (i_H)\cong s_{H, \Gamma}$ by \cite[Page 73]{K} and ${\rm CIND}_H^\Gamma (s_H) \cong s_\Gamma$ by \cite[Proposition A.2]{K1}. So (ii) implies (i).

Conversely assume that $s_{H, \Gamma}\preceq s_\Gamma$. View here $s_{H, \Gamma}$ as the shift action of $\Gamma$ on $\bbT^{\Gamma/H}$, which is an action of $\Gamma$ by automorphisms on the abelian compact Polish group $\bbT^{\Gamma/H}$. Now $\kappa_0^{s_{H, \Gamma}} \preceq \kappa_0^{s_\Gamma} \simeq \lambda_\Gamma$ (see \cref{4.1}, \cite[Appendix D, {\bf (E)}]{K} and \cite[Exercise E.4.5]{BdlHV}). Thus $s_{H, \Gamma}$ is tempered. Consider the shift action of $\Gamma$ on the dual group of $\bbT^{\Gamma/H}$, i.e., $\bbZ^{<\Gamma/H}$ (the group of all elements of $\bbZ^{\Gamma/H}$ of finite support).  Then by \cite[Theorem 4.6 and Section 5 (A)]{K4} it follows that the stabilizer of any non-zero element of this countable group is amenable. Since the stabilizer of the characteristic function of the element $\{H\}$ of $\Gamma/H$ is equal to $H$, it follows that $H$ is amenable.
\end{proof}
We first prove (i) and (ii) of \cref{irs1}. If $\theta$ is amenable, then, by \cref{17.1}, $\undertilde{s}_{\theta, \Gamma}\preceq\undertilde{s}_\Gamma$. If $\theta$ is not core-free, then $\undertilde{s}_{\theta, \Gamma}$ is not free, so $\undertilde{s}_{\theta, \Gamma}\prec\undertilde{s}_\Gamma$, while if it is core-free, $\undertilde{s}_{\theta, \Gamma}$ is free, so (by  \cref{min}) $\undertilde{s}_{\theta, \Gamma} =\undertilde{s}_\Gamma$. Similarly (iii) and (iv) follow from \cref{17.1} and the freeness properties of $\undertilde{s}_{\theta, \Gamma}$.
\end{proof}

We note that \cref{17.1} admits the following generalization:

\begin{prop}
Let $\Delta\leq H \leq \Gamma$. Then the following are equivalent:
\begin{enumerate}[\upshape (i)]
\item $s_{H, \Gamma}\preceq s_{\Delta, \Gamma}$.
\item $\Delta$ is co-amenable in $H$.
\end{enumerate}
\end{prop}
\begin{proof}(i)$\implies$(ii): Note that the map $ p\in [0,1]^{\Gamma/H}\mapsto p(H)$ shows that $i_H\sqsubseteq s_{H,\Gamma}|H$. Since $s_{H, \Gamma}|H\preceq s_{\Delta, \Gamma}|H$, we have that $i_H\preceq s_{\Delta,\Gamma}|H$.

Let $T$ be a transversal for the right-cosets of $H$ in $\Gamma$. Then $\Gamma/\Delta = \bigsqcup_{t\in T} (H/\Delta)t$ and $(H/\Delta)t$ is $H$-invariant, so the action of $H$ on $\Gamma/\Delta$ is the direct sum of $N= |T|\in \{1.2, \dots, \aleph_0\}$ copies of the action of $H$ on $H/\Delta$. It follows that $s_{\Delta, \Gamma}|H \cong s_{\Delta, H}^N \cong s_{\Delta, H}$, therefore $i_H\preceq s_{\Delta, H}$. By \cite[Theorem 1.2]{KT} this implies that $\Delta$ is co-amenable in $H$.

(ii)$\implies$(i): First we note that if $[H:\Delta] <\infty$, then actually $s_{H, \Gamma}\sqsubseteq s_{\Delta, \Gamma}$. To see this view the action $s_{H, \Gamma}$ as the shift action of $\Gamma$ on the product space $\bbR^{\Gamma/H}$ with $
\bbR$ having the (normalized, centered) Gaussian measure $\gamma$. Similarly for $s_{\Delta, \Gamma}$. We now have the following lemma which immediately implies this fact:

\begin{lemma}
Let a group $\Gamma$ act on two countable sets $X,Y$. Let $n\in \bbN, n\geq 1$, and $\pi\colon X\to Y$ be $\Gamma$-equivariant and such that for each $y\in Y$, $|\pi^{-1} (y)| =n $. Denote by $s_{\Gamma, X}$ the shift action of $\Gamma$ on the product space $\bbR^X$, where $\bbR$ has the Gaussian measure $\gamma$. Similarly define $s_{\Gamma, Y}$. Then $s_{\Gamma, X} \sqsubseteq s_{\Gamma,Y}$.

\end{lemma}
\begin{proof} Define $\Phi\colon \bbR^X\to \bbR^Y$ by
\[
\Phi (p) (y) = \frac{1}{\sqrt{n}}\sum_{x\in \pi^{-1} (y)} p(x).
\]
Then $\Phi$ is $\Gamma$-equivariant, so its is enough to check that $\Phi_* \gamma^X = \gamma^Y$. This reduces to showing that if $\Theta_n\colon \bbR^n \to \bbR$ is defined by $\Theta_n (x_1, \dots , x_n) = \frac{1}{\sqrt{n}}\sum_{j =1}^n x_j$, then $\Theta_* \gamma^n = \gamma$. Now $\Theta_* \gamma^n$ is Gaussian with mean
\[
\mu = \frac{1}{\sqrt{n}} \sum_{j=1}^n E(x_j) = 0
\]
and variance 
\[
\sigma^2 = \sum_{j=1}^n (\frac{1}{\sqrt{n}})^2 \textrm{Var}(x_j) =1
\]
 (see \cite[Exercise 16.7]{JP}, so we are done.
 \end{proof}
 Thus we can assume that $[H:\Delta] =\infty$. Then $s_{\Delta, H}$ is ergodic (see \cite[Proposition 2.1]{KT}). Also, by \cite[Theorem 1.2]{KT}, $s_{\Delta, H}$ has (non-trivial) almost invariant sets, so by \cite[Proposition 10.6]{KT}, $i_H\preceq s_{\Delta, H}$. But $\textrm{CIND}_H^\Gamma (i_H) \cong s_{H, \Gamma}$ (see \cite[Page 73]{K}) and $\textrm{CIND}_H^\Gamma (s_{\Delta, H}) \cong s_{\Delta, \Gamma}$ (see \cite[Proposition A.2]{K1}), therefore $s_{H, \Gamma}\preceq s_{\Delta, \Gamma}$. \end{proof}

\printindex

\noindent Department of Mathematics

\noindent University of Texas at Austin

\noindent Austin, TX 78712-1202

\noindent\textsf{pjburton@math.utexas.edu}

\bigskip

\noindent Department of Mathematics

\noindent California Institute of Technology

\noindent Pasadena, CA 91125

\noindent\textsf{kechris@caltech.edu }

\end{document}